\documentclass[11pt]{amsart}

\usepackage[centertags]{amsmath}
\usepackage{amsfonts}
\usepackage{amssymb}
\usepackage{amsthm}
\usepackage[english]{babel}
\usepackage[active]{srcltx}
\usepackage{enumerate}
\usepackage{esint}
\usepackage{mathrsfs}
\usepackage[colorlinks, citecolor = blue, pagebackref]{hyperref}
\usepackage{stackengine}
\usepackage{xcolor}
\stackMath

\numberwithin{equation}{section}

\newtheorem{thm}{Theorem}[section]
\newtheorem{cor}[thm]{Corollary}
\newtheorem{lem}[thm]{Lemma}
\newtheorem{prop}[thm]{Proposition}

\def\N{{{\Bbb N}}}

\def\L{\text{Lip}}
\def\R{\mathbb{R}}
\def\T{\mathbb{T}}

\theoremstyle{remark}
\newtheorem{rem}[thm]{Remark}

\newcommand{\vertiii}[1]{{\left\vert\kern-0.25ex\left\vert\kern-0.25ex\left\vert #1
    \right\vert\kern-0.25ex\right\vert\kern-0.25ex\right\vert}}

\begin{document}

\title[
Sobolev embeddings, extrapolations, and related inequalities
]{Sobolev embeddings, extrapolations, and related inequalities}
\author
{\'Oscar Dom\'inguez}

\address{O. Dom\'inguez, Departamento de An\'alisis Matem\'atico y Matem\'atica Aplicada, Facultad de Matem\'aticas, Universidad Complutense de Madrid\\
Plaza de Ciencias 3, 28040 Madrid, Spain.}
\email{oscar.dominguez@ucm.es}

\author{Sergey Tikhonov}

 \address{S. Tikhonov, Centre de Recerca Matem\`{a}tica\\
Campus de Bellaterra, Edifici C
08193 Bellaterra (Barcelona), Spain;
ICREA, Pg. Llu\'{i}v Companys 23, 08010 Barcelona, Spain,
 and Universitat Aut\`{o}noma de Barcelona.}
\email{ stikhonov@crm.cat}

\maketitle

\bigskip
\begin{abstract}

In this paper we propose a unified approach, based on limiting interpolation, to investigate the embeddings for the Sobolev space $(\dot{W}^k_p(\mathcal{X}))_0, \, \mathcal{X} \in \{\R^d, \T^d, \Omega\}$, in the subcritical case ($k < d/p$), critical case ($k=d/p$) and supercritical case ($k > d/p$). 
We characterize the Sobolev embeddings in terms of pointwise
inequalities involving rearrangements and moduli of
smoothness/derivatives of functions and via {extrapolation} theorems
for corresponding smooth function spaces. Applications include
Ulyanov-Kolyada type inequalities for rearrangements, inequalities
for moduli of smoothness,
sharp Jawerth-Franke embeddings for Lorentz-Sobolev spaces, various
characterizations of Gagliardo-Nirenberg, Tru\-din\-ger,
Maz'ya-Hansson-Br\'ezis-Wainger and Br\'ezis-Wainger embeddings,  
among others. In particular, we show that the Tao's extrapolation theorem holds true in the setting of Sobolev inequalities. This gives a positive answer to a question recently posed by Astashkin and Milman.  

\end{abstract}
\tableofcontents

\newpage
\section{Introduction}

Sobolev inequalities constitute an important part of functional analysis and geometry with a wide range of remarkable applications in the theory of PDE's, calculus of variations and mathematical physics \cite{AdamsHedberg, EdmundsEvans, Mazya, SobolevBook}.
 The classical Sobolev theorem \cite{Sobolev} reads as follows:
\begin{equation}\label{121=}
	(\dot{W}^k_p(\R^d))_0 \hookrightarrow  L_{p^\ast}(\R^d), \quad p^\ast = \frac{d p}{d - k p}
\end{equation}
provided that  $k \in \mathbb{N}, 1 \leq p < \infty$ and $k < d/p$. 

The original proof of Sobolev embedding \eqref{121=}   (\cite{Sobolev}, see also \cite{SobolevBook} and \cite{Mazya}) relies on an integral representation formula which allows us to reconstruct functions from their derivatives. This rather complicated proof was later significantly simplified by Peetre \cite{Peetre} (with \cite{Hunt} and \cite{Oneil} as forerunners) applying classical \emph{interpolation} methods. In fact, Peetre's approach uses a more refined technique
  replacing the target space $L_{p^\ast}(\R^d)$ in \eqref{121=}  by the smaller Lorentz space $ L_{p^\ast,p}(\R^d)$ (note that $p < p^*$), that is,
\begin{equation}\label{121=new}
	(\dot{W}^k_p(\R^d))_0 \hookrightarrow  L_{p^\ast,p}(\R^d).
\end{equation}
Furthermore, this embedding is optimal within the class of r.i. spaces.

\subsection{Main results and observations}
The main objective of this paper is to provide, in a unifying fashion, two different  characterizations of the homogeneous/in\-ho\-mo\-ge\-neo\-us Sobolev embeddings
\begin{equation}\label{121}
	(\dot{W}^kX)_0 \hookrightarrow Y \quad \text{and} \quad W^kX \hookrightarrow Y
\end{equation}
where $X$ and $Y$ are suitable function spaces defined on $\R^d, \T^d$ or smooth domains $\Omega$ in $\R^d$. 
  Firstly, we show that  (\ref{121}) is equivalent to
certain \emph{pointwise inequalities} involving rearrangements and moduli of smoothness/derivatives of functions. Secondly,
 embedding \eqref{121} can be equivalently written via \emph{extrapolation} theorems for corresponding smooth spaces. Let us give more detail
to explain each approach.
We will consider three cases: 
\begin{itemize}
  \item[$\cdot$] subcritical case:
  \begin{equation*}
  (\dot{W}^k_p(\R^d))_0 \hookrightarrow L_{p^\ast,p}(\Omega)\quad  \mbox{for}\quad k < d/p,
  \end{equation*}
  \item[$\cdot$] critical case: $$(\dot{W}^{k}_p(\R^d))_0 \hookrightarrow Y \quad  \mbox{for}\quad k = d/p \quad\mbox{and appropriate} \;Y,$$
  \item[$\cdot$] and supercritical case:
$$(\dot{W}^{k}_p(\R^d))_0 \hookrightarrow Y \quad  \mbox{for}\quad k > d/p \quad\mbox{and appropriate} \;Y.$$
\end{itemize}


\textbf{Subcritical case.} In this setting ($k < d/p$) it is well known that the embedding
\begin{equation}\label{223}
(\dot{W}^k_p(\R^d))_0 \hookrightarrow L_{p^\ast,q}(\R^d)
\end{equation}
 holds if and only if  $q \geq p$ (cf. \eqref{121=new}).
We obtain (see Theorems \ref{ThmSobRdHom} and \ref{ThmSobRearRdHom} below for the precise statement) the following four new equivalent conditions  for this embedding
  to hold: 
\begin{enumerate}[\upshape(i)]
		\item if $f \in L_p(\R^d) + (\dot{W}^k_p(\R^d))_0$ and $t > 0$ then
		\begin{equation}\label{ThmSob1*+}
			\left(\int_0^{t^d} f^\ast(u)^p du \right)^{1/p} + t^k \left(\int_{t^d}^\infty u^{q/p^\ast} f^\ast(u)^q \frac{du}{u} \right)^{1/q} \lesssim \omega_k(f,t)_{p;\R^d},
		\end{equation}
\item if $f$ (and its derivatives up to order $k-1$) has a certain decay at infinity and $ |\nabla^k f |\in L_p(\R^d)$ then
		\begin{equation}\label{SharpKolyada2*'Intro}
	\left(\int_t^\infty \left(v^{1/p-1} \int_0^v u^{1 - k/d} f^\ast(u) \frac{du}{u} \right)^q \frac{dv}{v} \right)^{1/q} \lesssim \left(\int_t^\infty ( |\nabla^k f|^{\ast \ast}(v))^p dv \right)^{1/p},
\end{equation}

		\item if $s \to k-$ then there exists $C > 0$, which is independent of $s$, such that
		\begin{equation}\label{Intro3}
				\|f\|_{L_{\frac{d p}{d - s p}, q}(\R^d)} \leq C (k - s)^{1/q}  \|f\|_{\dot{B}^s_{p,q}(\R^d),k}
		\end{equation}
		for $f \in \dot{B}^s_{p,q}(\R^d),$

\item 
if $r \to p-$ then there exists $C > 0$, which is independent of $r$, such that
		\begin{equation}\label{SharpFrankeJawerth2*'Intro}
			\|f\|_{L_{r^\ast,q}(\R^d)} \leq C \| |\nabla^k f|\|_{L_{r,p}(\R^d)}
		\end{equation}
		whenever $f$ (and its derivatives up to order $k-1$) has a certain decay at infinity.  
	\end{enumerate}

To understand better the strength of the previous result, let us consider 
 the endpoint case $p=1$. This case deserves special attention because it corresponds to the sharp Gagliardo-Nirenberg inequality \cite{Poornima}
\begin{equation}\label{GN}
	\|f\|_{L_{d',1}(\R^d)} \leq C \| |\nabla f|\|_{L_1(\R^d)}, \qquad \frac{1}{d} + \frac{1}{d'} = 1,
\end{equation}
where $f \in C^\infty_0(\R^d)$. Further characterizations of \eqref{GN} were obtained in Mart\'in, Milman and Pustylnik \cite{MartinMilmanPustylnik}. In particular, they showed \cite[Theorem 1]{MartinMilmanPustylnik} that \eqref{GN} is equivalent to the symmetrization inequality
\begin{equation*}
	\int_0^t u^{-1/d} (f^{**}(u) - f^*(u)) du \lesssim \int_0^t |\nabla f|^*(u) du, \quad f \in C^\infty_0(\R^d).
\end{equation*}

The previous characterizations of the sharp Gagliardo-Nirenberg inequality are now complemented by \eqref{ThmSob1*+} and \eqref{Intro3}. Specifically, we obtain that \eqref{GN} holds if and only if 
\begin{equation*}
	\int_0^{t^d} f^*(u) du + t \int_{t^d}^\infty u^{1/d'} f^*(u) \frac{du}{u} \lesssim \omega_1(f,t)_{1;\R^d}, \quad f \in L_1(\R^d)
\end{equation*}
(see \eqref{ThmSob1*+}). Note that the latter inequality strengthens (see Remark \ref{Rembeta}(i) below) the well-known Ulyanov-Kolyada inequality \cite{Kolyada}, which asserts that if $f \in L_p(\R^d), \, 1 \leq p < \infty$, then
	\begin{equation*}
		t \left(\int_{t^d}^\infty u^{-p/d} \int_0^{u} (f^\ast(v) - f^\ast(u))^p dv \frac{du}{u} \right)^{1/p} \lesssim \omega_1(f,t)_{p;\R^d}.
	\end{equation*}
This inequality plays a key role in the study of embedding theorems of  smooth function spaces (see, e.g., \cite{Kolyada}, \cite{KolyadaNafsa}, \cite{CaetanoGogatishviliOpic} and \cite{CaetanoGogatishviliOpic11}).

We also show that the Gagliardo-Nirenberg inequality \eqref{GN}, or more generally, Sobolev inequalities \eqref{121=new}, can be characterized by extrapolation means. Indeed,  assertion \eqref{Intro3} provides sharp estimates for the blow-up rates of the norm of the classical embedding
$\dot{B}^s_{p,q}(\R^d) \hookrightarrow L_{\frac{d p}{d-sp},q}(\R^d)$, $s<d/p$, as the smoothness para\-meter $s$ approaches the certain critical value. This extrapolation assertion fits into the research program initiated by Bourgain, Br\'ezis, Mironescu \cite{BourgainBrezisMironescu} and Maz'ya, Shaposhnikova \cite{MazyaShaposhnikova} for the case $p=q$ and extended to the general case $p \neq q$ by Karadzhov, Milman and Xiao \cite{KaradzhovMilmanXiao} (see also \cite{EdmundsEvansKaradzhov, EdmundsEvansKaradzhov2, KolyadaLerner}).

Let us examine the new connection between \eqref{223} and \eqref{Intro3}  more carefully. In light of the Jawerth-Milman extrapolation (see \cite{JawerthMilman, Milman}), it follows that
  \eqref{Intro3} implies $(\dot{W}^k_p(\R^d))_0 \hookrightarrow L_{p^\ast,q}(\R^d)$. However, it is remarkable that the converse statement also holds true, that is, we are able to obtain converse extrapolation theorems for Sobolev embeddings in the spirit of Tao's paper \cite{Tao}.
   In other words, 
     from an endpoint embedding (i.e., $(\dot{W}^k_p(\R^d))_0 \hookrightarrow L_{p^\ast,q}(\R^d)$) one can  derive all intermediate embeddings with sharp blows up of the norms (i.e, $\dot{B}^s_{p,q}(\R^d) \hookrightarrow L_{\frac{d p}{d-sp},q}(\R^d)$ with the norm bound $\mathcal{O} ((k-s)^{1/q})$ as $s \to k-$). A similar comment also applies to \eqref{SharpFrankeJawerth2*'Intro}. In particular, these results provide an answer in the context of smooth function spaces  to an open problem mentioned in the recent paper by Astashkin and Milman \cite[Problem 27]{AstashkinMilman}.

The problem of optimality of Sobolev embeddings can be considered from the different perspective. More specifically, so far we have dealt with the optimality of the target space in \eqref{121=new}. Conversely,
 fixing the target space,
  the counterpart of the embedding $(V^k_p(\R^d))_0 \hookrightarrow L_{p^\ast,q}(\R^d)$
is the embedding
$(V^k L _{p,q}(\R^d))_0 \hookrightarrow L_{p^\ast}(\R^d)$ which holds if and only if $
		q \leq {p^\ast}.
	$
We again give four  different equivalent conditions for the latter  embedding
(see Theorems \ref{ThmStein2RdHom} and
\ref{ThmStein3RdHom} below for precise statements):
	\begin{enumerate}[\upshape(i)]
		\item for $ f \in (V^k L_{p,q}(\R^d))_0 +  (V^k_{p^*}(\R^d))_0$ and $t > 0$, we have
		\begin{equation}\label{SharpKolyada2+}
	\omega_k(f,t)_{p^\ast;\R^d} \lesssim \left(\int_0^{t^d} (u^{1/p} |\nabla^k f|^\ast(u))^q  \frac{du}{u} \right)^{1/q} + t^k \left( \int_{t^d}^\infty (|\nabla^k f|^\ast(u))^{p^\ast} du\right)^{1/{p^\ast}},
\end{equation}

\item for $ f \in (V^k_1(\R^d))_0 + (V^k L_{p,q}(\R^d))_0$ and $t > 0$, we have
		\begin{equation}\label{SharpKolyada2*Intro}
	\left(\int_t^\infty \left(v^{1/p - 1} \int_0^v u^{1 - k/d} f^\ast(u) \frac{du}{u} \right)^{p^\ast} \frac{dv}{v} \right)^{1/{p^\ast}} \lesssim \left(\int_t^\infty (v^{ 1/p} |\nabla^k f|^{\ast \ast}(v))^q \frac{dv}{v} \right)^{1/q},
\end{equation}
			\item
if $r \to p+$ and $m \in \N$ then there exists $C > 0$, which is independent of $r$, such that
		\begin{equation}\label{SharpFrankeJawerth2+}
			\|f\|_{\dot{B}^{d(1/p - 1/r) }_{p^\ast,q}(\R^d),m} \leq C (r -p)^{-1/q} \||\nabla^k f|\|_{L_{r,q}(\Omega)}
		\end{equation}
		for $f \in (V^k L_{r,q}(\R^d))_0$,
\item 
if $r \to p-$ then there exists $C > 0$, which is independent of $r$, such that
		\begin{equation}\label{Intro4}
			\|f\|_{L_{r^\ast,p^\ast}(\R^d)} \leq C \||\nabla^k f|\|_{L_{r,q}(\R^d)}
		\end{equation}
		for $f \in (V^k L_{r,q}(\R^d))_0$.
	\end{enumerate}

Note that the inhomogeneous counterpart for domains of (\ref{SharpKolyada2+}) with $q=p$ has been  recently obtained by Gogatishvili, Neves and Opic \cite[Theorem 3.2]{GogatishviliNevesOpic} in order to establish embeddings between Sobolev-type spaces and H\"older spaces; while
assertions (ii)--(iv) are new.
  The extrapolations of $(V^k L _{p,q}(\R^d))_0 \hookrightarrow L_{p^\ast}(\R^d)$ given in  (\ref{SharpFrankeJawerth2+}) and \eqref{Intro4} involve the Sobolev spaces $(V^k L_{r,q}(\R^d))_0$ with $r \to p+$ and $r \to p-$, respectively. In particular, \eqref{SharpFrankeJawerth2+} provides sharp blow up of the norm of the recently obtained  Jawerth-Franke embedding for Lorentz-Sobolev spaces \cite[Theorem 1.2]{SeegerTrebels}. It is worthwhile to mention that, unlike \eqref{Intro3}, the constant in \eqref{SharpFrankeJawerth2+} becomes arbitrary large as $r \to p+$. This phenomenon shows an interesting distinction between optimal range space and optimal domain space in Sobolev inequalities.

Our approach is flexible enough to be applied to the more general Sobolev embeddings $(V^k L_{p,q_0}(\R^d))_0 \hookrightarrow L_{p^*,q_1}(\R^d), \, q_0 \leq q_1$ (cf. \cite{Talenti, MilmanPustylnik}); see Theorem \ref{ThmSobRearFull} below.

\vspace{0.3cm}
\textbf{Critical case.} Let us switch temporarily to inhomogeneous function spaces on domains.  
If 
 $k={d}/{p} \in \mathbb{N}$ then the corresponding embedding (\ref{121=})  fails to be true, that is,
\begin{equation}\label{NonSobBoundedIntro}
	W^{d/p}_p(\Omega) \not \hookrightarrow L_\infty(\Omega), \qquad p > 1.
\end{equation}
To overcome this drawback we can use two different strategies.
On the one hand,  Sobolev embeddings can be obtained for the fixed domain space (i.e., $W^{d/p}_p(\Omega)$) by enlarging the target space (i.e., $L_\infty(\Omega)$). On the other hand,
 for the fixed target space (i.e., $L_\infty(\Omega)$) one can restrict 
   the domain space (i.e., $W^{d/p}_p(\Omega)$).


Firstly, we will concentrate on the case when $W^{d/p}_p(\Omega)$ in (\ref{NonSobBoundedIntro}) is fixed.
We start with the well-known embedding

\begin{equation}\label{AMTIntro}
	W^{d/p}_p(\Omega) \hookrightarrow L_\infty (\log L)_{-1/p'}(\Omega),
\end{equation}
which 
 is traditionally attributed to Trudinger \cite{Trudinger} with Peetre \cite{Peetre}, Pohozhaev \cite{Pohozhaev} and Yudovich \cite{Yudovich} as forerunners (cf. also \cite{Strichartz}). 
%

In this paper, following the program suggested  
  in a study of the subcritical case, we establish new characterizations of the Trudinger inequality via rearrangement inequalities in terms of moduli of smoothness and extrapolation means (see Theorem \ref{ThmTruMos} below).


Despite its importance, embedding \eqref{AMTIntro} is not optimal among all r.i. spaces. This is illustrated by the fact that the embedding
$W^{d/p}_p(\Omega) \hookrightarrow L_q(\Omega)$, $q < \infty$,
   can be improved involving Lorentz spaces.
 Indeed, we have
\begin{equation}\label{CriticalLorentzIntro}
	W^{d/p}_p(\Omega) \hookrightarrow L_{q,p}(\Omega) \quad \text{for all} \quad q < \infty.
\end{equation}
In view of (\ref{CriticalLorentzIntro}),
it was shown independently by Hansson \cite{Hansson} and Br\'ezis, Wainger \cite{BrezisWainger} that
\begin{equation}\label{BW+}
	W^{d/p}_p(\Omega) \hookrightarrow L_{\infty,p} (\log L)_{-1}(\Omega);
\end{equation}
cf. also \cite[p. 232]{Mazya} and \cite[(7.6.1), p. 209]{AdamsHedberg}. Related results may be found in the papers by Brudnyi \cite{Brudnyi} and Hedberg \cite{Hedberg}; a more general assertion was obtained by Cwikel and Pustylnik \cite{CwikelPustylnik}.

Furthermore, the target space in (\ref{BW+}) is the best possible among the class of r.i. spaces (see \cite{Hansson} and \cite{CwikelPustylnik}). In particular, we have
\begin{equation}\label{BWOptimal+}
	W^{d/p}_p(\Omega) \hookrightarrow L_{\infty,p} (\log L)_{-b}(\Omega) \iff b \geq 1.
\end{equation}

Our  goal is to establish new links between the Maz'ya-Hansson-Br\'ezis-Wainger embedding \eqref{BW+}, logarithmic inequalities for moduli of smoothness and extrapolation constructions. To be more precise, in this paper we obtain (cf. Theorem \ref{ThmBreWain} for the precise statement) that the embedding $W^{d/p}_p(\Omega) \hookrightarrow L_{\infty,p} (\log L)_{-b}(\Omega)$ is equivalent to either of the conditions:

	\begin{enumerate}[\upshape(i)]
		\item for $f \in L_p(\Omega)$ and $t \in (0,1)$, we have
			\begin{align}
			(1-\log t)^{1/p}\left(\int_0^{t^d} f^*(u)^p du \right)^{1/p} + t^{d/p} (1-\log t)^{b} \left(\int_{t^d}^1(1 - \log u)^{- b p} f^\ast(u)^p \frac{du}{u} \right)^{1/p} & \nonumber\\
			& \hspace{-11cm} \lesssim t^{d/p} (1-\log t)^{b} \|f\|_{L_p(\Omega)} + (1-\log t)^{b}  \omega_{d/p}(f, t)_{p;\Omega}, \label{IntroNew}
		\end{align}
		\item if $\lambda  \to 0+$ then there exists $C > 0$, which is independent of $\lambda$, such that
		\begin{equation}\label{SharpSobolevLog*+}
		\|f\|_{L_{d/\lambda,p}(\log L)_{-b}(\Omega)} \leq C \lambda^{1/p}
			\|f\|_{B^{d/p-\lambda}_{p,p}(\Omega), d/p}.
	\end{equation}
		
	\end{enumerate}

The new rearrangement inequality \eqref{IntroNew} provides the quantitative version of the
 Maz'ya-Hansson-Br\'ezis-Wainger embedding. Moreover,  inequality (\ref{SharpSobolevLog*+}) can be considered as the logarithmic counterpart of the Bourgain-Br\'ezis-Mironescu-Maz'ya-Shaposhnikova theorem; see the discussion in Remark \ref{RemEta}(i) below. Furthermore, our result answers  a question raised by Mart\'in and Milman in \cite{MartinMilman} concerning characterizations of the Maz'ya-Hansson-Br\'ezis-Wainger embedding in terms of limits of a family of norms (i.e., $\{\|f\|_{L_{d/\lambda,p}(\log L)_{-1}(\Omega)} : \lambda \to 0+\}$); see Remark \ref{RemEta}(ii). In fact, we go a step further and show that the converse is also true, that is, from  embedding \eqref{BW+} we achieve all intermediate Sobolev embeddings $B^{d/p-\lambda}_{p,p}(\Omega) \hookrightarrow L_{d/\lambda,p}(\log L)_{-1}(\Omega)$ with sharp behaviour $\mathcal{O} (\lambda^{1/p})$ of the norms as $\lambda \to 0+$.

Let us now focus on Sobolev embeddings into $L_\infty(\R^d)$. In the one-dimensional setting, the fundamental theorem of Calculus implies
$
	(V^1_1(\R))_0 \hookrightarrow L_\infty(\R).
$ 
In the higher dimensional case 
the corresponding estimate fails to be true, that is, one can find functions $f \in (V^1_d(\R^d))_0$ such that $f \not \in L_\infty(\R^d)$ (see (\ref{NonSobBoundedIntro}) with $p=d$). This obstruction can be overcome with the help of the refined scale given by Lorentz-Sobolev spaces. 
 Namely, Stein \cite{Stein81} showed that
\begin{equation*}
	(V^1 L_{d,1}(\R^d))_0 \hookrightarrow L_\infty(\R^d);
\end{equation*}
an alternative proof, as well as the optimality assertion, may be found in \cite{DeVoreSharpley}. Note that $L_{d,1}(\R^d) \subsetneq L_d(\R^d)$ if $d \geq 2$. These results can be extended to higher order derivatives (cf. \cite{MilmanAddendum, MilmanPustylnik}). To be more precise, if $k < d$ then
	\begin{equation}\label{SharpStein+}
		(V^k L_{d/k,q}(\R^d))_0 \hookrightarrow L_\infty(\R^d) \iff  q \leq 1.
	\end{equation}

Our purpose is to provide characterizations of 
embedding \eqref{SharpStein+} in terms of sharp inequalities of the $L_\infty$-moduli of smoothness, as well as extrapolation estimates of Jawerth-Franke type embeddings for Lorentz-Sobolev spaces. In more details (cf. Theorems \ref{ThmSteinRdHom} and \ref{ThmStein*RdHom}), the embedding $(V^k L_{d/k,q}(\R^d))_0 \hookrightarrow L_\infty(\R^d)$
is equivalent to any of the following four conditions:

\begin{enumerate}[\upshape(i)]
		
		\item for $ f \in(V^k L_{d/k,q}(\R^d))_0 + (V^k_\infty(\R^d))_0$ and $t > 0$, we have
		\begin{equation}\label{SharpKolyada+}
	\omega_k(f,t)_{\infty;\R^d} \lesssim \left(\int_0^{t^d} (u^{k/d} |\nabla^k f|^\ast(u))^q  \frac{du}{u} \right)^{1/q},
\end{equation}
\item  for $ f \in (V^k_1(\R^d))_0 + (V^k L_{d/k,q}(\R^d))_0$ and $t \in (0,\infty)$, we have
\begin{equation}\label{SharpKolyada**Intro}
	 t^{k/d-1} \int_0^t u^{1 - k/d} f^\ast(u) \frac{du}{u}  \lesssim \left(\int_{t}^\infty (u^{k/d} |\nabla^k f|^{\ast \ast}(u))^q \frac{du}{u} \right)^{1/q},
\end{equation}
			\item
 for any $m \in \N$, if $r \to (d/k)+$ then there exists $C > 0$, which is independent of $r$, such that
		\begin{equation}\label{SharpFrankeJawerth+}
			\|f\|_{\dot{B}^{k - d/r}_{\infty,q}(\R^d), m} \leq C \Big(r - \frac{d}{k}\Big)^{-1/q} \||\nabla^k f|\|_{L_{r,q}(\R^d)}
		\end{equation}
		for $f \in (V^k L_{r,q}(\R^d))_0$,


	\item
if $r \to (d/k)-$ then there exists $C > 0$, which is independent of $r$, such that
		\begin{equation}\label{SharpFrankeJawerth**Intro}
			\|f\|_{L_{r^\ast, q}(\R^d)} \leq C \Big(\frac{d}{k} - r\Big)^{-1/q} \||\nabla^k f|\|_{L_{r,q}(\R^d)}
		\end{equation}
		for $f \in (V^k L_{r,q}(\R^d))_0$.
	\end{enumerate}

To the best of our knowledge, this  is the first result to
 establish equivalence between the Stein inequality \eqref{SharpStein+} and pointwise rearrangement inequalities (cf. \eqref{SharpKolyada+} and \eqref{SharpKolyada**Intro}). Note that  inequality (\ref{SharpKolyada+}) with $q=1$, i.e.,
	\begin{equation*}
	\omega_k(f,t)_{\infty;\R^d} \lesssim \int_0^t u^k |\nabla^k f|^\ast(u^d)  \frac{du}{u}
\end{equation*}
is due to DeVore and Sharpley \cite[Lemma 2]{DeVoreSharpley} if $k=1$ and, Kolyada and P\'erez L\'azaro \cite[(1.6)]{KolyadaPerez} for higher-order derivatives. It plays a central role in the theory of function spaces as can be seen in \cite{Haroske}, \cite{GogatishviliMouraNevesOpic}, and the references within. On the other hand, the new inequality \eqref{SharpKolyada**Intro} gives a nontrivial improvement of the Kolyada inequality \cite[Corollary 3.2]{KolyadaNafsa}
\begin{equation*}
	f^{\ast \ast}(t)  \lesssim \int_{t}^\infty u^{k/d} |\nabla^k f|^{\ast \ast}(u) \frac{du}{u}, \quad f \in C^\infty_0(\R^d).
\end{equation*}
For further details, see Remark \ref{RemEps}(i) below.

{ Another } observation concerns the recent Jawerth-Franke embedding for Lorentz-Sobolev spaces obtained by Seeger and Trebels \cite{SeegerTrebels}, which asserts that if $1 < r < p < \infty, 0 < q \leq \infty$ and $k > d (\frac{1}{r} - \frac{1}{p})$, then
\begin{equation}\label{SeegerTrebels*+}
	(V^k L_{r,q}(\R^d))_0 \hookrightarrow \dot{B}^{k - d/r + d/p}_{p,q}(\R^d).
\end{equation}
 However, the interesting limiting
 case $p = \infty$ was left open in \cite{SeegerTrebels}. Remarkably, this  case is closely related to the Stein inequality \eqref{SharpStein+}. Indeed, as a byproduct of \eqref{SharpStein+}, we can consider the case $p=\infty$ in \eqref{SeegerTrebels*+} and, in addition, obtain sharp estimates of the rates of blow up of the corresponding embedding constant (cf. \eqref{SharpFrankeJawerth+}). Furthermore, the converse assertion is also valid, that is, the Stein inequality \eqref{SharpStein+} follows from Jawerth-Franke embeddings \eqref{SeegerTrebels*+} with $p=\infty$ via extrapolation. In a similar { vein}, we prove that the Stein inequality is equivalent to extrapolate the sharp version of  Talenti's embedding \cite{Talenti} (see also \cite{MilmanPustylnik})
 \begin{equation}\label{TalentiIntro}
  (V^k L_{r,q}(\R^d))_0 \hookrightarrow L_{r^*,q}(\R^d), \quad r < \frac{d}{k}
  \end{equation}
   (cf. \eqref{SharpFrankeJawerth**Intro}). It is worth mentioning  that the optimal constant in \eqref{TalentiIntro} with $k=1$ was obtained in \cite{Alvino} and \cite{Talenti}. To the best of our knowledge, the corresponding question for higher-order derivatives (i.e., $k > 1$) still remains open. Note that in  \eqref{SharpFrankeJawerth**Intro} we derive the optimal asymptotic behaviour of the constant in \eqref{TalentiIntro} with respect to the integrability parameter $r \to \frac{d}{k}-$.

\vspace{0.3cm}
\textbf{Supercritical case.}
It is well known that there are functions from the periodic (fractional) Sobolev space $\dot{H}^{1+d/p}_p(\T^d), \, 1 < p < \infty,$ that are not Lipschitz-continuous. However,
  due to the celebrated Br\'ezis-Wainger theorem \cite{BrezisWainger}, 
  functions $f$ from $\dot{H}^{1+d/p}_p(\T^d)$ are almost Lipschitz-continuous in the sense that
\begin{equation*}
	|f(x) - f(y)| \lesssim |x-y| |\log |x-y||^{1/p'} \|f\|_{\dot{H}^{1+d/p}_p(\T^d)}
\end{equation*}
for all $0 < |x-y| < 1/2$. Note that this inequality can be interpreted in terms of the embedding
\begin{equation*}
	\dot{H}^{1 + d/p}_p(\T^d) \hookrightarrow \text{Lip}^{(1, -1/p')}_{\infty, \infty}(\T^d).
\end{equation*}
This result has found profound  applications in function spaces and PDEs. { Just to mention}  some of them, it was the starting point of the theory of continuity envelopes of function spaces \cite{Triebel06, Haroske}.  It also plays a central role in studying the eigenvalue distribution of certain pseudo-differential operators \cite{EdmundsHaroske, EdmundsHaroske2}. For extensions of the Br\'ezis-Wainger inequality to the more general class of Triebel-Lizorkin spaces and Besov spaces, we refer the reader to Edmunds and Haroske \cite{EdmundsHaroske, EdmundsHaroske2}.

For convenience, we temporarily restrict our discussion to $d=1$. In this paper we study  characterizations of Br\'ezis-Wainger embeddings for the Sobolev spaces $\dot{H}^{\alpha + 1/p}_p(\T), \, \alpha > 0$. Namely, in Theorem \ref{UlyanovBrezisWainger}, we show that the following statements are equivalent:
	\begin{enumerate}[\upshape(i)]
	\item \begin{equation}\label{BWIntro}
		\dot{H}^{\alpha + 1/p}_p(\T) \hookrightarrow \text{Lip}^{(\alpha, -b)}_{\infty, \infty}(\T),
	\end{equation}
	\item for $f \in \dot{B}^{1/p}_{p,1}(\T)$ and $t \in (0,1)$, we have
		\begin{equation}\label{UlyanovTikSharpIntro}
		\omega_\alpha (f, t)_{\infty;\T} \lesssim  \int_0^{t (1 - \log t)^{b/\alpha}} u^{-1/p} \omega_{\alpha + 1/p} (f, u)_{p;\T} \frac{du}{u},
	\end{equation}
		\item
if $\alpha_0 \to \alpha-$ then there exists $C > 0$, which is independent of $\alpha_0$, such that
		\begin{equation}\label{UlyanovBrezisWaingerSharp**<<}
		\vertiii{f}_{\mathcal{C}^{\alpha_0}(\T), \alpha} \leq C (\alpha- \alpha_0)^{-b} \|f\|_{\dot{H}^{\alpha + 1/p}_p(\T)},
		\end{equation}
	\item \begin{equation*}
	b \geq 1/p'.
	\end{equation*}
	\end{enumerate}
This result shows that Br\'ezis-Wainger inequalities are closely connected to Ulyanov type inequalities \cite{Ulyanov}. The latter is central  in approximation theory, function spaces and interpolation theory (cf. \cite{KolomoitsevTikhonov} and the references therein). Taking  $b=1/p'$ in \eqref{UlyanovTikSharpIntro} we obtain a new Ulyanov-type inequality, which improves the known estimate \cite[(1.5)]{Tikhonov} 	
\begin{equation*}
		\omega_\alpha (f, t)_{\infty;\T} \lesssim \int_0^t u^{-1/p} (1 - \log u)^{1/p'} \omega_{\alpha + 1/p} (f, u)_{p;\T} \frac{du}{u}
	\end{equation*}
in several directions; see discussion in Remark \ref{RemUlyanovBrezisWainger}(i). The higher-dimensional version of \eqref{UlyanovTikSharpIntro} also holds true for functions $f$ on $\T^d$ or  $\R^d$; see Remark \ref{Remark6.3}.

The Br\'ezis-Wainger embedding \eqref{BWIntro} can be complemented by the well-known Jawerth-Franke embedding (cf. \cite{Jawerth, Franke}; see also \cite{Marschall} and \cite{Vybiral})
\begin{equation}\label{FJIntro}
		\dot{H}^{\alpha + d/p}_{p}(\T^d) \hookrightarrow \dot{B}^\alpha_{\infty,p}(\T^d).
		\end{equation}
Accordingly, we study characterizations of \eqref{FJIntro} via pointwise inequalities for moduli of smoothness and extrapolations. More precisely, Theorem \ref{ThmJawerth} asserts that the following statements are equivalent:
	\begin{enumerate}[\upshape(i)]
	\item \begin{equation*}
		\dot{H}^{\alpha + d/p}_{p}(\T^d) \hookrightarrow \dot{B}^\alpha_{\infty,q}(\T^d),
	\end{equation*}
		\item for $f \in \dot{B}^{d/p}_{p,1}(\T^d)$ and $t > 0$, we have
		 \begin{equation}\label{ThmJawerth1+}
		t^\alpha \left(\int_t^\infty (u^{-\alpha} \omega_{\alpha + d/p} (f,u)_{\infty;\T^d})^q \frac{du}{u} \right)^{1/q} \lesssim \int_0^{t} u^{-d/p} \omega_{\alpha + d/p} (f, u)_{p;\T^d} \frac{du}{u},
	\end{equation}
	\item
if $\lambda \to 0+$ then there exists $C > 0$, which is independent of $\lambda$, such that
		\begin{equation}\label{ThmJawerth1*Intro}
			\|f\|_{\dot{B}^{\alpha - \lambda}_{\infty, q}(\T^d), \alpha + d/p} \leq C \lambda^{1/q}  \|f\|_{\dot{B}^{\alpha + d/p - \lambda}_{p,q}(\T^d), \alpha + d/p},
		\end{equation}
	\item \begin{equation*}
	q \geq p.
	\end{equation*}
	\end{enumerate}
In the special case  $\alpha + d/p \in \N$,  inequality \eqref{ThmJawerth1+}
	is due to Kolyada \cite{Kolyada2} (see also \cite{Netrusov}), while \eqref{ThmJawerth1*Intro} was obtained in Kolyada and Lerner \cite{KolyadaLerner}.
Our result provides the non-trivial extensions of both  \eqref{ThmJawerth1+} and \eqref{ThmJawerth1*Intro} to the fractional setting $\alpha + d/p \not \in \N$.

Our technique   can also be applied to characterize embeddings involving the space 
 of functions with bounded variation. Specifically, Theorem \ref{UlyanovBrezisWainger2} establishes the connection between the embedding
$
		\text{BV}(\T^d) \hookrightarrow \text{Lip}^{(d/q,-1/q)}_{q, \infty}(\T^d), \, q > 1,
$
and the sharp Ulyanov inequality for the $L_1$-moduli of smoothness.

\subsection{Methodology}
Our method is partially inspired by Calder\'on's program \cite{Calderon}, which establishes the equivalence between the boundedness properties of (quasi-) linear operators acting on r.i. spaces and pointwise rearrangement inequalities via $K$-functionals (cf. also \cite{JawerthMilman, Milman}). In the setting of fractional Sobolev spaces, similar ideas appeared in Nilsson \cite{Nilsson}, Trebels \cite{Trebels} and Mart\'in, Milman \cite{MartinMilman14}. However, we need to introduce some modifications as
  Calder\'on's method does not necessarily yield optimal results in borderline cases. This obstruction already occurs in the critical case and supercritical case of Sobolev embeddings. To circumvent this  issue,
 we apply the new machinery of \emph{limiting interpolation} (cf. \cite{Astashkin}, \cite{DominguezTikhonov19} and \cite{DominguezTikhonov}).

In light of the Jawerth-Milman theory (cf. \cite{JawerthMilman, Milman}), one can recover endpoint estimates from intermediate estimates with sharp behaviour of the constants. Hence, we can show that extrapolation estimates imply Sobolev embeddings. On the other hand, the converse assertion requires deeper analysis. Indeed, it is well known that not all endpoint estimates can be obtained via extrapolation (for an elementary proof of this fact, we refer to \cite{Tao}). However, it was shown by Tao \cite{Tao} that, in the very special setting of translation invariant operators on compact symmetric spaces, it is still possible to establish the converse Yano extrapolation theorem (cf. also \cite{AstashkinMilman}). This strong result is crucial in our analysis to prove
  that limiting Sobolev inequalities imply extrapolation estimates.

\subsection{Structure of the paper}
In Section \ref{SectionNotation} we collect the main notations and definitions. 
 Definitions of function spaces are given in Section \ref{SectionFunctionSpaces} while the  interpolation methods are discussed in Section \ref{SectionInterpolation}.
Section \ref{SectionAuxResults} contains auxiliary results, namely, Hardy-type inequalities (cf. Section \ref{SectionHardy}), basic properties of moduli of smoothness (cf. Section \ref{SectionModuli}) and interpolation results (cf. Section \ref{SectionInterpolationLemma}).

 The main results and proofs are divided  into three sections:
the   subcritical case (Section \ref{subcritical}), critical (Section \ref{critical}), and   supercritical  (Section \ref{supercritical}).
We conclude with Appendix A,  which contains several useful equivalences for the $K$-functionals.

\vspace{2mm}

{\bf{Acknowledgements.}}
The first author was partially supported by MTM 2017-84058-P. The second
author was partially supported by
 MTM 2017-87409-P,  2017 SGR 358, and
 the CERCA Programme of the Generalitat de Catalunya. Part of this work was done during the visit of the authors to the Isaac Newton Institute for Mathematical Sciences, Cambridge, EPSCR Grant no EP/K032208/1.

\newpage

\section{Notation and definitions}\label{SectionNotation}

Given two (quasi-) Banach spaces $X$ and $Y$, we write $X \hookrightarrow Y$ if $X \subset Y$ and the natural embedding from $X$ into $Y$ is continuous.

As usual, $\R^d$ denotes the Euclidean $d$-space, $\T^d=[0, 2 \pi]^d$ is the $d$-dimensional torus, $\T = \T^1$, $\N$ is the collection of all natural numbers, $\N_0 = \N \cup \{0\}$ and $\N_0^d$ is the set of all points in $\R^d$ with components in $\N_0$.

For $1 \leq p \leq \infty$, $p'$ is defined by $\frac{1}{p} + \frac{1}{p'}=1$.

We will assume that $A\lesssim B$ means that $A\leq C B$ with a positive constant $C$ depending only on nonessential parameters.
If $A\lesssim B\lesssim A$, then $A\asymp B$.

Let $|\cdot|_d$ stand for the $d$-dimensional Lebesgue measure on $\R^d$.

\subsection{Function spaces}\label{SectionFunctionSpaces}

Throughout the paper, $\Omega$ is a bounded Lipschitz domain in $\R^d, \, d \geq 2$ (for the precise definition, see \cite[Definition 4.3, p. 195]{Triebel06}). Without loss of generality we will assume that $|\Omega|_d = 1$. 

Let $\mathcal{X} \in \{ \R^d, \Omega\}$. The \emph{decreasing rearrangement} $f^\ast : [0,\infty) \rightarrow [0,\infty)$ of a Lebesgue-measurable function $f$ in $\mathcal{X}$ is defined by
\begin{equation}\label{rearrangement}
	f^\ast(t) = \inf \{\lambda \geq 0 : |\{x \in \mathcal{X}: |f(x)| > \lambda\}|_d \leq t\}, \quad t \in [0,\infty),
\end{equation}
and the \emph{maximal function} $f^{\ast \ast}$ of $f^\ast$ is given by
\begin{equation}\label{maximal}
	f^{\ast \ast}(t) = \frac{1}{t} \int_0^t f^\ast(u) du.
\end{equation}
For $0 < p, q \leq \infty$, and $-\infty < b < \infty$, the \emph{Lorentz-Zygmund space} $L_{p,q}(\log L)_b(\mathcal{X})$ is formed by all Lebesgue-measurable functions $f$ on $\mathcal{X}$ having a finite quasi-norm
\begin{equation}\label{DefLZ}
	\|f\|_{L_{p,q}(\log L)_b(\mathcal{X})} = \left(\int_0^{|\mathcal{X}|_d} (t^{1/p} (1+|\log t|)^b f^\ast(t))^q \frac{dt}{t} \right)^{1/q}
\end{equation}
(with the usual modification if $q=\infty$). Note that $L_{p,q}(\log L)_b(\mathcal{X})$ becomes trivial when $p=\infty, 0 < q < \infty$ and $b \geq -1/q$, or $p=q=\infty$, but $b > 0$. If $p=q$ in $L_{p,q}(\log L)_b(\mathcal{X})$ then we obtain the \emph{Zygmund space} $L_p(\log L)_b(\mathcal{X})$. Setting $b=0$ in $L_{p,q}(\log L)_b(\mathcal{X})$ we recover the \emph{Lorentz spaces} $L_{p,q}(\mathcal{X})$ and if, in addition, $p=q$ then we obtain the \emph{Lebesgue spaces} $L_p(\mathcal{X})$. For more details, standard references are \cite{BennettSharpley}, \cite{BennettRudnick}, and \cite{EdmundsEvans}.

Assume $k \in \mathbb{N}, 1 \leq p \leq \infty$ and $0 < q \leq \infty$. The \emph{Lorentz-Sobolev space} $W^k L_{p,q}(\mathcal{X})$ is defined as the set of all $k$-times weakly differentiable functions $f$ in $\mathcal{X}$ with $|\nabla^m f| \in L_{p,q}(\mathcal{X})$ for $m=0, \ldots, k$. Here, $\nabla^0 f = f$ and $\nabla^m f,  \, m \in \N,$ denotes the vector of all $m$-th order weak derivatives $D^\alpha f, |\alpha| =m,$ of $f$ and $|\nabla^m f| = \sum_{|\alpha| = m} |D^\alpha f|$. The space $W^k L_{p,q}(\mathcal{X})$ is equipped with the norm
\begin{equation*}
	\|f\|_{W^k L_{p,q}(\mathcal{X})} = \sum_{m=0}^k \||\nabla^m f| \|_{L_{p,q}(\mathcal{X})}.
\end{equation*}
Obviously, setting $p=q$ in $W^k L_{p,q}(\mathcal{X})$ we obtain the classical Sobolev space $W^k_p(\mathcal{X}) := W^k L_p(\mathcal{X})$.

By $C^\infty_0(\mathcal{X})$ (respectively, $C^\infty_0(\R^d)$) we denote the space of all infinitely times differentiable functions on $\Omega$ (respectively, $\R^d$) vanishing at $\partial \Omega$ (respectively, with compact support). The \emph{(homogeneous) Lorentz-Sobolev space} $(\dot{W}^k L_{p,q}(\mathcal{X}))_0$ is defined as the closure of $C^\infty_0(\mathcal{X})$ in the seminorm
\begin{equation*}
	\|f\|_{(\dot{W}^k L_{p,q}(\mathcal{X}))_0} := \||\nabla^k f|\|_{L_{p,q}(\mathcal{X})}.
\end{equation*}
We will also consider the class of Sobolev spaces $(V^k L_{p,q}(\R^d))_0$ formed by all $k$-times weakly differentiable functions $f : \R^d \to \R$ whose derivatives up to order $k-1$ vanish at infinity, i.e.,
\begin{equation*}
	|\{x \in \R^d : |\nabla^m f (x)| > \lambda\}|_d < \infty \quad \text{for} \quad m \in \{0, \ldots, k-1\} \quad \text{and} \quad \lambda > 0.
\end{equation*}
For $f \in (V^k L_{p,q}(\R^d))_0$, we set
\begin{equation*}
	\|f\|_{(V^k L_{p,q}(\R^d))_0} :=  \||\nabla^k f|\|_{L_{p,q}(\R^d)}.
\end{equation*}
Note that if $p=q$ in $(\dot{W}^k L_{p,q}(\mathcal{X}))_0$ (respectively, $(V^k L_{p,q}(\R^d))_0$) then one recovers the classical space $(\dot{W}^k_p(\mathcal{X}))_0$ (respectively, $(V^k_p(\R^d))_0$).

In order to introduce Sobolev spaces of fractional order, we first recall the concept of directional derivatives. The directional derivative of $f$ of order $s > 0$ along a vector $\zeta \in \R^d$ is given by
\begin{equation*}
	D^s_\zeta f (x) = ((i \xi \cdot \zeta)^s \widehat{f}(\xi))^\vee(x), \quad x \in \R^d.
\end{equation*}
Here, $\widehat{f}$ is the Fourier transform of $f \in \mathcal{S}'(\R^d)$ given by
\begin{equation}\label{FouTrans}
	\widehat{f}(\xi) = \int_{\R^d} f(x) e^{-i x \cdot \xi} \, dx, \quad \xi \in \R^d.
\end{equation}
As usual, $f^\vee$ stands for the inverse Fourier transform, given by the right-hand side of \eqref{FouTrans} with $i$ in place of $-i$.

For $1 \leq p \leq \infty$ and $s > 0$, the \emph{(fractional) Sobolev space} $\dot{H}^s_p(\R^d)$ is formed by all $f$ such that
\begin{equation*}
	\|f\|_{\dot{H}^s_p(\R^d)} = \sup_{|\zeta|=1, \zeta \in \R^d} \| D^s_\zeta f\|_{L_p(\R^d)} < \infty.
\end{equation*}
We set $\dot{H}^0_p(\R^d) = L_p(\R^d)$. Note that if $1 < p < \infty$ then $\dot{H}^s_p(\R^d)$ coincides with the Riesz potential space and
\begin{equation*}
 \|f\|_{\dot{H}^s_p(\R^d)} \asymp \|(|\xi|^s \widehat{f}(\xi))^\vee\|_{L_p(\R^d)},
\end{equation*}
cf. \cite{Wilmes, Wilmes2}, and, in particular, setting $s=k \in \N$ one recovers the (homogeneous) Sobolev space $\dot{W}^k_p(\R^d)$ endowed with the semi-norm
\begin{equation*}
	\|f\|_{\dot{W}^k_p(\R^d)} := \| |\nabla^k f|\|_{L_p(\R^d)}.
\end{equation*}
The periodic space $\dot{H}^s_p(\T^d)$ can be introduced similarly.

For $h \in \R^d$, we let $\Omega_h = \{x \in \Omega : x + t h \in \Omega, \, 0 \leq t \leq 1\}$. As usual, we denote by $\Delta_h f = \Delta^1_h f$ the first difference of $f$ with step $h$, that is, $\Delta_h f (x) = f(x + h)-f(x), \, x \in \Omega_h$. Given $k \in \N$, the higher order differences $\Delta^{k+1}_h$ are defined inductively by $\Delta^{k+1}_h f (x)= \Delta_h (\Delta^k_h f)(x)$ for all $x \in \Omega_{(k+1) h}$. It is plain to check that
\begin{equation*}
	\Delta^k_h f (x) = \sum_{j=0}^k (-1)^{j}\binom{k}{j}  f(x + (k- j) h), \quad x \in \Omega_{k h}.
\end{equation*}

Let $k \in \N$ and $1 \leq p \leq \infty$. The \emph{$k$-th order modulus of smoothness} of $f$ in $L_p(\Omega)$ is defined by
\begin{equation}\label{DefModuli}
	\omega_k(f,t)_{p;\Omega} = \sup_{|h| \leq t} \|\Delta^k_h f\|_{L_p(\Omega_{k h})}, \quad t > 0.
\end{equation}

We also recall the definition of the modulus of smoothness of fractional order for functions in $L_p(\R^d)$. For $s > 0$, we let
\begin{equation*}
	\Delta^s_h f (x) = \sum_{j=0}^\infty(-1)^{j}  \binom{s}{j} f(x + (s- j) h), \quad x \in \R^d,
\end{equation*}
where $\binom{s}{j} = \frac{s (s-1) \ldots (s-j+1)}{j!}, \quad \binom{s}{0} =1$
and  \begin{equation*}
	\omega_s(f,t)_{p; \R^d} = \sup_{|h| \leq t} \|\Delta^s_h f\|_{L_p(\R^d)}, \quad t > 0.
\end{equation*}
Clearly, if $s \in \N$ then we recover the classical modulus of smoothness (\ref{DefModuli}) for functions $f \in L_p(\R^d)$. Analogously, one can define $\omega_s(f,t)_{p;\T^d}$ for $f \in L_p(\T^d)$.


Let $0 < s < k, k \in \N, 1 \leq p \leq \infty$ and $0 < q \leq \infty$. The \emph{(homogeneous) Besov space} $\dot{B}^s_{p,q}(\R^d)$ is defined as the completion of $C^\infty_0(\R^d)$ with respect to the (quasi)-seminorm
\begin{equation}\label{DefBesov}
	\|f\|_{\dot{B}^s_{p,q}(\R^d),k} := \left( \int_0^\infty (t^{-s} \omega_k (f,t)_{p;\R^d})^q \frac{dt}{t}\right)^{1/q}
\end{equation}
(suitably interpreted when $q=\infty$). The corresponding inhomogeneous Besov space $B^s_{p,q}(\R^d)$ is formed by all functions $f \in L_p(\R^d)$ such that
\begin{equation}\label{DefBesovInh}
	\|f\|_{B^s_{p,q}(\R^d),k} := c_{s,k,q} \|f\|_{L_p(\R^d)}  + \vertiii{f}_{B^s_{p,q}(\R^d),k} < \infty
\end{equation}
where
\begin{equation}\label{DefHolZyg}
	 \vertiii{f}_{B^s_{p,q}(\R^d),k} := \left( \int_0^1 (t^{-s} \omega_k (f,t)_{p;\R^d})^q \frac{dt}{t}\right)^{1/q}
\end{equation}
(suitably interpreted when $q=\infty$) and $c_{s,k,q} := \big(\frac{k}{(k-s) s q}\big)^{1/q}$ if $q < \infty$ ($c_{s,k,\infty} = 1)$. Similarly, one can introduce the spaces $\dot{B}^s_{p,q}(\Omega)$ and $B^s_{p,q}(\Omega)$, as well as the functionals $\vertiii{f}_{B^s_{p,q}(\Omega),k}$. Concerning periodic spaces, the space $\dot{B}^s_{p,q}(\T^d)$ (respectively, $B^s_{p,q}(\T^d)$) is formed by all those $f \in L_p(\T^d)$ such that the corresponding functional \eqref{DefBesov} (respectively, \eqref{DefBesovInh}) is finite. The normalization constant $c_{s,k,q}$ in \eqref{DefBesovInh} will play an essential role in extrapolation theorems (cf. \cite{Astashkin}, \cite{JawerthMilman}, \cite{KaradzhovMilmanXiao} an \cite{Milman05}). It is well known that $\|f\|_{\dot{B}^s_{p,q}(\R^d),k_1} \asymp \|f\|_{\dot{B}^s_{p,q}(\R^d),k_2}$ and $\|f\|_{B^s_{p,q}(\R^d),k_1} \asymp \|f\|_{B^s_{p,q}(\R^d),k_2}$ for different values of $k_1, k_2>s$ (including fractional) but the hidden equivalence constants depend on $k_1, k_2$.



{
Taking  $p=q=\infty$ in the definition of $B^s_{p,q}$, one recovers the \emph{H\"older-Zygmund spaces} $\mathcal{C}^s(\R^d)$ and $\mathcal{C}^s(\T^d)$.
In particular,
$	 \vertiii{f}_{\mathcal{C}^s(\T^d),k} :=  \sup_{0 < t < 1} (t^{-s} \omega_k (f,t)_{\infty; \T^d})$ for $k > s$. 
Similarly we define $\dot{\mathcal{C}}^s(\R^d)$ and $\dot{\mathcal{C}}^s(\T^d)$.}

Let $s > 0, 1 \leq p \leq \infty, 0 < q \leq \infty,$ and $-\infty < b < \infty$. The \emph{logarithmic Lipschitz space}
 ${\L}^{(s,-b)}_{p,q}(\R^d)$ is the collection of all $f \in L_p(\R^d)$ such that
\begin{equation}\label{DefLipschitz}
	\|f\|_{\L^{(s,-b)}_{p,q}(\R^d)} = \left(\int_0^{1} (t^{-s} (1 - \log t)^{-b} \omega_s(f,t)_{p;\R^d})^q \frac{dt}{t} \right)^{1/q} < \infty
\end{equation}
(the usual interpretation is made when $q=\infty$). Note that, unlike Besov spaces (see (\ref{DefBesov})), the semi(quasi)-norms given in (\ref{DefLipschitz}) are defined over the interval $(0,1)$. In addition, we will assume that $b > 1/q$ ($b \geq 0$ if $q=\infty$). These assumptions allow us to avoid trivial spaces.
For detailed study of the  spaces $\L^{(s,-b)}_{p,q}(\R^d)$ we refer the reader to \cite{EdmundsEvans}, \cite{Haroske}, and \cite{DominguezTikhonov19}. Some distinguished examples are 
\begin{equation}\label{LipEx1}
	\L^{(k,0)}_{p,\infty}(\R^d) = \dot{W}^k_p(\R^d), \quad 1 < p < \infty;
\end{equation}
\begin{equation}\label{LipEx2}
	\L^{(1,0)}_{1,\infty}(\R^d) = \text{BV}(\R^d) \quad (\text{the space formed by bounded variation functions});
\end{equation}
\begin{equation*}
	\L^{(1,0)}_{\infty,\infty}(\R^d) = \L(\R^d) \quad (\text{classical Lipschitz space}).
\end{equation*}
The periodic counterpart $\L^{(s,-b)}_{p,q}(\T^d)$ can be introduced in analogy to (\ref{DefLipschitz}).

\subsection{Interpolation methods}\label{SectionInterpolation}

Let $(A_0, A_1)$ be a compatible pair of quasi-Banach spaces (in fact, for the purposes of this paper, it would be enough to assume that $(A_0, A_1)$ is a compatible pair of seminormed spaces). The \emph{Peetre $K$-functional} is defined by
\begin{equation}\label{DefKFunct}
	K(t,f) = K(t, f; A_0, A_1) = \inf_{f_1 \in A_1} \{\|f - f_1\|_{A_0} + t \|f_1\|_{A_1}\}, \quad t > 0, \quad f \in A_0 + A_1.
\end{equation}

Let $0 <\theta < 1, -\infty < b < \infty$, and $0 < q \leq \infty$. The \emph{logarithmic interpolation space} $(A_0,A_1)_{\theta,q;b}$ is the set formed by all $f \in A_0 + A_1$ such that
\begin{equation}\label{DefInter}
	\|f\|_{(A_0,A_1)_{\theta,q;b}} = \left(\int_0^\infty (t^{-\theta} (1 + |\log t|)^b K(t,f))^q \frac{dt}{t} \right)^{1/q} < \infty
\end{equation}
with the usual modification for $q=\infty$. See \cite{Gustavsson} and \cite{GogatishviliOpicTrebels}. In particular, if $b=0$ in $(A_0,A_1)_{\theta,q;b}$ then we obtain the classical real interpolation space $(A_0, A_1)_{\theta,q}$; see \cite{BennettSharpley, BerghLofstrom}.

Assume now that $A_1 \hookrightarrow A_0$. Then it is plain to check that $K(t, f) \asymp \|f\|_{A_0}$ for $t > 1$. Consequently, we have
\begin{equation}\label{DefInter*}
	\|f\|_{(A_0,A_1)_{\theta,q;b}} \asymp \left(\int_0^1 (t^{-\theta} (1 + |\log t|)^b K(t,f))^q \frac{dt}{t} \right)^{1/q}.
\end{equation}
This fact together with the finer tuning given by logarithmic weights allows us to introduce \emph{limiting interpolation spaces} with $\theta=1$. Namely, the space $(A_0, A_1)_{(1,b),q}$ is the collection of all $f \in A_0$ for which
\begin{equation}\label{Klimitspace}
	\|f\|_{(A_0,A_1)_{(1,b),q}} = \left(\int_0^1 (t^{-1} (1 + |\log t|)^b K(t,f))^q \frac{dt}{t} \right)^{1/q} < \infty.
\end{equation}
Note that this space becomes trivial if $b \geq -1/q$ ($b > 0$ if $q=\infty$). Then we will assume that $b < -1/q$ ($b \leq 0$ if $q=\infty$). We also remark that the integral $\int_0^1$ in (\ref{Klimitspace}) can be replaced (with equivalence of norms) by $\int_0^{a}$ for any $a > 0$. For further details and properties, we refer the reader to \cite{Astashkin} and \cite{DominguezTikhonov19} (see also  \cite{FernandezSignes} and \cite{GogatishviliOpicTrebels}.)

\newpage 
\section{Auxiliary results}\label{SectionAuxResults}

\subsection{Hardy-type inequalities}\label{SectionHardy}
Below we collect some Hardy-type inequalities for averages  that we will use on several occasions in the rest of the paper.

\begin{lem}[{\cite[p. 196]{SteinWeiss}}]\label{LemmaHardySharp}
		Let $\alpha > 0$ and $1 \leq p < \infty$. Then for any non-negative measurable function $f$ on $(0,\infty)$,
		\begin{equation}\label{HardyIneq*}
			\left( \int_0^\infty \left( \int_0^t f(u) du\right)^p t^{-\alpha} \frac{dt}{t}\right)^{1/p} \leq \frac{p}{\alpha} \left( \int_0^\infty (t f(t))^p t^{-\alpha} \frac{dt}{t}\right)^{1/p}
		\end{equation}
		and
		\begin{equation}\label{HardyIneq}
			\left( \int_0^\infty \left( \int_t^\infty f(u) du\right)^p t^{\alpha} \frac{dt}{t}\right)^{1/p} \leq \frac{p}{\alpha} \left( \int_0^\infty (t f(t))^p t^\alpha \frac{dt}{t}\right)^{1/p}.
		\end{equation}
		The corresponding results also hold true for non-negative measurable functions on $(a,b) \subset (0,\infty)$.
	\end{lem}

\begin{lem}[{\cite[Theorem 6.4]{BennettRudnick}}]\label{LemmaHardyLog}
	Let $\alpha > 0, 1 \leq p \leq \infty$ and $-\infty < b < \infty$. Then for any non-negative measurable function $f$ on $(0,1)$,
	\begin{equation}\label{HardyInequal1**}
		\left(\int_0^1 \left(\int_0^t f(u) du \right)^p t^{-\alpha} (1 +| \log t|)^b  \frac{dt}{t}\right)^{1/p} \lesssim \left(\int_0^1 (t  f(t))^p t^{-\alpha} (1+|\log t|)^b \frac{dt}{t}\right)^{1/p}
	\end{equation}
	and
	\begin{equation}\label{HardyInequal2**}
		\left(\int_0^1 \left( \int_t^1 f(u) du \right)^p t^{\alpha} (1 +| \log t|)^b \frac{dt}{t}\right)^{1/p} \lesssim \left(\int_0^1 (t f(t))^p t^\alpha (1+|\log t|)^b  \frac{dt}{t}\right)^{1/q}.
	\end{equation}
	Furthermore, if $f (t) = t^{\lambda -1} g(t)$ with $\lambda > 0$ and $g$ is a decreasing function, then  inequalities (\ref{HardyInequal1**}) and (\ref{HardyInequal2**}) still hold true when $0 < p < 1$.
	
	The corresponding results also hold true for non-negative measurable functions on $(a,b) \subset (0,\infty)$.
\end{lem}

The limiting version of Lemma \ref{LemmaHardyLog} reads as follows.

 \begin{lem}[{\cite[Theorem 6.5]{BennettRudnick}}]
	Let $1 \leq p \leq \infty$ and $b + 1/p \neq 0$. Then for any non-negative measurable function $f$ on $(0,1)$,
	\begin{enumerate}[\upshape(i)]
	 \item if $b + 1/p > 0$,
	\begin{equation*}
		\left(\int_0^1 \left((1 +| \log t|)^b \int_0^t f(u) du \right)^p \frac{dt}{t}\right)^{1/p} \lesssim \left(\int_0^1 (t (1+|\log t|)^{b + 1} f(t))^p \frac{dt}{t}\right)^{1/p};
	\end{equation*}
	\item if $b + 1/p < 0$,
	\begin{equation}\label{HardyInequal4}
		\left(\int_0^1 \left((1 +| \log t|)^b \int_t^1 f(u) du \right)^p \frac{dt}{t}\right)^{1/p} \lesssim \left(\int_0^1 (t (1+|\log t|)^{b + 1} f(t))^p \frac{dt}{t}\right)^{1/p}.
	\end{equation}
	\end{enumerate}
\end{lem}

\begin{lem}\label{vspom1}
Let $\alpha < 0, -\infty < \beta < \infty$, and $0 < q \leq \infty$. Then for any non-negative decreasing function $f$ on $(0,1)$ and $t \in (0,1/2)$, 
\begin{equation}\label{HardyInequal1***}
	\left(\int_t^1 \left(v^\alpha \int_0^v u^\beta f(u) du\right)^q \frac{dv}{v} \right)^{1/q} \asymp t^{\alpha} \int_0^t v^\beta f(v) dv + \left(\int_t^1 (v^{\alpha +\beta + 1}  f(v))^q \frac{dv}{v} \right)^{1/q}.
\end{equation}
The corresponding result also holds true for functions $f$ on $(0,\infty)$ and $t > 0$.
\end{lem}
\begin{proof}
	We have
	\begin{equation}\label{LemH1}
	\left(\int_t^1 \left(v^\alpha \int_0^v u^\beta f(u) du\right)^q \frac{dv}{v} \right)^{1/q}  \asymp t^{\alpha} \int_0^t v^\beta f(v) dv + \left(\int_t^1 \left(v^{\alpha} \int_t^v u^\beta f(u) du \right)^q \frac{dv}{v} \right)^{1/q}.
	\end{equation}
It is clear that (\ref{HardyInequal1***}) immediately follows from (\ref{LemH1}) and the following estimates
	\begin{align}
	\left(\int_{2 t}^1 (v^{\alpha + \beta + 1} f(v))^q \frac{dv}{v} \right)^{1/q} & \lesssim \left(\int_t^1 \left(v^{\alpha} \int_t^v u^\beta f(u) du \right)^q \frac{dv}{v} \right)^{1/q} \nonumber \\
	&\lesssim \left(\int_t^1 (v^{\alpha + \beta + 1} f(v))^q \frac{dv}{v} \right)^{1/q}.\label{LemH2}
	\end{align}
To show (\ref{LemH2}), we first use the monotonicity of $f$:
	\begin{equation*}
		\int_t^1 \left(v^{\alpha} \int_t^v u^\beta f(u) du \right)^q \frac{dv}{v} \gtrsim \int_{2 t}^1 (v^{\alpha + \beta +1}  f(v))^q \frac{dv}{v}.
	\end{equation*}
Second, the right-hand side estimate in (\ref{LemH2})  is an immediate consequence of Lemma \ref{LemmaHardySharp} for $q \geq 1$. If $q < 1$, we let $j \in \N_0$ be such that $2^{-j-1} \leq t < 2^{-j}$. Then
	\begin{align*}
		\int_t^1 \left(v^{\alpha} \int_t^v u^\beta f(u) du \right)^q \frac{dv}{v} & \lesssim \sum_{k=0}^j 2^{-k \alpha q} \left(\sum_{\nu = k}^j 2^{-\nu (\beta + 1)}  f(2^{-\nu}) \right)^q \\
		& \hspace{-3cm} \leq \sum_{k=0}^j 2^{-k \alpha q} \sum_{\nu = k}^j (2^{-\nu (\beta + 1) } f(2^{-\nu}))^q
\lesssim \int_t^1 (v^{\alpha + \beta + 1} f(v))^q \frac{dv}{v},
	\end{align*}
which completes the proof of (\ref{LemH2}).
\end{proof}

	\subsection{Properties of the modulus of smoothness}\label{SectionModuli}
	For later use, we recall some well-known properties of the moduli of smoothness \cite{BennettSharpley, KT}. Let $\mathcal{X} \in \{\R^d, \T^d, \Omega\}$ and let $k, m \in \N$ and $1 \leq p \leq \infty$. We have
	\begin{enumerate}[(a)]
	\item $\omega_k(f,t)_{p;\mathcal{X}}$ is a non-negative non-decreasing function of $t$;
	\item if $u \geq 1$ then
	\begin{equation}\label{HomoMod}
		\omega_k(f, t u)_{p;\mathcal{X}} \lesssim u^k \omega_k(f,t)_{p;\mathcal{X}};
	\end{equation}
	\item if $k < m$ then
	\begin{equation}\label{JacksonInequal}
		\omega_m(f,t)_{p;\mathcal{X}} \lesssim \omega_k(f,t)_{p;\mathcal{X}}
	\end{equation}
	and
	\begin{equation}\label{MarchaudInequal}
		\omega_k (f,t)_{p;\mathcal{X}} \lesssim t^{k} \int_t^\infty \frac{\omega_m(f,u)_{p; \mathcal{X}}}{u^k} \frac{du}{u};
	\end{equation}
	\item if $|\nabla^k f| \in L_p(\mathcal{X})$ then
	\begin{equation}\label{DerMod}
		\omega_k(f,t)_{p;\mathcal{X}} \lesssim t^k \| |\nabla^k f|\|_{L_p(\mathcal{X})}.
	\end{equation}
	\end{enumerate}



It is well known that Sobolev spaces on $\R^d$ can be characterized through moduli of smoothness (see \eqref{LipEx1}). Namely, if $f \in W^k_p(\R^d)$ then
\begin{equation}\label{ddd2}
	\|f\|_{\dot{W}^k_p(\R^d)} = \||\nabla^k f|\|_{L_p(\R^d)} \asymp \sup_{t > 0} t^{-k} \omega_k(f,t)_{p;\R^d}, \quad k \in \mathbb{N}, \quad 1 < p < \infty
\end{equation}
see, e.g., \cite[Chapter V, Section 3.5, Proposition 3, p. 139]{Stein70}, \cite[Proposition 2.4]{KolyadaLerner} and \cite[page 174]{Triebel11}. Next we give a simple argument based on interpolation which enables us to extend the previous characterization to Sobolev spaces on domains.

\begin{lem}
	Let $\mathcal{X} \in \{\R^d, \T^d, \Omega\}$. Let $k \in \mathbb{N}$ and $1 < p < \infty$. Assume $f \in W^k_p(\mathcal{X})$. Then
	\begin{equation}\label{SobolevModuliD}
	\|f\|_{W^k_p(\mathcal{X})} \asymp \|f\|_{L_p(\mathcal{X})} +  \sup_{0 < t < 1} t^{-k} \omega_k(f,t)_{p;\mathcal{X}}.
\end{equation}
\end{lem}
\begin{proof}
	Since the the unit ball of $W^k_p(\mathcal{X})$ is closed in $L_p(\mathcal{X})$, we know that
	\begin{equation*}
		\|f\|_{W^k_p(\mathcal{X})} \asymp \sup_{t > 0} t^{-1} K(t, f; L_p(\mathcal{X}), W^k_p(\mathcal{X}));
	\end{equation*}
	this fact is well known in the interpolation community under the name of Gagliardo completeness of the couple $(L_p(\mathcal{X}), W^k_p(\mathcal{X}))$ (cf. \cite[p. 320]{BennettSharpley}). A simple change of variables and Lemma A.2 imply that
	\begin{align*}
		\|f\|_{W^k_p(\mathcal{X})} &\asymp \sup_{t > 0} t^{-k} (\min\{1,t^k\} \|f\|_{L_p(\mathcal{X})} + \omega_k(f,t)_{p;\mathcal{X}}) \\
		& \asymp \|f\|_{L_p(\mathcal{X})} +  \sup_{t > 0} t^{-k}  \omega_k(f,t)_{p;\mathcal{X}} \\
		& \asymp \|f\|_{L_p(\mathcal{X})} +  \sup_{0 < t < 1} t^{-k}  \omega_k(f,t)_{p;\mathcal{X}}
	\end{align*}
	where we have used $\omega_k(f,t)_{p;\mathcal{X}} \lesssim \|f\|_{L_p(\mathcal{X})}$ in the last step.
\end{proof}

The (fractional) modulus of smoothness $\omega_s(f,t)_{p; \R^d}, \, s > 0,$ also satisfy the properties (a)--(c) listed above. 


Let $1 \leq p \leq \infty, s > 0, 0 < q \leq \infty, 0 < \theta < 1$, and $b > 1/q \, (b \geq 0 \text{ if } q= \infty)$. According to Lemma \ref{LemmaKfunctLS}, we have
\begin{equation}\label{BInter}
	\dot{B}^{\theta s}_{p,q}(\R^d) = (L_p(\R^d),\dot{H}^s_p(\R^d))_{\theta,q},
\end{equation}
\begin{equation}\label{LipLimInter}
	\L^{(s,-b)}_{p,q}(\R^d) = (L_p(\R^d), \dot{H}^s_p(\R^d))_{(1,-b),q},
\end{equation}
see (\ref{DefLipschitz}) and (\ref{Klimitspace}). These formulas also hold true for function spaces over $\T^d$.

Using Lemma \ref{LemmaKfunctLS} and the closedness of the unit ball of $\dot{H}^s_p(\R^d)$ in $L_p(\R^d)$ for $p \in (1, \infty)$, one can obtain the homogeneous counterpart of \eqref{SobolevModuliD} for fractional Sobolev spaces. More specifically, we have

\begin{lem}\label{Lemma3.6}
	Let $\mathcal{X} \in \{\R^d, \T^d\}$. Let $s > 0$ and $1 < p < \infty$. Assume $f \in L_p(\mathcal{X}) \cap \dot{H}^s_p(\mathcal{X})$. Then
	\begin{equation}\label{SobolevModuliDPer'}
	\|f\|_{\dot{H}^s_p(\mathcal{X})} \asymp \sup_{t > 0} t^{-s} \omega_s(f,t)_{p;\mathcal{X}}.
\end{equation}
In particular, if $s=k \in \N$ then (cf. \eqref{ddd2})
	\begin{equation}\label{SobolevModuliDPer'8}
	\|f\|_{\dot{W}^k_p(\mathcal{X})} \asymp \sup_{t > 0} t^{-k} \omega_k(f,t)_{p;\mathcal{X}}.
\end{equation}
\end{lem}

Taking into account \eqref{HomoMod}, we can rewrite \eqref{SobolevModuliDPer'} as
\begin{equation}\label{SobLip}
	\dot{H}^s_p(\R^d) = \L^{(s,0)}_{p,\infty}(\R^d), \quad 1 < p < \infty, \quad s > 0;
\end{equation}
this extends (\ref{LipEx1}) to the fractional setting. On the other hand, it is worthwhile to mention that  formula \eqref{SobolevModuliDPer'8} also holds true for the Sobolev space $\dot{W}^k_p (\Omega)$, but this is not the case for the smaller Sobolev space $(\dot{W}^k_p (\Omega))_0$ (take, e.g., any non-zero constant function in $\Omega$ and thus the right-hand side of \eqref{SobolevModuliDPer'8} is zero but clearly this function does not belong to $(\dot{W}^k_p (\Omega))_0$).

\subsection{Some interpolation lemmas}\label{SectionInterpolationLemma}

For later use, we collect below some useful Holmstedt's formulas.

\begin{lem}\label{Holmstedt}
	 Assume that $0 < \theta < 1, 0 < q, r \leq \infty$ and $b < -1/r \, (b \leq 0 \text{ if } r=\infty)$. Let $K(t,f) = K(t,f; A_0, A_1), \,t \in (0, \infty)$. Then we have
	\begin{enumerate}[\upshape(i)]
		\item
	\begin{equation}\label{LemmaHolmstedt1}
		K(t^\theta, f; A_0, (A_0, A_1)_{\theta,r}) \asymp t^\theta \left(\int_t^\infty (u^{-\theta} K(u,f))^r \frac{du}{u} \right)^{1/r},
	\end{equation}
	\item
			\begin{equation}\label{LemmaHolmstedt1*}
			K(t^{1-\theta},f; (A_0,A_1)_{\theta,q}, A_1) \asymp \left(\int_0^t (u^{-\theta} K(u,f))^q\frac{du}{u}\right)^{1/q}.
		\end{equation}
		\end{enumerate}
		Assume further that $A_1 \hookrightarrow A_0$ and $t \in (0,1)$. Then we have
		\begin{enumerate}
	\item[\upshape(iii)]
	\begin{align}
		K(t (1-\log t)^{-b-1/r}, f; A_0, (A_0, A_1)_{(1,b),r}) & \nonumber\\
		&\hspace{-5cm}\asymp K(t,f) + t (1 - \log t)^{-b-1/r} \left(\int_t^1 (u^{-1} (1-\log u)^{b} K(u,f))^r \frac{du}{u}\right)^{1/r}, \label{LemmaHolmstedt2}
	\end{align}
		\item[\upshape(iv)]
	\begin{align}
		K(t^{1-\theta} (1 - \log t)^{-b -1/r}, f ; (A_0, A_1)_{\theta, q}, (A_0, A_1)_{(1,b), r})  &\asymp \Big(\int_0^t (u^{-\theta} K(u,f))^q \frac{du}{u} \Big)^{1/q}  \nonumber\\
		& \hspace{-6cm}  + t^{1-\theta} (1 - \log t)^{-b -1/r}\Big( \int_t^1 (u^{-1} (1 - \log u)^b K(u,f))^r \frac{du}{u}\Big)^{1/r}.\label{LemmaHolmstedt3}
	\end{align}
	\end{enumerate}
\end{lem}

For the proofs see \cite[Corollary 2.3, Chapter 5, page 310]{BennettSharpley} and \cite[Theorems 4.1 and 4.4]{FernandezSignes}.

%

It is well known that the spaces $L_{r,q}(\log L)_b(\Omega)$ can be generated from the couple $(L_p(\Omega), L_\infty(\Omega)), p < r$, applying the logarithmic interpolation method (\ref{DefInter}). Namely, if $0 < r < p < \infty, 0 < q \leq \infty$ and $-\infty < b < \infty$ then
\begin{equation*}
	(L_r(\Omega), L_\infty(\Omega))_{1-\frac{r}{p}, q; b} = L_{p, q}(\log L)_{b}(\Omega);
\end{equation*}
see \cite[Corollary 5.3]{GogatishviliOpicTrebels}. Next we complement this result by showing that the spaces $L_{\infty, q}(\log L)_{b}(\Omega)$ can be characterized as limiting interpolation spaces (see \eqref{Klimitspace}).

\begin{lem}
	Let $0 < p < \infty, 0 < q \leq \infty,$ and $b < - 1/q \, (b \leq 0 \, \text{if} \, \, q= \infty)$. Then we have
	\begin{equation}\label{LorentzZygmundLimiting}
		(L_p(\Omega), L_\infty(\Omega))_{(1,b),q} = L_{\infty, q}(\log L)_{b}(\Omega).
	\end{equation}
\end{lem}
\begin{proof}
	Since (cf. Lemma \ref{LemmaHolm4})
	\begin{equation*}
		K(t, f ; L_p(\Omega), L_\infty(\Omega)) \asymp \left(\int_0^{t^p} f^\ast(u)^p du \right)^{1/p}
	\end{equation*}
	 we arrive at
	\begin{equation*}
		\|f\|_{(L_p(\Omega), L_\infty(\Omega))_{(1,b),q}}  \asymp \left( \int_0^1 t^{-q/p} (1 - \log t)^{b q}  \left( \int_0^t f^\ast(u)^p du \right)^{q/p} \frac{dt}{t}\right)^{1/q}.
	\end{equation*}
	Obviously, we have
	\begin{equation*}
		\|f\|_{(L_p(\Omega), L_\infty(\Omega))_{(1,b),q}}  \gtrsim \left( \int_0^1 (1 - \log t)^{b q}  f^\ast(t)^q \frac{dt}{t}\right)^{1/q} = \|f\|_{L_{\infty,q}(\log L)_{b}(\Omega)}.
	\end{equation*}
	
	The converse inequality is a consequence of the Hardy's inequality (\ref{HardyInequal1**}), which in fact holds for any $0 < p < \infty$ and $0 < q \leq \infty$ due to monotonicity of $f^\ast(t)$.
\end{proof}

Extrapolation means allow us to characterize Zygmund spaces (respectively, Lorentz-Zygmund spaces) in terms of the simpler Lebesgue spaces (respectively, Lorentz spaces). See \cite{JawerthMilman}, \cite{Milman} and \cite{EdmundsTriebel}. For later use, we recall the extrapolation description of $L_\infty (\log L)_{b}(\Omega)$.

	\begin{lem}[{\cite[Section 2.6.2, pages 69--75]{EdmundsTriebel}}]
	Assume $b < 0$. We have
	\begin{equation}\label{ExtrapolationLZ}
		\|f\|_{L_\infty (\log L)_{b}(\Omega)} \asymp \sup_{j \geq 0} 2^{j b} \|f\|_{L_{2^j d}(\Omega)}.
	\end{equation}
	\end{lem}



\newpage

\section{Subcritical case}\label{subcritical}

We are concerned with the subcritical case of the Sobolev's embedding theorem which claims that if $k \in \mathbb{N}, 1 \leq p < \infty$ and $k < d/p$, then
\begin{equation}\label{SobolevClassical*}
	W^k_p(\Omega) \hookrightarrow L_{p^\ast,p}(\Omega), \quad p^\ast = \frac{d p}{d - k p}.
\end{equation}
See \cite{Hunt}, \cite{Oneil} and \cite{Peetre}.


\subsection{Sobolev embeddings with fixed domain space}\label{Section4.1}
We start with the following result.


\begin{thm}\label{ThmSob}
	Let $1 \leq p < \infty, 0 < q \leq \infty$ and $k \in \mathbb{N}$.
Assume that $k < d/p$. Let $p^\ast = d p /(d - k p)$. The following statements are equivalent:
	\begin{enumerate}[\upshape(i)]
		\item
		\begin{equation*}
			W^k_p(\Omega) \hookrightarrow L_{p^\ast,q}(\Omega),
		\end{equation*}
		\item for $f \in L_p(\Omega)$ and $t \in (0,1)$, we have
		\begin{equation}\label{ThmSob1*}
			\left(\int_0^{t^d} f^\ast(u)^p du \right)^{1/p} + t^k \left(\int_{t^d}^1 u^{q/p^\ast} f^\ast(u)^q \frac{du}{u} \right)^{1/q} \lesssim t^k \|f\|_{L_p(\Omega)} + \omega_k(f,t)_{p;\Omega},
		\end{equation}
		\item if $s \to k-$ then there exists $C > 0$, which is independent of $s$, such that
		\begin{equation}\label{ThmSob1*extrapol}
			\|f\|_{L_{\frac{d p}{d - s p}, q}(\Omega)} \leq C (k - s)^{1/q} \|f\|_{B^s_{p,q}(\Omega),k},
		\end{equation}
		\item
		\begin{equation*}
		q \geq p.
		\end{equation*}
	\end{enumerate}
\end{thm}

The counterpart of the previous result for $\R^d$ reads as follows.

\begin{thm}\label{ThmSobRd}
	Let $1 \leq p < \infty, 0 < q \leq \infty$ and $k \in \mathbb{N}$.
Assume that $k < d/p$. Let $p^\ast = d p /(d - k p)$. The following statements are equivalent:
	\begin{enumerate}[\upshape(i)]
		\item
		\begin{equation*}
			W^k_p(\R^d) \hookrightarrow L_{p^\ast,q}(\R^d),
		\end{equation*}
		\item for $f \in L_p(\R^d)$ and $t \in (0,\infty)$, we have
		\begin{equation}\label{ThmSob1*Rd}
			\left(\int_0^{t^d} f^\ast(u)^p du \right)^{1/p} + t^k \left(\int_{t^d}^\infty u^{q/p^\ast} f^\ast(u)^q \frac{du}{u} \right)^{1/q} \lesssim \min\{1,t^k\} \|f\|_{L_p(\R^d)} + \omega_k(f,t)_{p;\R^d},
		\end{equation}
		\item if $s \to k-$ then there exists $C > 0$, which is independent of $s$, such that
		\begin{equation*}
			\|f\|_{L_{\frac{d p}{d - s p}, q}(\R^d)} \leq C (k - s)^{1/q} \|f\|_{B^s_{p,q}(\R^d),k},
		\end{equation*}
		\item
		\begin{equation*}
		q \geq p.
		\end{equation*}
	\end{enumerate}
\end{thm}

Dealing with homogeneous spaces, we establish the following

\begin{thm}\label{ThmSobRdHom}
	Let $1 \leq p < \infty, 0 < q \leq \infty$ and $k \in \mathbb{N}$.
Assume that $k < d/p$. Let $p^\ast = d p /(d - k p)$. The following statements are equivalent:
	\begin{enumerate}[\upshape(i)]
		\item
		\begin{equation*}
			(\dot{W}^k_p(\R^d))_0 \hookrightarrow L_{p^\ast,q}(\R^d),
		\end{equation*}
		\item for $f \in L_p(\R^d) + (\dot{W}^k_p(\R^d))_0$ and $t \in (0,\infty)$, we have
		\begin{equation}\label{ThmSob1*Rd2}
			\left(\int_0^{t^d} f^\ast(u)^p du \right)^{1/p} + t^k \left(\int_{t^d}^\infty u^{q/p^\ast} f^\ast(u)^q \frac{du}{u} \right)^{1/q} \lesssim \omega_k(f,t)_{p;\R^d},
		\end{equation}
		\item
 if $s \to k-$ 
then there exists $C > 0$, which is independent of $s$, such that
		\begin{equation}
\label{ThmSob1*Rd2Extrapol}			\|f\|_{L_{\frac{d p}{d - s p}, q}(\R^d)} \leq C (k - s)^{1/q} \|f\|_{\dot{B}^s_{p,q}(\R^d),k},
		\end{equation}
		\item
		\begin{equation*}
		q \geq p.
		\end{equation*}
	\end{enumerate}
\end{thm}

 Our method can also be applied to deal with homogeneous Sobolev spaces on $\Omega$. In such a case, we can only expect to obtain local estimates for subsets $\Omega_0$ of $\Omega$ with positive distance to the boundary of $\Omega$ (see, e.g., \cite[Sections 5.8.2.(a) and 6.3]{Evans}). Accordingly, we will pay special attention to the dependence with respect to $\text{dist}(\Omega_0, \partial \Omega)$.

\begin{thm}\label{ThmSobOmegaHom}
	Let $1 \leq p < \infty$ and $k \in \mathbb{N}$.
Assume that $k < d/p$. Let $p^\ast = d p /(d - k p)$. Consider $\Omega_0\subset \Omega$ with $\emph{dist}(\Omega_0, \partial \Omega) > 0$.
	\begin{enumerate}[\upshape(i)]
		\item We have
		\begin{equation}\label{ThmSob1*OmegaHom}
			\left(\int_0^{t^d} f^\ast(u)^p du \right)^{1/p} + t^k \left(\int_{t^d}^1 u^{p/p^\ast} f^\ast(u)^p \frac{du}{u} \right)^{1/p} \leq C_1 (\emph{dist}(\Omega_0, \partial \Omega))^{-k} \omega_k(f,t)_{p;\Omega},
		\end{equation}
		for all $f \in L_p(\Omega)$ such that $\emph{supp } (f) \subseteq \Omega_0$ and $t \in \big(0,\min\{1, \frac{\emph{dist}(\Omega_0, \partial \Omega)}{k}\}\big)$. Here, $C_1$ is a positive constant which is independent of $f, t, \Omega_0$ and $\Omega$.
		\item Let $s \to k-$. We have
		\begin{equation}\label{ThmSob1*OmegaHomExtrapol}
			\|f\|_{L_{\frac{d p}{d - s p}, p}(\Omega)} \leq C_2 (k - s)^{1/p} (\emph{dist}(\Omega_0, \partial \Omega))^{-k} \|f\|_{\dot{B}^s_{p,p}(\Omega),k}
		\end{equation}
		for all $f \in L_p(\Omega)$ such that $\emph{supp } (f) \subseteq \Omega_0$. Here, $C_2$ is a positive constant which is independent of $f, s$ and $\Omega_0$ and the Lorentz norm \eqref{DefLZ} and the Besov norm \eqref{DefBesov} are normalized with respect to $\emph{dist} (\Omega_0, \partial \Omega)$, i.e.,
		\begin{equation}\label{ThmSob1*OmegaHomExtrapolNew}
			\|f\|_{L_{\frac{d p}{d - s p}, p}(\Omega)} = \bigg( \int_0^{\min\big\{1, \big(\frac{\emph{dist}(\Omega_0, \partial \Omega)}{k} \big)^d\big\} } t^{-\frac{s p}{d}}  f^\ast(t)^p \, dt\bigg)^{\frac{1}{p}}
		\end{equation}
		and
		\begin{equation}\label{ThmSob1*OmegaHomExtrapolNew2}
			\|f\|_{\dot{B}^s_{p,p}(\Omega),k}= \left( \int_0^{\min\{1, \frac{\emph{dist}(\Omega_0, \partial \Omega)}{k}\}} (t^{-s} \omega_k (f,t)_{p;\Omega})^p \frac{dt}{t}\right)^{1/p},
		\end{equation}
		respectively.
	\end{enumerate}
\end{thm}

Before we proceed with the proof of these theorems, some remarks are in order.

\begin{rem}\label{Rembeta}
	(i) {Estimate of oscillations of functions in terms of moduli of smoothness have shown  to be a very useful tool in embedding theorems. The following inequality of Kolyada \cite[Cor. 6]{Kolyada} (see the previous results in \cite{Ulyanov}) plays a central role in embedding theorems (see, e.g., \cite{CaetanoGogatishviliOpic}, \cite{CaetanoGogatishviliOpic11}, \cite{KolyadaNafsa}, \cite{KolyadaLerner}). Let $1 \leq p < \infty$, then
	\begin{equation}\label{K}
		t \left(\int_{t^d}^\infty u^{-p/d} \int_0^{u} (f^\ast(v) - f^\ast(u))^p dv \frac{du}{u} \right)^{1/p} \lesssim \omega_1(f,t)_{p;\R^d}, \quad f \in L_p(\R^d).
	\end{equation}
	This is a stronger version of the well-known oscillation inequality
	\begin{equation}\label{oscillation}
		f^{**}(t)-f^*(t) \lesssim t^{-1/p} \omega_1(f,t^{1/d})_{p;\R^d},
	\end{equation}
	which has been intensively investigated. It is worth mentioning that \eqref{oscillation} admits extensions to  r.i. spaces on $\R^d$ (or more generally, metric measure spaces; see \cite{MartinMilman06b} and \cite{Mastylo}) and the quasi-Banach case $p \in (0,1)$ (see \cite{EdmundsEvansKaradzhov2}). More closely to our interpolation-based methodology, Mart\'in and Milman \cite{MartinMilman14} showed an interesting connection between \eqref{oscillation} and optimal decompositions of $K$-functionals. }
	
	Assume $1\leq  p < d$. Then  inequality \eqref{ThmSob1*Rd2} with $k=1$ and $q=p$ reads as follows
	\begin{equation}\label{K00}
	\left(\int_0^{t^d} f^\ast(u)^p du \right)^{1/p} + t \left(\int_{t^d}^\infty u^{p/p^\ast} f^\ast(u)^p \frac{du}{u} \right)^{1/p} \lesssim \omega_1(f,t)_{p;\R^d}, \quad f \in L_p(\R^d).
	\end{equation}
	Next we show that both inequalities (\ref{K}) and (\ref{K00}) are equivalent if $p \in (1,d)$, but (\ref{K00}) improves (\ref{K}) if $p=1$.
	
	Assume first that $p \in (1,d)$. Then we claim that
	\begin{align}
	t \left(\int_{t^d}^\infty u^{-p/d} \int_0^{u} (f^\ast(v) - f^\ast(u))^p dv \frac{du}{u} \right)^{1/p} \nonumber \\
	&\hspace{-5cm} \asymp \left(\int_0^{t^d} f^\ast(u)^p du \right)^{1/p} + t \left(\int_{t^d}^\infty u^{p/p^\ast} f^\ast(u)^p \frac{du}{u} \right)^{1/p}. \label{K1}
	\end{align}
	 To derive this, we will make use of the following result
	 \begin{equation*}
	 	 \int_0^{u} (f^{\ast \ast}(v)- f^\ast(v))^p d v \lesssim \int_0^{u} (f^\ast(v) - f^\ast(u))^p dv \lesssim \int_0^{2 u} (f^{\ast \ast}(v)- f^\ast(v))^p d v;
	 \end{equation*}
	 see \cite[Proposition 4.5]{CaetanoGogatishviliOpic}. We have (recall that $p^* = \frac{d p}{d-p}$)
	 \begin{align}
	 	t \left(\int_{t^d}^\infty u^{-p/d} \int_0^{u} (f^\ast(v) - f^\ast(u))^p dv \frac{du}{u} \right)^{1/p} \nonumber \\
		& \hspace{-6cm}\asymp t  \left(\int_{t^d}^\infty u^{-p/d} \int_0^{u} (f^{\ast \ast} (v) - f^\ast(v))^p dv \frac{du}{u} \right)^{1/p} \nonumber \\
		& \hspace{-6cm}\asymp
 \left(\int_0^{t^d} (f^{\ast \ast} (v) - f^\ast(v))^p dv\right)^{1/p}
+ t  \left(\int_{t^d}^\infty u^{-p/d} \int_{t^d}^{u} (f^{\ast \ast} (v) - f^\ast(v))^p dv \frac{du}{u} \right)^{1/p}\nonumber \\
		& \hspace{-6cm} \asymp \left(\int_0^{t^d} (f^{\ast \ast} (v) - f^\ast(v))^p dv\right)^{1/p} + t \left( \int_{t^d}^\infty  v^{p/p^\ast}(f^{\ast \ast} (v) - f^\ast(v))^p  \frac{dv}{v} \right)^{1/p} \nonumber\\
			& \hspace{-6cm} =: I + II.\label{K2}
	 \end{align}
	
	 Obviously,
	 \begin{equation}\label{K3}
	 	I \leq  \left(\int_0^{t^d} f^{\ast \ast} (v)^p dv\right)^{1/p} \quad \text{and} \quad II \leq  t \left( \int_{t^d}^\infty v^{p/p^\ast} f^{\ast \ast} (v)^p   \frac{dv}{v} \right)^{1/p}.
	 \end{equation}
	
	 On the other hand, since $(f^{\ast \ast}(t))' = \frac{f^\ast(t) - f^{\ast \ast}(t)}{t}$ (see (\ref{maximal})), it follows from the fundamental theorem of calculus that
	 \begin{equation}\label{K4}
	 	f^{\ast \ast}(t) = - \int_t^\infty (f^{\ast \ast}(u))' du = \int_t^\infty  \frac{f^{\ast \ast}(u) - f^{ \ast}(u)}{u} du,
	 \end{equation}
	 where we have also used that $\lim_{t \to \infty} f^{\ast \ast}(t) = 0$ (since $f \in L_p(\R^d)$). By (\ref{K4}), applying Hardy's inequality (\ref{HardyIneq}) together with H\"older's inequality, we obtain
	 \begin{align}
	 	 \left(\int_0^{t^d} f^{\ast \ast} (v)^p dv\right)^{1/p}  & \asymp \left(\int_0^{t^d} \left( \int_v^{t^d}  \frac{
f^{\ast \ast}(u) - f^{ \ast}(u)
}{u} du \right)^p dv \right)^{1/p}
\nonumber \\
		 & \hspace{1cm}+ t^{d/p}  \int_{t^d}^\infty  \frac{
f^{\ast \ast}(v) - f^{ \ast}(v)
}{v} dv \nonumber \\
		 &
\hspace{-2cm}\lesssim \left( \int_0^{t^d}(f^{\ast \ast}(v) - f^\ast(v))^p dv \right)^{1/p}
+ t \left(\int_{t^d}^\infty v^{p/p^\ast} (f^{\ast \ast}(v) - f^\ast(v))^p \frac{dv}{v} \right)^{1/p} \nonumber \\
		 & = I + II. \label{K5}
	 \end{align}
	 Similarly, we derive
	 \begin{equation}\label{K6}
	 	  t \left( \int_{t^d}^\infty v^{p/p^\ast} f^{\ast \ast} (v)^p   \frac{dv}{v} \right)^{1/p} \lesssim t \left(\int_{t^d}^\infty v^{p/p^\ast} (f^{\ast \ast}(v) - f^\ast(v))^p \frac{dv}{v} \right)^{1/p} = II.
	 \end{equation}
	 So a combination of (\ref{K2}), (\ref{K3}), (\ref{K5}) and (\ref{K6}) results in
	 \begin{align}
	 		t \left(\int_{t^d}^\infty u^{-p/d} \int_0^{u} (f^\ast(v) - f^\ast(u))^p dv \frac{du}{u} \right)^{1/p}  \nonumber\\
			&\hspace{-4cm}\asymp  \left(\int_0^{t^d} f^{\ast \ast} (v)^p dv\right)^{1/p}  +   t \left( \int_{t^d}^\infty v^{p/p^\ast} f^{\ast \ast} (v)^p   \frac{dv}{v} \right)^{1/p}. \label{K7}
	 \end{align}
It is clear that 
	 \begin{align}
	 	 \left(\int_0^{t^d} f^{\ast \ast} (v)^p dv\right)^{1/p}  +   t \left( \int_{t^d}^\infty v^{p/p^\ast} f^{\ast \ast} (v)^p   \frac{dv}{v} \right)^{1/p} \nonumber \\
		 & \hspace{-6cm}\geq \left(\int_0^{t^d} f^{\ast} (v)^p dv\right)^{1/p}  +   t \left( \int_{t^d}^\infty v^{p/p^\ast} f^{\ast} (v)^p   \frac{dv}{v} \right)^{1/p}. \label{K8}
	 \end{align}
	 Conversely, since $p > 1$, we  make use of  Hardy's inequality (\ref{HardyIneq*}) to estimate
	 \begin{equation}\label{K9}
	 	\left(\int_0^{t^d} f^{\ast \ast} (v)^p dv\right)^{1/p}  \lesssim  \left(\int_0^{t^d} f^{\ast} (v)^p dv\right)^{1/p}
		\end{equation}
		and
		\begin{align}
		 \left( \int_{t^d}^\infty v^{p/p^\ast} f^{\ast \ast} (v)^p   \frac{dv}{v} \right)^{1/p} & \lesssim t^{-d+d/p^\ast } \int_0^{t^d} f^\ast(v) dv \nonumber \\
		 & \hspace{1cm}+ \left( \int_{t^d}^\infty \left( v^{1/p^\ast} \frac{1}{v} \int_{t^d}^v f^\ast(u) du \right)^p \frac{dv}{v} \right)^{1/p} \nonumber\\
		 &\hspace{-1cm} \lesssim t^{-1} \left(\int_0^{t^d} f^\ast(v)^p dv \right)^{1/p}
 + \left( \int_{t^d}^\infty v^{p/p^\ast}  f^{\ast} (v)^p  \frac{dv}{v} \right)^{1/p}, \label{K10}
	 \end{align}
	 where we have also applied H\"older's inequality in the last step. Hence, by (\ref{K7})--(\ref{K10}), we conclude (\ref{K1}).
	
	 Suppose now that $f \in L_1(\R^d)$. Then  inequalities (\ref{K}) and (\ref{K00}) read as follows
\begin{equation}\label{K*}
		t \int_{t^d}^\infty u^{-1/d} \int_0^{u} (f^\ast(v) - f^\ast(u)) dv \frac{du}{u}  \lesssim \omega_1(f,t)_{1;\R^d}
	\end{equation}
	and
		\begin{equation}\label{K00*}
	\int_0^{t^d} f^\ast(u) du  + t \int_{t^d}^\infty u^{-1/d} f^\ast(u) du  \lesssim \omega_1(f,t)_{1;\R^d},
	\end{equation}
respectively.
	Further, (\ref{K*}) can be rewritten as
	\begin{equation}\label{K**}
		\int_0^{t^d} f^\ast(u) du  \lesssim \omega_1(f,t)_{1;\R^d}
	\end{equation}
since the left-hand sides of (\ref{K*}) and (\ref{K**}) are equivalent.
	Clearly, (\ref{K00*}) is stronger than (\ref{K**}). Furthermore, the terms
	\begin{equation*}
		\int_0^{t^d} f^\ast(u) du \quad \text{and} \quad  t \int_{t^d}^\infty u^{-1/d} f^\ast(u) du
	\end{equation*}
	given in the left-hand side of (\ref{K00*}) are not comparable. To check this, consider, e.g., the functions $f_1(u) = u^{-1} (-\log u)^{-\beta} \chi_{(0,1)}(u), \beta > 1,$ and $f_2(u) = u^{-1-1/d}(|\log u|)^{-\eta}, \eta > 1$.
	
	{
	(ii) Working with domains, some pointwise inequalities for rearrangement of functions in terms of the moduli of smoothness were obtained in Mart\'in \cite{Martin}. In particular, he obtained that (cf. \cite[Theorem 1]{Martin})
	\begin{equation}\label{Ma}
		f^{**}(t) \lesssim t^{\frac{1}{d} - \frac{1}{p}} \|f\|_{L_p(\Omega)} + \int_{t^{\frac{1}{d}}}^1 u^{-\frac{d}{p}} \omega_1(f,u)_{p;\Omega} \frac{du}{u}, \quad t \in (0,1),
	\end{equation}
	for all $f \in L_p(\Omega)$. Inequality  \eqref{ThmSob1*} with $k=1$ and $q=p$, i.e.,
			\begin{equation}\label{Ma1}
			\left(\int_0^{t} f^\ast(u)^p du \right)^{1/p} + t^{1/d} \left(\int_{t}^1 u^{1-p/d} f^\ast(u)^p \frac{du}{u} \right)^{1/p} \lesssim t^{1/d} \|f\|_{L_p(\Omega)} + \omega_1(f,t^{1/d})_{p;\Omega}
		\end{equation}
		is stronger than \eqref{Ma}. Indeed, it follows from H\"older's inequality and \eqref{Ma1} that
		\begin{align*}
		f^{**}(t) &\leq t^{-1/p} \left(\int_0^{t} f^\ast(u)^p du \right)^{1/p} \lesssim t^{1/d-1/p} \|f\|_{L_p(\Omega)} + t^{-1/p}\omega_1(f,t^{1/d})_{p;\Omega} \\
		& \lesssim t^{1/d - 1/p} \|f\|_{L_p(\Omega)} + \int_{t^{1/d}}^1 u^{-d/p} \omega_1(f,u)_{p;\Omega} \frac{du}{u}.
		\end{align*}
		
		(iii) 
 Note that  inequality \eqref{ThmSob1*OmegaHom} improves considerably \eqref{Ma1} (since the term $t^{1/d} \|f\|_{L_p(\Omega)}$ given in the right-hand side can be dropped) but, on the other hand, the use of \eqref{ThmSob1*OmegaHom} is restricted to functions whose supports are strictly contained in $\Omega$. More precisely, $\eqref{ThmSob1*OmegaHom}$ fails to be true when we consider, e.g., a sequence of functions $(f_n)$ in $C^\infty_0(\Omega)$ such that $\text{dist}(\text{supp } (f_n),\partial \Omega) \to 0$ as $n \to \infty$. This phenomenon is illustrated by the fact that the constant $(\text{dist}(\Omega_0, \partial \Omega))^{-k}$ given in the right-hand side of \eqref{ThmSob1*OmegaHom} blows up as $\text{dist}(\Omega_0, \partial \Omega) \to 0$. For instance, let us consider a sequence $(f_n)$ in $C^\infty_0(\Omega)$ converging to $\chi_\Omega$ in $L_p(\Omega)$. Assuming that \eqref{ThmSob1*OmegaHom} holds uniformly in $\Omega_0$, for each $t \in (0,1)$, we have
		\begin{equation}\label{p0}
			\left(\int_0^{t^d} f_n^\ast(u)^p du \right)^{1/p} + t^k \left(\int_{t^d}^1 u^{p/p^\ast} f_n^\ast(u)^p \frac{du}{u} \right)^{1/p} \leq C \omega_k(f_n,t)_{p;\Omega}, \quad n \in \N,
		\end{equation}
		where $C$ is a positive constant independent of $n$. Since
		\begin{equation*}
		 \lim_{n \to \infty}\omega_k(f_n,t)_{p;\Omega} = \omega_k(\chi_\Omega,t)_p=0
		 \end{equation*}
		 and, by \eqref{LemmaKFunctLorentz},
		 \begin{align*}
	 \lim_{n \to \infty}\left[ \left(\int_0^{t^d} f_n^\ast(u)^p du \right)^{1/p} + t^k \left(\int_{t^d}^1 u^{p/p^\ast} f_n^\ast(u)^p \frac{du}{u} \right)^{1/p} \right] &\asymp \lim_{n \to \infty} K(t^k, f_n ; L_p(\Omega), L_{p^\ast,p}(\Omega))  \\
	 & \hspace{-7cm}= K(t^k, \chi_\Omega; L_p(\Omega), L_{p^\ast,p}(\Omega)),
	\end{align*}
	it follows from \eqref{p0} that $K(t^k, \chi_\Omega; L_p(\Omega), L_{p^\ast,p}(\Omega)) =0$, which is obviously false.

	(iv) Let $1 \leq p \leq q \leq \infty$ and $k < d/p$. The extrapolation assertion in $\R^d$ given in \eqref{ThmSob1*Rd2Extrapol}, that is,
	 	\begin{equation}\label{ThmSob1*OmegaHomExtrapol2}
			\|f\|_{L_{\frac{d p}{d - s p}, q}(\R^d)} \leq C (k - s)^{1/q} \|f\|_{\dot{B}^s_{p,q}(\R^d),k }, \quad s \to k-,
		\end{equation}
		with $C > 0$ independent of $s$, has been also obtained in Karadzhov, Milman and Xiao \cite[Theorem 4]{KaradzhovMilmanXiao} based on different techniques, namely, best constants in Sobolev inequalities and sharp versions of the Holmstedt's reiteration formula. Related results also in $\R^d$ may be found in \cite{EdmundsEvansKaradzhov, EdmundsEvansKaradzhov2}. Note that the techniques of these papers cannot be adapted to treat domains. To the best of our knowledge, the extrapolation assertions for domains given in \eqref{ThmSob1*extrapol} and \eqref{ThmSob1*OmegaHomExtrapol} are new. In particular, \eqref{ThmSob1*OmegaHomExtrapol} shows an interesting phenomenon, namely, the counterpart of \eqref{ThmSob1*OmegaHomExtrapol2} for domains holds under the corresponding normalization of the involved norms given by \eqref{ThmSob1*OmegaHomExtrapolNew} and \eqref{ThmSob1*OmegaHomExtrapolNew2}. Again, we stress that \eqref{ThmSob1*OmegaHomExtrapol} is  valid only for functions with support strictly contained in $\Omega$.
		}
\end{rem}

\begin{proof}[Proof of Theorem \ref{ThmSob}]
	(i) $\Longrightarrow$ (ii):
 Let $f \in L_p(\Omega)$ and $t \in (0,1)$. Using (i) and Lemma A.2, we have
	\begin{equation*}
		K(t^k, f ; L_p(\Omega), L_{p^\ast,q}(\Omega)) \lesssim K(t^k,f; L_p(\Omega), W^k_p(\Omega)) \asymp t^k \|f\|_{L_p(\Omega)} +  \omega_k(f,t)_{p;\Omega}.
	\end{equation*}
	By Lemma A.1, we derive
	\begin{equation*}
		K(t, f ; L_p(\Omega), L_{p^\ast,q}(\Omega)) \asymp \left(\int_0^{t^{d/k}} f^\ast(u)^p du \right)^{1/p} + t \left(\int_{t^{d/k}}^1 u^{q/p^\ast} f^\ast(u)^q \frac{du}{u} \right)^{1/q},
	\end{equation*}
	which implies (ii).
	
	(ii) $\Longrightarrow$ (iii): From (ii) we have
	\begin{equation*}
		 t^k \left(\int_{t^d}^1 u^{q/p^\ast} f^\ast(u)^q \frac{du}{u} \right)^{1/q} \lesssim t^k \|f\|_{L_p(\Omega)} +\omega_k(f,t)_{p;\Omega}.
		\end{equation*}
		Let $0 < s < k$ and $q < \infty$. Then by \eqref{DefBesovInh}, we have
		\begin{align}
		 \int_0^1 t^{\frac{k -s}{d} q} \int_{t}^1 u^{q/p^\ast} f^\ast(u)^q \frac{du}{u} \frac{dt}{t} &\lesssim  \|f\|_{L_p(\Omega)}^q  \int_0^1 t^{(k-s)q}  \frac{dt}{t}+  \int_0^1 t^{-s q} \omega_k(f,t)_{p;\Omega}^q \frac{dt}{t} \nonumber \\
		 &\hspace{-3.5cm} \asymp (k-s)^{-1} \|f\|_{L_p(\Omega)}^q + \int_0^1 t^{-s q} \omega_k(f,t)_{p;\Omega}^q \frac{dt}{t} \asymp \|f\|_{B^s_{p,q}(\Omega),k}^q.
\label{ThmSob1}
		\end{align}
		Applying Fubini's theorem, we establish
		\begin{align}
		\int_0^1 t^{\frac{k -s}{d} q} \int_{t}^1 u^{q/p^\ast} f^\ast(u)^q \frac{du}{u} \frac{dt}{t}
\nonumber\\
		& \hspace{-4cm}\asymp (k - s)^{-1}  \int_0^1 u^{q/p^\ast} f^\ast(u)^q u^{\frac{k - s}{d} q} \frac{du}{u} = (k - s)^{-1}  \|f\|_{ L_{\frac{d p}{d - s p}, q}(\Omega)}^q.	\label{ThmSob2}
		\end{align}
		Combining (\ref{ThmSob1}) and (\ref{ThmSob2}) implies
		\begin{equation*}
			 \|f\|_{ L_{\frac{d p}{d - s p}, q}(\Omega)} \lesssim (k-s)^{1/q} \|f\|_{B^s_{p,q}(\Omega),k} .
		\end{equation*}
		The case $q=\infty$ is easier and so, we omit the proof.
		
		(iii) $\Longrightarrow$ (i): Since $s \to k-$, elementary computations yield that
		\begin{equation*}
			  \vertiii{f}^q_{B^s_{p,q}(\Omega),k} \lesssim \sup_{0 < t < 1} t^{-k q} \omega_k(f,t)_{p;\Omega}^q \int_0^1 t^{(k-s) q} \frac{dt}{t}  \lesssim (k-s)^{-1} \||\nabla^k f|\|_{L_p(\Omega)}^q,
		\end{equation*}
		where we have also used \eqref{DerMod}. Then, by (iii) and \eqref{DefBesovInh}, we derive
		\begin{equation*}
			\|f\|_{L_{\frac{d p}{d - s p}, q}(\Omega)} \lesssim (k - s)^{1/q} \|f\|_{B^s_{p,q}(\Omega),k} \asymp \|f\|_{L_p(\Omega)} + (k - s)^{1/q} \vertiii{f}_{B^s_{p,q}(\Omega),k}  \lesssim \|f\|_{W^k_p(\Omega)}.
		\end{equation*}
		Taking limits $s \to k-$ we arrive at (i).
		
		The equivalence between (i) and (iv) is well known (see the methodology developed in \cite{EdmundsKermanPick}).
		
\end{proof}

The proof of Theorem \ref{ThmSobRd} goes the same lines as that of Theorem \ref{ThmSob}.

\begin{proof}[Proof of Theorem \ref{ThmSobRdHom}]

		(iii) $\Longrightarrow$ (i): We will make use of the following formula obtained in \cite[Theorem 2]{KaradzhovMilmanXiao},
		\begin{equation*}
			\lim_{s \to k-} (k-s)^{1/q} \|f\|_{\dot{B}^s_{p,q}(\R^d),k} \asymp \| |\nabla^k f|\|_{L_p(\R^d)}, \quad f \in C^\infty_0(\R^d).
		\end{equation*}
		 Then, taking limits $s \to k-$ in (iii), we infer that
		\begin{equation*}
			\|f\|_{L_{p^*,q}(\R^d)} = \lim_{s \to k-}\|f\|_{L_{\frac{d p}{d - s p}, q}(\R^d)} \lesssim \| |\nabla^k f|\|_{L_p(\R^d)}, \quad f \in C^\infty_0(\R^d).
		\end{equation*}
		Thus (i) is obtained by the density of $C^\infty_0(\R^d)$ in $(\dot{W}^k_p(\R^d))_0$.
		
The rest of the proof is similar to the proof of Theorem \ref{ThmSob}.	
\end{proof}

\begin{proof}[Proof of Theorem \ref{ThmSobOmegaHom}]
 (i): By the classical Sobolev embedding theorem and Proposition A.3, we have
	\begin{equation*}
		K(t^k, f ; L_p(\Omega), L_{p^\ast,p}(\Omega)) \lesssim K(t^k,f; L_p(\Omega), (\dot{W}^k_p(\Omega))_0) \lesssim (\text{dist}(\Omega_0, \partial \Omega))^{-k}  \omega_k(f,t)_{p;\Omega}.
	\end{equation*}
	By (\ref{LemmaKFunctLorentz}), we have
	\begin{equation*}
		K(t, f ; L_p(\Omega), L_{p^\ast,p}(\Omega)) \asymp \left(\int_0^{t^{d/k}} f^\ast(u)^p du \right)^{1/p} + t \left(\int_{t^{d/k}}^1 u^{p/p^\ast} f^\ast(u)^p \frac{du}{u} \right)^{1/p},
	\end{equation*}
	which gives \eqref{ThmSob1*OmegaHom}.
	
 (ii): Assume that $f \in L_p(\Omega)$ and $\text{supp } (f) \subseteq \Omega_0$. It follows from \eqref{ThmSob1*OmegaHom} that
	\begin{equation*}
		 t^k \left(\int_{t^d}^{\min\big\{1, \big(\frac{\text{dist}(\Omega_0, \partial \Omega)}{k} \big)^d \big\}} u^{p/p^\ast} f^\ast(u)^p \frac{du}{u} \right)^{1/p} \lesssim (\text{dist}(\Omega_0, \partial \Omega))^{-k} \omega_k(f,t)_{p;\Omega}
		\end{equation*}
		for all $t \in \big(0,\min\{1, \frac{\text{dist}(\Omega_0, \partial \Omega)}{k}\}\big)$.
		Let $0 < s < k$. Therefore,
		\begin{align}
		 \int_0^{\min\big\{1, \big(\frac{\text{dist}(\Omega_0, \partial \Omega)}{k} \big)^d\big\} } t^{\frac{k -s}{d} p} \int_{t}^{\min\big\{1, \big(\frac{\text{dist}(\Omega_0, \partial \Omega)}{k} \big)^d \big\}} u^{p/p^\ast} f^\ast(u)^p \frac{du}{u} \frac{dt}{t} &\lesssim \nonumber \\
		 & \hspace{-7cm}(\text{dist}(\Omega_0, \partial \Omega))^{-k/p}   \int_0^{\min\{1, \frac{\text{dist}(\Omega_0, \partial \Omega)}{k}\}} t^{-s p} \omega_k(f,t)_{p;\Omega}^p \frac{dt}{t}. \label{ThmSob15}
		\end{align}
		If $s \to k-$ then we can apply Fubini's theorem to derive
		\begin{align}
		 \int_0^{\min\big\{1, \big(\frac{\text{dist}(\Omega_0, \partial \Omega)}{k} \big)^d\big\} } t^{\frac{k -s}{d} p} \int_{t}^{\min\big\{1, \big(\frac{\text{dist}(\Omega_0, \partial \Omega)}{k} \big)^d \big\}} u^{p/p^\ast} f^\ast(u)^p \frac{du}{u} \frac{dt}{t}
\nonumber\\
		& \hspace{-10cm}\asymp (k - s)^{-1}  \int_0^{\min\big\{1, \big(\frac{\text{dist}(\Omega_0, \partial \Omega)}{k} \big)^d\big\} }  u^{p/p^\ast} f^\ast(u)^p u^{\frac{k - s}{d} p} \frac{du}{u} = (k - s)^{-1}  \|f\|_{ L_{\frac{d p}{d - s p}, p}(\Omega)}^p	\label{ThmSob25}
		\end{align}
		where we have used \eqref{ThmSob1*OmegaHomExtrapolNew} in the last step. The proof is concluded from (\ref{ThmSob15}) and (\ref{ThmSob25}).		
		
\end{proof}

So far we have investigated characterizations of the Sobolev inequality (\ref{SobolevClassical*}) in terms of estimates involving only rearrangements and moduli of smoothness. Now we provide different  characterizations in terms of estimates for rearrangements and derivatives. 

	\begin{thm}\label{ThmSobRear}
Let $1 < p < \infty, 0 < q \leq \infty$ and $k \in \mathbb{N}$. Assume that $k < d/p$. Let $p^\ast = d p /(d - k p)$ and $1/\alpha = 1 - k/d$. The following statements are equivalent:
	\begin{enumerate}[\upshape(i)]
		\item
		\begin{equation*}
			W^k_p(\Omega) \hookrightarrow L_{p^\ast,q}(\Omega),
		\end{equation*}
		\item for $ f \in W^k_1(\Omega)$ and $t \in (0,1)$, we have
		\begin{equation}\label{SharpKolyada2*'}
	\left(\int_t^1 \left(v^{1/p-k/d - 1/\alpha} \int_0^v u^{1/\alpha} f^\ast(u) \frac{du}{u} \right)^q \frac{dv}{v} \right)^{1/q} \lesssim \sum_{l=0}^k \left(\int_t^1 ( |\nabla^l f|^{\ast \ast}(v))^p dv \right)^{1/p},
\end{equation}
			\item if $r \to p-$ then
			\begin{equation*}
			W^k L_{r,p}(\Omega) \hookrightarrow L_{r^\ast,q}(\Omega), \quad r^\ast = \frac{d r}{d - k r},
			\end{equation*}
			 with norm $\mathcal{O}(1)$, i.e., there exists $C > 0$, which is independent of $r$, such that
		\begin{equation}\label{SharpFrankeJawerth2*'}
			\|f\|_{L_{r^\ast,q}(\Omega)} \leq C \|f\|_{W^k L_{r,p}(\Omega)}, \quad f \in W^k L_{r,p}(\Omega),
		\end{equation}
		\item
		\begin{equation*}
		q \geq p.
		\end{equation*}
	\end{enumerate}

\end{thm}

\begin{thm}\label{ThmSobRearRd}
Let $1 < p < \infty, 0 < q \leq \infty$ and $k \in \mathbb{N}$. Assume that $k < d/p$. Let $p^\ast = d p /(d - k p)$ and $1/\alpha = 1 - k/d$. The following statements are equivalent:
	\begin{enumerate}[\upshape(i)]
		\item
		\begin{equation*}
			W^k_p(\R^d) \hookrightarrow L_{p^\ast,q}(\R^d),
		\end{equation*}
		\item for $ f \in W^k_1(\R^d) + W^k_p(\R^d)$ and $t \in (0,\infty)$, we have
		\begin{equation}\label{SharpKolyada2*'Rd}
	\left(\int_t^\infty \left(v^{1/p-k/d - 1/\alpha} \int_0^v u^{1/\alpha} f^\ast(u) \frac{du}{u} \right)^q \frac{dv}{v} \right)^{1/q} \lesssim \sum_{l=0}^k \left(\int_t^\infty ( |\nabla^l f|^{\ast \ast}(v))^p dv \right)^{1/p},
\end{equation}
			\item if $r \to p-$ then
			\begin{equation*}
			W^k L_{r,p}(\R^d) \hookrightarrow L_{r^\ast,q}(\R^d), \quad r^\ast = \frac{d r}{d - k r},
			\end{equation*}
			 with norm $\mathcal{O}(1)$, i.e., there exists $C > 0$, which is independent of $r$, such that
		\begin{equation}\label{SharpFrankeJawerth2*'Rd}
			\|f\|_{L_{r^\ast,q}(\R^d)} \leq C \|f\|_{W^k L_{r,p}(\R^d)}, \quad f \in W^k L_{r,p}(\R^d),
		\end{equation}
		\item
		\begin{equation*}
		q \geq p.
		\end{equation*}
	\end{enumerate}

\end{thm}

\begin{thm}\label{ThmSobRearRdHom}
Let $1 < p < \infty, 0 < q \leq \infty$ and $k \in \mathbb{N}$. Assume that $k < d/p$. Let $p^\ast = d p /(d - k p)$ and $1/\alpha = 1 - k/d$. The following statements are equivalent:
	\begin{enumerate}[\upshape(i)]
		\item
		\begin{equation*}
			(V^k_p(\R^d))_0 \hookrightarrow L_{p^\ast,q}(\R^d),
		\end{equation*}
		\item for $ f \in  (V^k_1(\R^d))_0+(V^k_p(\R^d))_0$ and $t \in (0,\infty)$, we have
		\begin{equation}\label{SharpKolyada2*'RdHom}
	\left(\int_t^\infty \left(v^{1/p-k/d - 1/\alpha} \int_0^v u^{1/\alpha} f^\ast(u) \frac{du}{u} \right)^q \frac{dv}{v} \right)^{1/q} \lesssim  \left(\int_t^\infty ( |\nabla^k f|^{\ast \ast}(v))^p dv \right)^{1/p},
\end{equation}
			\item if $r \to p-$ then
			\begin{equation*}
			(V^k L_{r,p}(\R^d))_0 \hookrightarrow L_{r^\ast,q}(\R^d), \quad r^\ast = \frac{d r}{d - k r},
			\end{equation*}
			 with norm $\mathcal{O}(1)$, i.e., there exists $C > 0$, which is independent of $r$, such that
		\begin{equation}\label{SharpFrankeJawerth2*'RdHom}
			\|f\|_{L_{r^\ast,q}(\R^d)} \leq C \||\nabla^k f|\|_{L_{r,p}(\R^d)}, \quad f \in (V^k L_{r,p}(\R^d))_0,
		\end{equation}
		\item
		\begin{equation*}
		q \geq p.
		\end{equation*}
	\end{enumerate}

\end{thm}

\begin{proof}[Proof of Theorem \ref{ThmSobRear}]
	The equivalence between (i) and (iv) is well known.
	
	(i) $\Longrightarrow$ (ii): It follows from
	\begin{equation*}
	W^k_p(\Omega) \hookrightarrow L_{p^\ast,q}(\Omega) \quad \text{and} \quad W^k_1(\Omega) \hookrightarrow L_{\alpha,1}(\Omega)
	\end{equation*}
	(see (\ref{SobolevClassical*})) that
	\begin{equation}\label{ThmSobRear1}
		K(t,f; L_{\alpha,1}(\Omega), L_{p^\ast,q}(\Omega) ) \lesssim K(t, f; W^k_1(\Omega), W^k_p(\Omega)) \quad \text{for} \quad f \in W^k_1(\Omega).
	\end{equation}
	By (\ref{LemmaKFunctLorentz}), (\ref{LemmaKFunctLorentzSobolev}), and Lemma \ref{vspom1}, we have
	\begin{equation*}
		K(t,f; L_{\alpha,1}(\Omega), L_{p^\ast,q}(\Omega) ) \gtrsim t \left(\int_{t^{p'}}^1 \left( v^{1/p-k/d - 1/\alpha} \int_0^v u^{1/\alpha -1} f^\ast(u) du\right)^q \frac{dv}{v} \right)^{1/q}
	\end{equation*}
	and
	\begin{equation*}
		K(t, f; W^k_1(\Omega), W^k_p(\Omega)) \asymp t \sum_{l=0}^k \left(\int_{t^{p'}}^1  (|\nabla^l f|^{\ast \ast}(v))^p  dv \right)^{1/p}.
	\end{equation*}
	Inserting these estimates into (\ref{ThmSobRear1}) we arrive at (\ref{SharpKolyada2*'}).
	
	(ii) $\Longrightarrow$ (i): Applying monotonicity properties and (\ref{SharpKolyada2*'}),
	\begin{align*}
			\|f\|_{L_{p^\ast,q}(\Omega)} & = \left(\int_0^1 (v^{1/p - k/d} f^\ast(v))^q \frac{dv}{v} \right)^{1/q} \\
	& \lesssim \left(\int_0^1 \left(v^{1/p-k/d - 1/\alpha} \int_0^v u^{1/\alpha} f^\ast(u) \frac{du}{u} \right)^q \frac{dv}{v} \right)^{1/q} \\
			 & \lesssim \sum_{l=0}^k \left(\int_0^1 ( |\nabla^l f|^{\ast \ast}(v))^p dv \right)^{1/p} \lesssim \|f\|_{W^k_p(\Omega)},
	\end{align*}
	where the last estimate is an immediate consequence of the Hardy's inequality (\ref{HardyIneq*}) (noting that $p > 1$).
	
	(ii) $\Longrightarrow$ (iii): Let (ii) hold. Then we can assume that $q \geq p$ (because (ii) $\iff$ (iv)) and $r > 1$ (because $r \to p-$ and $p > 1$). By Fubini's theorem and (\ref{HardyIneq*}), we have for any $1<r<p$,
	\begin{align}
		I & := \left(\int_0^1 t^{(1/r -1/p)p} \left(\int_t^1 \left(v^{1/p-k/d - 1/\alpha} \int_0^v u^{1/\alpha} f^\ast(u) \frac{du}{u} \right)^q \frac{dv}{v} \right)^{p/q} \frac{dt}{t} \right)^{1/p} \nonumber \\
		& \lesssim \sum_{l=0}^k \left(\int_0^1 t^{(1/r -1/p)p} \int_t^1 ( |\nabla^l f|^{\ast \ast}(v))^p dv \frac{dt}{t} \right)^{1/p} \nonumber  \\
		& \asymp (1/r -1/p)^{-1/p} \sum_{l=0}^k \left(\int_0^1 (v^{1/r} |\nabla^l f|^{\ast \ast}(v))^p \frac{dv}{v} \right)^{1/p} \nonumber  \\
		& \lesssim (1/r -1/p)^{-1/p} (1 -1/r)^{-1} \sum_{l=0}^k \left(\int_0^1 (v^{1/r} |\nabla^l f|^\ast(v))^p \frac{dv}{v} \right)^{1/p} \nonumber  \\
		& \lesssim (1/r -1/p)^{-1/p} \|f\|_{W^k L_{r,p}(\Omega)}. \label{ThmSobRear2}
	\end{align}
	On the other hand, applying Minkowski's inequality we derive
	\begin{align*}
		I & \geq \left(\int_0^1 \left(v^{1/p-k/d-1/\alpha}  \int_0^v u^{1/\alpha} f^\ast(u) \frac{du}{u} \right)^q \left(\int_0^v t^{(1/r-1/p) p} \frac{dt}{t} \right)^{q/p} \frac{dv}{v} \right)^{1/q} \\
		& \asymp (1/r -1/p)^{-1/p} \left(\int_0^1 \left(v^{1/r -k/d -1/\alpha}  \int_0^v u^{1/\alpha} f^\ast(u) \frac{du}{u}  \right)^q \frac{dv}{v} \right)^{1/q} \\
		& \gtrsim  (1/r -1/p)^{-1/p} \left(\int_0^1 (v^{1/r -k/d} f^\ast(v))^q \frac{dv}{v} \right)^{1/q} =  (1/r -1/p)^{-1/p}  \|f\|_{L_{r^\ast,q}(\Omega)},
	\end{align*}
where $r^\ast = \frac{d r}{d - k r}$.
	Inserting this estimate into (\ref{ThmSobRear2}), we obtain
	\begin{equation*}
	\|f\|_{L_{r^\ast,q}(\Omega)} \lesssim  \|f\|_{W^k L_{r,p}(\Omega)}.
	\end{equation*}
	
	(iii) $\Longrightarrow$ (i): The embedding (i) follows from (\ref{SharpFrankeJawerth2*'}) by taking limits as $r \to p-$ and noting that $r^\ast \to p^\ast-$ as $r \to p-$.
	
\end{proof}

The proof of Theorem \ref{ThmSobRearRd} goes the same lines as the one of  Theorem \ref{ThmSobRear} with
suitable changes.

\begin{proof}[Proof of Theorem \ref{ThmSobRearRdHom}]
	The equivalence between (i) and (iv) is well known.
	
	(i) $\Longrightarrow$ (ii): It follows from
	\begin{equation*}
	(V^k_p(\R^d))_0 \hookrightarrow L_{p^\ast,q}(\R^d) \quad \text{and} \quad (V^k_1(\R^d))_0 \hookrightarrow L_{\alpha,1}(\R^d)
	\end{equation*}
 that
	\begin{equation}\label{ThmSobRear1RdHom}
		K(t,f; L_{\alpha,1}(\R^d), L_{p^\ast,q}(\R^d) ) \lesssim K(t, f; (V^k_1(\R^d))_0, (V^k_p(\R^d))_0)
	\end{equation}
	for  $f \in  (V^k_1(\R^d))_0+(V^k_p(\R^d))_0$. By (\ref{LemmaKFunctLorentz}), (\ref{LemmaKFunctLorentzSobolev}), and Lemma \ref{vspom1}, we have
	\begin{equation*}
		K(t,f; L_{\alpha,1}(\R^d), L_{p^\ast,q}(\R^d) ) \asymp t \left(\int_{t^{p'}}^\infty \left( v^{1/p-k/d - 1/\alpha} \int_0^v u^{1/\alpha -1} f^\ast(u) du\right)^q \frac{dv}{v} \right)^{1/q}
	\end{equation*}
	and
	{
	\begin{equation*}
		K(t, f; (V^k_1(\R^d))_0, (V^k_p(\R^d))_0) \asymp t  \left(\int_{t^{p'}}^\infty  (|\nabla^k f|^{\ast \ast}(v))^p  dv \right)^{1/p}.
	\end{equation*}}
	Inserting these estimates into (\ref{ThmSobRear1RdHom}) we arrive at (\ref{SharpKolyada2*'RdHom}).
	
	The implication (ii) $\Longrightarrow$ (iii) $\Longrightarrow$ (i) can be shown
as in the proof of  Theorem \ref{ThmSobRear}.
	
\end{proof}

\subsection{Sobolev embeddings with fixed target space}

The classical Sobolev inequality \eqref{121=} can be rewritten as
\begin{equation}\label{SobolevClassical**}
	(W^k_{\frac{d p}{k p + d}}(\Omega))_0 \hookrightarrow L_{p}(\Omega), \quad d > k, \quad \frac{d}{d-k} \leq p < \infty.
\end{equation}
In Section \ref{Section4.1} we have established characterizations of the sharp version of (\ref{SobolevClassical**}), which are obtained by using the finer scale of Lorentz spaces as target spaces. Alternatively, (\ref{SobolevClassical**}) can be strengthened by enlarging the domain space. More precisely, the following inequality holds (see \cite{Alvino} and \cite{Talenti})
 \begin{equation}\label{SobolevClassical***}
		(W^k L_{\frac{d p}{k p + d},p}(\Omega))_0 \hookrightarrow L_p(\Omega),  \quad d > k, \quad \frac{d}{d-k} < p < \infty.
		\end{equation}
Note that $L_{\frac{d p}{k p + d}}(\Omega) \subsetneq L_{\frac{d p}{k p + d},p}(\Omega)$.

Our next goal is to characterize (\ref{SobolevClassical***}) by using 
  $L_p$-moduli of smoothness (Theorems \ref{ThmStein2}--\ref{ThmStein2RdHom}) and rearrangements of functions in $L_p(\Omega)$ (Theorems \ref{ThmStein3}--\ref{ThmStein3RdHom}).
We again present separate results in the inhomogeneous case (Theorems \ref{ThmStein2},  \ref{ThmStein2Rd}, \ref{ThmStein3} and \ref{ThmStein3Rd}) and the homogeneous case (Theorem \ref{ThmStein2RdHom} and \ref{ThmStein3RdHom}).

\begin{thm}\label{ThmStein2}
Let $k \in \mathbb{N}, d > k,  \frac{d}{d-k} < p < \infty$ and $0 < q \leq \infty$. The following statements are equivalent:
	\begin{enumerate}[\upshape(i)]
		\item
		\begin{equation*}
		W^k L_{\frac{d p}{k p + d},q}(\Omega) \hookrightarrow L_p(\Omega),
		\end{equation*}
		\item for $ f \in W^k L_{\frac{d p}{k p + d},q}(\Omega)$ and $t \in (0,1)$, we have
		\begin{equation}\label{SharpKolyada2}
	t^k \|f\|_{L_p(\Omega)} + \omega_k(f,t)_{p;\Omega} \lesssim \sum_{l=0}^k \bigg[ \left(\int_0^{t^d} (u^{k/d + 1/p} |\nabla^l f|^\ast(u))^q  \frac{du}{u} \right)^{1/q} + t^k \left( \int_{t^d}^1 (|\nabla^l f|^\ast(u))^p du\right)^{1/p} \bigg],
\end{equation}
			\item if $r \to \frac{d p}{k p + d}+$ then
			\begin{equation*}
			W^k L_{r,q}(\Omega) \hookrightarrow B^{k - d/r + d/p}_{p,q}(\Omega)
			\end{equation*}
			 with norm $\mathcal{O}((r-\frac{d p}{k p + d})^{-1/q})$, i.e., for any $m \in \N$ there exists $C > 0$, which is independent of $r$, such that
		\begin{equation}\label{SharpFrankeJawerth2}
			\|f\|_{B^{k - d/r + d/p}_{p,q}(\Omega),m} \leq C \Big(r - \frac{d p}{k p + d}\Big)^{-1/q} \|f\|_{W^k L_{r,q}(\Omega)},
		\end{equation}
		\item
		\begin{equation*}
		q \leq p.
		\end{equation*}
	\end{enumerate}
\end{thm}

\begin{thm}\label{ThmStein2Rd}
Let $k \in \mathbb{N}, d > k,  \frac{d}{d-k} < p < \infty$ and $0 < q \leq \infty$. The following statements are equivalent:
	\begin{enumerate}[\upshape(i)]
		\item
		\begin{equation*}
		W^k L_{\frac{d p}{k p + d},q}(\R^d) \hookrightarrow L_p(\R^d),
		\end{equation*}
		\item for $ f \in W^k L_{\frac{d p}{k p + d},q}(\R^d) + W^k_p(\R^d)$ and $t \in (0,\infty)$, we have
		\begin{eqnarray}\label{SharpKolyada2Rd}\nonumber \qquad \;
	\min\{1,t^k\} \|f\|_{L_p(\R^d)} + \omega_k(f,t)_{p;\R^d} &\lesssim& \sum_{l=0}^k \bigg[ \left(\int_0^{t^d} (u^{k/d + 1/p} |\nabla^l f|^\ast(u))^q  \frac{du}{u} \right)^{1/q}\\ &&\qquad+ t^k \left( \int_{t^d}^\infty (|\nabla^l f|^\ast(u))^p du\right)^{1/p} \bigg],
\end{eqnarray}
			\item if $r \to \frac{d p}{k p + d}+$ then
						\begin{equation*}
			W^k L_{r,q}(\R^d) \hookrightarrow B^{k - d/r + d/p}_{p,q}(\R^d)
			\end{equation*}
			 with norm $\mathcal{O}((r-\frac{d p}{k p + d})^{-1/q})$, i.e., for any $m \in \N$ there exists $C > 0$, which is independent of $r$, such that
		\begin{equation}\label{SharpFrankeJawerth2Rd}
			\|f\|_{B^{k - d/r + d/p}_{p,q}(\R^d),m} \leq C \Big(r - \frac{d p}{k p + d}\Big)^{-1/q} \|f\|_{W^k L_{r,q}(\R^d)},
		\end{equation}
		\item
		\begin{equation*}
		q \leq p.
		\end{equation*}
	\end{enumerate}
\end{thm}

\begin{thm}\label{ThmStein2RdHom}
Let $k \in \mathbb{N}, d > k,  \frac{d}{d-k} < p < \infty$ and $0 < q \leq \infty$. The following statements are equivalent:
	\begin{enumerate}[\upshape(i)]
		\item
		\begin{equation*}
		(V^k L_{\frac{d p}{k p + d},q}(\R^d))_0 \hookrightarrow L_p(\R^d),
		\end{equation*}
		\item for $ f \in (V^k L_{\frac{d p}{k p + d},q}(\R^d))_0 +  (V^k_p(\R^d))_0$ and $t \in (0,\infty)$, we have
		\begin{equation}\label{SharpKolyada2RdHom}
	 \omega_k(f,t)_{p;\R^d} \lesssim \left(\int_0^{t^d} (u^{k/d + 1/p} |\nabla^k f|^\ast(u))^q  \frac{du}{u} \right)^{1/q} + t^k \left( \int_{t^d}^\infty (|\nabla^k f|^\ast(u))^p du\right)^{1/p} ,
\end{equation}
			\item if $r \to \frac{d p}{k p + d}+$ then
			\begin{equation*}
			(V^k L_{r,q}(\R^d))_0 \hookrightarrow \dot{B}^{k - d/r + d/p}_{p,q}(\R^d)
			\end{equation*}
			 with norm $\mathcal{O}((r-\frac{d p}{k p + d})^{-1/q})$, i.e., for any $m \in \N$ there exists $C > 0$, which is independent of $r$, such that
		\begin{equation}\label{SharpFrankeJawerth2RdHom}
			\|f\|_{\dot{B}^{k - d/r + d/p}_{p,q}(\R^d),m} \leq C \Big(r - \frac{d p}{k p + d}\Big)^{-1/q} \||\nabla^k f|\|_{L_{r,q}(\R^d)}
		\end{equation}
		for $f \in (V^k L_{r,q}(\R^d))_0$,
		\item
		\begin{equation*}
		q \leq p.
		\end{equation*}
	\end{enumerate}
\end{thm}

\begin{rem}
	(i) Inequality (\ref{SharpKolyada2}) with $q=p$ has been recently obtained by Gogatishvili, Neves and Opic \cite[Theorem 3.2]{GogatishviliNevesOpic}. The corresponding inequalities for functions on $\R^d$ in the homogeneous/inhomogeneous setting (see \eqref{SharpKolyada2Rd} and \eqref{SharpKolyada2RdHom}) are new.

(ii) The Jawerth-Franke type embedding for Lorentz-Sobolev spaces was recently obtained by Seeger and Trebels \cite[Theorem 1.2]{SeegerTrebels}. Let $1 < r < p < \infty, 0 < q \leq \infty$ and $k > d \Big(\frac{1}{r} - \frac{1}{p} \Big)$. Then
\begin{equation}\label{SeegerTrebels}
	W^k L_{r,q}(\R^d) \hookrightarrow B^{k - d/r + d/p}_{p,q}(\R^d).
\end{equation}
Hence,  (\ref{SharpFrankeJawerth2Rd}) is the sharp version of embedding (\ref{SeegerTrebels}) in the spirit of Br\'ezis-Bourgain-Mironescu \cite{BourgainBrezisMironescu} (see also Remark \ref{RemEta}(i) below for further details on this point.)

(iii) {The limiting case $p=\infty$ in Theorem \ref{ThmStein2} will be settled in Theorem \ref{ThmStein} below.}
\end{rem}

\begin{proof}[Proof of Theorem \ref{ThmStein2}]
	(i) $\Longrightarrow$ (ii): The proof of this implication follows essentially the lines of \cite[Theorem 3.2]{GogatishviliNevesOpic}. 
 Due to Lemma A.2(iv)
and (\ref{LemmaKFunctLorentz}), we obtain
	\begin{align*}
K(t^k, f ; L_p(\Omega), W^k_p(\Omega)) &\lesssim	K(t, f; W^k L_{\frac{d p}{k p + d},q}(\Omega), W^k_p(\Omega))  \\
	 & \asymp \sum_{l=0}^k K(t, |\nabla^l f| ; L_{\frac{d p}{k p + d},q}(\Omega), L_p(\Omega))  \\
		& \hspace{-5cm} \asymp \sum_{l=0}^k \bigg[ \left(\int_0^{t^{d/k}} (u^{k/d + 1/p} |\nabla^l f|^\ast(u))^q \frac{du}{u} \right)^{1/q} + t \left(\int_{t^{d/k}}^1 (|\nabla^l f|^\ast(u))^p du \right)^{1/p} \bigg],
	\end{align*}
	which gives 
 (\ref{SharpKolyada2}).
	
	(ii) $\Longrightarrow$ (iii):  Since $r \to \frac{d p}{k p + d}+$, we may assume that $k - d/r + d/p < 1$. Further, it suffices to show (\ref{SharpFrankeJawerth2}) with $m=1$ (see (\ref{JacksonInequal})).
	
	  By \eqref{DefHolZyg} and (\ref{MarchaudInequal}), we have
	\begin{align}
		\vertiii{f}^q_{B^{k - d/r + d/p}_{p,q}(\Omega), 1} & =  \int_0^1 t^{-(k-d/r + d/p) q} \omega_1(f, t)_{p;\Omega}^q \frac{dt}{t} \nonumber \\
		& \hspace{-2cm} \lesssim \int_0^1 t^{(1-k+d/r-d/p) q} \left(\int_t^\infty \frac{\omega_k(f,u)_{p;\Omega}}{u} \frac{du}{u} \right)^q \frac{dt}{t} \nonumber \\
		& \hspace{-2cm} \lesssim \int_0^1 t^{(1-k+d/r-d/p) q} \left(\int_t^1 \frac{\omega_k(f,u)_{p;\Omega}}{u} \frac{du}{u} \right)^q \frac{dt}{t}  + \|f\|_{L_p(\Omega)}^q  \label{ThmStein2.2prev}.
		\end{align}
		Next we show that
		\begin{align}
			  \left(\int_0^1 t^{(1-k+d/r-d/p) q} \left(\int_t^1 \frac{\omega_k(f,u)_{p;\Omega}}{u} \frac{du}{u} \right)^q \frac{dt}{t}  \right)^{1/q} \nonumber\\
			  & \hspace{-5cm} \leq C \left(\int_0^1 t^{-(k-d/r + d/p) q} \omega_k(f,t)_{p;\Omega}^q \frac{dt}{t} \right)^{1/q}, \label{ThmStein2.2prev1}
		\end{align}
		where $C > 0$ does not depend on $r$. Indeed, if $q \geq 1$ then (\ref{ThmStein2.2prev1}) is an immediate consequence of the Hardy's inequality (\ref{HardyIneq}) where $C \asymp 1 - k +\frac{d}{r} - \frac{d}{p}$, which is uniformly bounded as $r \to \frac{d p}{k p + d}+$. Assume now $q < 1$. By monotonicity properties and Fubini's theorem, we have
		\begin{align*}
			  \left(\int_0^1 t^{(1-k+d/r-d/p) q} \left(\int_t^1 \frac{\omega_k(f,u)_{p;\Omega}}{u} \frac{du}{u} \right)^q \frac{dt}{t}  \right)^{1/q} \\
			  & \hspace{-7cm}\asymp \left(\sum_{i=0}^\infty 2^{-i (1-k+d/r-d/p) q} \left(\sum_{j=0}^i \frac{\omega_k(f,2^{-j})_{p;\Omega}}{2^{-j}}  \right)^q \right)^{1/q} \\
			  & \hspace{-7cm} \leq \left( \sum_{i=0}^\infty 2^{-i (1-k+d/r-d/p) q}  \sum_{j=0}^i \left( \frac{\omega_k(f,2^{-j})_{p;\Omega}}{2^{-j}} \right)^q \right)^{1/q} \\
			  & \hspace{-8.5cm} \asymp \left(\sum_{j=0}^\infty 2^{-j (-k+d/r-d/p) q} \omega_k(f,2^{-j})_{p;\Omega}^q    \right)^{1/q}
\asymp \left(\int_0^1 t^{-(k-d/r + d/p) q} \omega_k(f,t)_{p;\Omega}^q \frac{dt}{t} \right)^{1/q},		\end{align*}
		 which completes the proof of (\ref{ThmStein2.2prev1}).
		
		Inserting the estimate (\ref{ThmStein2.2prev1}) into (\ref{ThmStein2.2prev}) and invoking \eqref{SharpKolyada2}, we establish
		\begin{align}
		\vertiii{f}^q_{B^{k - d/r + d/p}_{p,q}(\Omega), 1} &\lesssim \int_0^1 t^{-(k-d/r + d/p) q} \omega_k(f,t)_{p;\Omega}^q \frac{dt}{t} + \|f\|_{L_p(\Omega)}^q \nonumber \\
		&\lesssim \sum_{l=0}^k  \int_0^1 t^{-(k/d-1/r + 1/p) q} \int_0^{t} (u^{k/d + 1/p} |\nabla^l f|^\ast(u))^q  \frac{du}{u} \frac{dt}{t} \nonumber  \\
		& \hspace{0.35cm}+ \sum_{l=0}^k \int_0^1 t^{(1/r - 1/p) q} \left( \int_{t}^1 (|\nabla^l f|^\ast(u))^p du \right)^{q/p} \frac{dt}{t} + \|f\|_{L_p(\Omega)}^q \nonumber \\
		& = I + II +  \|f\|_{L_p(\Omega)}^q, \label{ThmStein2.2}
	\end{align}
	where
	\begin{equation*}
		I:= \sum_{l=0}^k  \int_0^1 t^{-(k/d-1/r + 1/p) q} \int_0^{t} (u^{k/d + 1/p} |\nabla^l f|^\ast(u))^q  \frac{du}{u} \frac{dt}{t}
	\end{equation*}
	and
	\begin{equation*}
		II:= \sum_{l=0}^k \int_0^1 t^{(1/r - 1/p) q} \left( \int_{t}^1 (|\nabla^l f|^\ast(u))^p du \right)^{q/p} \frac{dt}{t}.
	\end{equation*}
	
	We estimate $I$ as follows:
		\begin{align}
		I & = \sum_{l=0}^k \int_0^1 (u^{k/d + 1/p} |\nabla^l f|^\ast(u))^q \int_u^1  t^{-(k/d-1/r + 1/p) q} \frac{dt}{t} \frac{du}{u}  \nonumber \\
		& \lesssim (k/d-1/r + 1/p)^{-1} \sum_{l=0}^k \int_0^1 (u^{1/r} |\nabla^l f|^\ast(u))^q \frac{du}{u} \nonumber \\
		& \asymp (k/d-1/r + 1/p)^{-1} \|f\|_{W^k L_{r,q} (\Omega)}^q. \label{ThmStein2.3}
	\end{align}
	
	To deal with $II$ we will distinguish two cases. Suppose first that $q \geq p$. Note that we may assume without loss of generality that $r < p$. In virtue of Lemma \ref{LemmaHardySharp}, we derive
	\begin{equation*}
		II \lesssim \left(\frac{1}{r} - \frac{1}{p} \right)^{-q/p} \sum_{l=0}^k \int_0^1 t^{q/r} (|\nabla^l f|^\ast(t))^q \frac{dt}{t} \asymp \left(\frac{1}{r} - \frac{1}{p} \right)^{-q/p}  \|f\|^q_{W^k L_{r,q}(\Omega)}.
	\end{equation*}
	Further, since
	\begin{equation}\label{ThmStein2.4}
r \to \frac{d p}{k p + d}+  \iff \frac{1}{r} - \frac{1}{p} \to \frac{k}{d}-
	\end{equation}
	 we derive
	\begin{equation}\label{ThmStein2.5}
		II \lesssim \|f\|^q_{W^k L_{r,q}(\Omega)}, \quad q \geq p.
	\end{equation}
	
	Secondly, let $q < p$. Applying monotonicity properties of rearrangements, we obtain
	\begin{align*}
		II &\asymp \sum_{l=0}^k \sum_{i= 0}^\infty 2^{-i(1/r -1/p) q} \left(\sum_{j = 0}^i (|\nabla^l f|^\ast(2^{-j}))^p 2^{-j} \right)^{q/p}  \\
		& \leq \sum_{l=0}^k \sum_{i= 0}^\infty 2^{-i(1/r -1/p) q} \sum_{j=0}^i (|\nabla^l f|^\ast(2^{-j}))^q 2^{-j q/p} \\
		& \lesssim \left(\frac{1}{r} - \frac{1}{p} \right)^{-1} \sum_{l=0}^k \sum_{j = 0}^\infty (|\nabla^l f|^\ast(2^{-j}))^q 2^{-j q/r}
\asymp \left(\frac{1}{r} - \frac{1}{p} \right)^{-1} \sum_{l=0}^k \int_0^1 t^{q/r} (|\nabla^l f|^\ast(t))^q \frac{dt}{t}.
	\end{align*}
	This implies (see (\ref{ThmStein2.4}))
	\begin{equation}\label{ThmStein2.6}
		II \lesssim  \|f\|^q_{W^k L_{r,q}(\Omega)}, \quad q < p.
	\end{equation}
	
	On the other hand, according to \eqref{SharpKolyada2}, we have
	\begin{equation}\label{ThmStein2.6++}
		\|f\|_{L_p(\Omega)} \lesssim \|f\|_{W^k L_{\frac{d p}{k p + d},q}(\Omega)}  \leq \|f\|_{W^k L_{r,q}(\Omega)},
	\end{equation}
	where the last inequality follows from the fact that $r >  \frac{d p}{k p + d}$ and $\Omega$ is bounded.
	
	Combining \eqref{DefBesovInh},
 (\ref{ThmStein2.2}), (\ref{ThmStein2.3}), (\ref{ThmStein2.5}), (\ref{ThmStein2.6}), and \eqref{ThmStein2.6++}, we obtain 
	\begin{align*}
		\|f\|^q_{B^{k - d/r + d/p}_{p,q}(\Omega),1}  &\asymp \left(\frac{k}{d}-\frac{1}{r} + \frac{1}{p}\right)^{-1} \|f\|^q_{L_p(\Omega)} +  \vertiii{f}^q_{B^{k - d/r + d/p}_{p,q}(\Omega),1} \\
		& \lesssim \left(1 + \left(\frac{k}{d}-\frac{1}{r} + \frac{1}{p}\right)^{-1} \right) (\|f\|^q_{W^k L_{r,q}(\Omega)} + \|f\|_{L_p(\Omega)}^q) \\
		& \asymp \left(r - \frac{d p}{k p + d}\right)^{-1} ( \|f\|^q_{W^k L_{r,q}(\Omega)} + \|f\|_{L_p(\Omega)}^q) \\
		& \lesssim  \left(r - \frac{d p}{k p + d}\right)^{-1} \|f\|^q_{W^k L_{r,q}(\Omega)},
	\end{align*}
i.e.,  (\ref{SharpFrankeJawerth2}) follows.
	
	(iii) $\Longrightarrow$ (i): By \eqref{DefBesovInh}, we have
	\begin{equation*}
		\|f\|_{L_p(\Omega)} \lesssim \left(r - \frac{d p}{k p + d}\right)^{\frac{1}{q}}
		\|f\|_{B^{k - d/r + d/p}_{p,q}(\Omega),m}, \quad r \to \frac{d p}{k p + d}+,
	\end{equation*}
	where we have also used (\ref{ThmStein2.4}). Combining this inequality with  (\ref{SharpFrankeJawerth2}), we arrive at
	\begin{equation*}	
		\|f\|_{L_p(\Omega)} \leq C  \|f\|_{W^k L_{r,q}(\Omega)},  \quad r \to \frac{d p}{k p + d}+,
	\end{equation*}
	where $C$ is a positive constant which does not depend on $r$. Then, in virtue of the monotone convergence theorem, we derive $W^k L_{\frac{d p}{ k p + d},q}(\Omega) \hookrightarrow L_p(\Omega)$.
	

	The equivalence (i) $\iff$ (iv) is well known (see the methodology developed in  \cite{EdmundsKermanPick}).
	
\end{proof}

\begin{proof}[Proof of Theorem \ref{ThmStein2Rd}]
The proof of the part 	(i) $\Longrightarrow$ (ii) is similar to the one in the previous theorem with the help of Lemma A.2(ii).
	
	(ii) $\Longrightarrow$ (iii):
Following the same steps as in (\ref{ThmStein2.2prev})--(\ref{ThmStein2.6}),
we establish

\begin{align}\label{vspom12}
		\vertiii{f}^q_{B^{k - d/r + d/p}_{p,q}(\R^d), 1} \lesssim
 \|f\|_{L_p(\R^d)}^q + \|f\|^q_{W^k L_{r,q}(\R^d)}
+
(k/d-1/r + 1/p)^{-1} \|f\|_{W^k L_{r,q} (\R^d)}^q.
	\end{align}

Taking limits as $t \to \infty$ in \eqref{SharpKolyada2Rd}, we have
	\begin{equation*}
		\|f\|_{L_p(\R^d)} \lesssim \|f\|_{W^k L_{\frac{d p}{k p + d},q}(\R^d)}, \quad f \in W^k L_{\frac{d p}{k p + d},q}(\R^d),
	\end{equation*}
	and thus
	\begin{equation}\label{gamma}
	\|f\|_{L_p(\R^d)} \lesssim \|f\|_{W^k L_{\frac{d p}{k p + d},q}(\R^d) + W^k_p(\R^d)}.
	\end{equation}
	 Using the monotonicity properties of the real interpolation scale, more precisely,
	\begin{equation*}
		\|f\|_{A_0 + A_1} \leq (\theta (1-\theta) q)^{1/q}\|f\|_{(A_0, A_1)_{\theta,q}}, \quad \theta \in (0,1), \quad 1 \leq q \leq \infty,
	\end{equation*}
	(cf. \cite[p. 19]{JawerthMilman}) together with \eqref{gamma}, we obtain
	\begin{equation}\label{ThmStein2.6++Rd}
		\|f\|_{L_p(\R^d)} \lesssim \|f\|_{W^k L_{\frac{d p}{k p + d},q}(\R^d) + W^k_p(\R^d)} \lesssim \theta^{1/q} \|f\|_{(W^k L_{\frac{d p}{k p + d},q}(\R^d), W^k_p(\R^d))_{\theta, q}}
	\end{equation}
	for the specific parameter $\theta = 1-\frac{d \big(\frac{1}{r} - \frac{1}{p} \big)}{k} \in (0,1)$ (recall that $\frac{d p}{k p + d} < r < p$). Furthermore, we claim that
	\begin{equation}\label{gamma2}
		\theta^{1/q} \|f\|_{(W^k L_{\frac{d p}{k p + d},q}(\R^d), W^k_p(\R^d))_{\theta, q}} \asymp \|f\|_{W^k L_{r,q}(\R^d)}
	\end{equation}
	where the hidden constants are independent of $\theta$ (and, in particular, $r$). Assuming this momentarily, it follows from \eqref{ThmStein2.6++Rd} and \eqref{gamma2} that
	\begin{equation}\label{gamma3}
		\|f\|_{L_p(\R^d)} \leq C  \|f\|_{W^k L_{r,q}(\R^d)}
	\end{equation}
	where $C > 0$ does not depend on $r$.
	
	
	Collecting  \eqref{DefBesovInh},
 (\ref{vspom12}),  and \eqref{gamma3}, we obtain 
	\begin{align*}
		\|f\|^q_{B^{k - d/r + d/p}_{p,q}(\R^d),1}  &\asymp \left(\frac{k}{d}-\frac{1}{r} + \frac{1}{p}\right)^{-1} \|f\|^q_{L_p(\R^d)} +  \vertiii{f}^q_{B^{k - d/r + d/p}_{p,q}(\R^d),1} \\
		& \lesssim \left(1 + \left(\frac{k}{d}-\frac{1}{r} + \frac{1}{p}\right)^{-1} \right) (\|f\|^q_{W^k L_{r,q}(\R^d)} + \|f\|_{L_p(\R^d)}^q) \\
		& \asymp \left(r - \frac{d p}{k p + d}\right)^{-1} ( \|f\|^q_{W^k L_{r,q}(\R^d)} + \|f\|_{L_p(\R^d)}^q) \\
		& \lesssim  \left(r - \frac{d p}{k p + d}\right)^{-1} \|f\|^q_{W^k L_{r,q}(\R^d)},
	\end{align*}
hence  (\ref{SharpFrankeJawerth2Rd}) is shown.

To conclude the proof of (iii), it remains to show that the interpolation formula \eqref{gamma2} holds. We first note that by (A.15)
it is enough to prove that
 	\begin{equation}\label{gamma2new}
		\theta^{1/q} \|f\|_{(L_{\frac{d p}{k p + d},q}(\R^d), L_p(\R^d))_{\theta, q}} \asymp \|f\|_{L_{r,q}(\R^d)}.
	\end{equation}
	By \eqref{LemmaKFunctLorentz} and a simple change of variables, we have
	\begin{equation}\label{gamma4}
		\|f\|_{(L_{\frac{d p}{k p + d},q}(\R^d), L_p(\R^d))_{\theta, q}}^q  \asymp J_1 + J_2,
	\end{equation}
	where
	\begin{equation*}
		J_1:= \int_0^\infty t^{-\theta k q/d} \int_0^t (u^{k/d + 1/p} f^*(u))^q \frac{du}{u} \frac{dt}{t}
	\end{equation*}
	and
	\begin{equation*}
		J_2 := \int_0^\infty t^{(1-\theta) k q/d} \Big(\int_t^\infty (f^*(u))^p \, du \Big)^{q/p} \frac{dt}{t}.
	\end{equation*}
	On the one hand, applying the Fubini's theorem,
	\begin{equation}\label{gamma5}
		J_1 \asymp \theta^{-1} \int_0^\infty (u^{(1-\theta) k/d + 1/p} f^*(u))^q \frac{du}{u} = \theta^{-1} \|f\|_{L_{r,q}(\R^d)}^q.
	\end{equation}
	On the other hand, dealing with $J_2$, we can apply \eqref{HardyIneq} if $q \geq p$ to derive
	\begin{equation}\label{gamma6}
		J_2 \lesssim (1-\theta)^{-p/q} \int_0^\infty (t^{1/r} f^*(t))^q \frac{dt}{t} \asymp \|f\|_{L_{r,q}(\R^d)}^q,
	\end{equation}
	where we have also used that $1-\theta$ is uniformly bounded as $r \to \frac{d p}{k p + d}+$. If $q < p$ we can use monotonicity properties and change the order of summation to establish
	\begin{align}
		J_2 & \asymp \sum_{i=-\infty}^\infty 2^{-i(1-\theta)kq/d} \Big(\sum_{j=-\infty}^i  2^{-j} (f^*(2^{-j}))^p  \Big)^{q/p} \nonumber \\
		& \leq \sum_{i=-\infty}^\infty 2^{-i(1-\theta)kq/d} \sum_{j=-\infty}^{i} (2^{-j/p} f^*(2^{-j}))^q \nonumber \\
		& \asymp (1-\theta)^{-1} \sum_{j=-\infty}^\infty  (2^{-j/p} f^*(2^{-j}))^q 2^{-j(1-\theta)kq/d} \nonumber  \\
		& \asymp \sum_{j=-\infty}^\infty (2^{-j/r} f^*(2^{-j}))^q  \asymp \|f\|_{L_{r,q}(\R^d)}^q. \label{gamma7}
	\end{align}
	According to \eqref{gamma4}--\eqref{gamma7} we have
	\begin{equation*}
		\|f\|_{(L_{\frac{d p}{k p + d},q}(\R^d), L_p(\R^d))_{\theta, q}}^q  \asymp \theta^{-1} \|f\|^q_{L_{r,q}(\R^d)}
	\end{equation*}
	as $\theta \to 0+$ (or equivalently, $r \to \frac{d p}{k p + d}+$). The proof of \eqref{gamma2new} is complete.
	
The proof of (iii) $\Longrightarrow$  (i) $\iff$ (iv)
is as in Theorem \ref{ThmStein2}.
\end{proof}

\begin{proof}[Proof of Theorem \ref{ThmStein2RdHom}]
	(i) $\Longrightarrow$ (ii): We have
	\begin{equation}\label{ThmStein2.1RdHom}
		K(t^k, f ; L_p(\R^d), (V^k_p(\R^d))_0) \lesssim K(t^k, f; (V^k L_{\frac{d p}{k p + d},q}(\R^d))_0, (V^k_p(\R^d))_0).
	\end{equation}
Formulas  (\ref{LemmaKFunctLorentzSobolev}) and (\ref{LemmaKFunctLorentz}) yield
	\begin{align}
		K(t, f; (V^k L_{\frac{d p}{k p + d},q}(\R^d))_0, (V^k_p(\R^d))_0)& \asymp  K(t, |\nabla^k f| ; L_{\frac{d p}{k p + d},q}(\R^d), L_p(\R^d)) \nonumber  \\
		& \hspace{-5cm} \asymp  \left(\int_0^{t^{d/k}} (u^{k/d + 1/p} |\nabla^k f|^\ast(u))^q \frac{du}{u} \right)^{1/q} + t \left(\int_{t^{d/k}}^1 (|\nabla^k f|^\ast(u))^p du \right)^{1/p}. \label{New24**}
	\end{align}
	On the other hand, given any $g \in (V^k_p(\R^d))_0$, it follows from the triangle inequality and \eqref{DerMod} that
	\begin{equation*}
		\omega_k(f,t)_{p;\R^d} \leq \omega_k(f-g,t)_{p;\R^d} + \omega_k(g,t)_{p;\R^d} \lesssim \|f-g\|_{L_p(\R^d)} + t^k \| |\nabla^k g| \|_{L_p(\R^d)}.
	\end{equation*}
	Thus taking the infimum over all $g \in (V^k_p(\R^d))_0$, one obtains
	\begin{equation}\label{New24*}
	\omega_k(f,t)_{p;\R^d} \lesssim K(t^k,f ; L_p(\R^d), (V^k_p(\R^d))_0).
	\end{equation}
	The desired inequality (\ref{SharpKolyada2RdHom}) follows now from \eqref{ThmStein2.1RdHom}--\eqref{New24*}.
	
	The proof of the implication (ii) $\Longrightarrow$ (iii) follows similar steps as the corresponding implication in Theorem \ref{ThmStein2}. We only mention that the inequality \eqref{SharpKolyada2RdHom} can be applied for $f \in (V^k L_{r,q}(\R^d))_0$ because $(V^k L_{r,q}(\R^d))_0 \subset (V^k L_{\frac{d p}{k p + d},q}(\R^d))_0 +  (V^k_p(\R^d))_0$. 
	
	The equivalence (i) $\iff$ (iv)
was already discussed in Theorem \ref{ThmStein2}.
	
	(iii) $\Longrightarrow$ (i): Taking limits in homogeneous Besov norms (cf. \cite[Theorem 2]{Milman05} where the case $q \geq 1$ is treated, but the arguments given there can be easily extended to $q>0$)
	\begin{equation*}
	\lim_{ r \to \frac{d p}{k p + d}+} \left(r - \frac{d p}{k p + d}\right)^{1/q} \|f\|_{\dot{B}^{k - d/r + d/p}_{p,q}(\R^d),m}  \asymp \|f\|_{L_p(\R^d)}
	\end{equation*}
	and thus, by (\ref{SharpFrankeJawerth2RdHom}),
	\begin{equation*}	
		\|f\|_{L_p(\R^d)} \leq C  \||\nabla^k f|\|_{L_{r,q}(\R^d)}.
	\end{equation*}

\end{proof}

\begin{rem}
	The methods of proof given above to show that (i) $\iff$ (ii) $\iff$ (iii) in Theorems \ref{ThmStein2}, \ref{ThmStein2Rd} and \ref{ThmStein2RdHom}  also work with the limiting case $p = \frac{d}{d-k}$. In this case, the corresponding emdeddings given in (i) involve Sobolev spaces based on $\| \cdot \|_{L_{1,q}(\mathcal{X})}, \, q \leq 1$, which is not equivalent to a r.i. function norm if $q < 1$.
\end{rem}

	The counterparts of Theorems \ref{ThmStein2}, \ref{ThmStein2Rd} and \ref{ThmStein2RdHom} in terms of estimates involving only rearrangements read as follows.
	
	\begin{thm}\label{ThmStein3}
Let $k\in \mathbb{N}, d > k,  \frac{d}{d-k} < p < \infty, 1/\alpha = 1 - k/d$ and $0 < q \leq \infty$. The following statements are equivalent:
	\begin{enumerate}[\upshape(i)]
		\item
		\begin{equation*}
		W^k L_{\frac{d p}{k p + d},q}(\Omega) \hookrightarrow L_p(\Omega),
		\end{equation*}
		\item for $ f \in W^k_1(\Omega)$ and $t \in (0,1)$, we have
		\begin{equation}\label{SharpKolyada2*}
	\left(\int_t^1 \left(v^{1/p - 1/\alpha} \int_0^v u^{1/\alpha} f^\ast(u) \frac{du}{u} \right)^p \frac{dv}{v} \right)^{1/p} \lesssim \sum_{l=0}^k \left(\int_t^1 (v^{k/d + 1/p} |\nabla^l f|^{\ast \ast}(v))^q \frac{dv}{v} \right)^{1/q},
\end{equation}
			\item if $r \to \frac{d p}{k p + d}-$ then
			\begin{equation*}
			W^k L_{r,q}(\Omega) \hookrightarrow L_{r^\ast,p}(\Omega), \quad r^\ast = \frac{d r}{d - k r},
			\end{equation*}
			 with norm $\mathcal{O}(1)$, i.e., there exists $C > 0$, which is independent of $r$, such that
		\begin{equation}\label{SharpFrankeJawerth2*}
			\|f\|_{L_{r^\ast,p}(\Omega)} \leq C \|f\|_{W^k L_{r,q}(\Omega)},
		\end{equation}
		\item
		\begin{equation*}
		q \leq p.
		\end{equation*}
	\end{enumerate}

\end{thm}

\begin{thm}\label{ThmStein3Rd}
Let $k\in \mathbb{N}, d > k,  \frac{d}{d-k} < p < \infty, 1/\alpha = 1 - k/d$ and $0 < q \leq \infty$. The following statements are equivalent:
	\begin{enumerate}[\upshape(i)]
		\item
		\begin{equation*}
		W^k L_{\frac{d p}{k p + d},q}(\R^d) \hookrightarrow L_p(\R^d),
		\end{equation*}
		\item for $ f \in W^k_1(\R^d) + W^k L_{\frac{d p}{k p + d},q}(\R^d)$ and $t \in (0,\infty)$, we have
		\begin{equation}\label{SharpKolyada2*Rd}
	\left(\int_t^\infty \left(v^{1/p - 1/\alpha} \int_0^v u^{1/\alpha} f^\ast(u) \frac{du}{u} \right)^p \frac{dv}{v} \right)^{1/p} \lesssim \sum_{l=0}^k \left(\int_t^\infty (v^{k/d + 1/p} |\nabla^l f|^{\ast \ast}(v))^q \frac{dv}{v} \right)^{1/q},
\end{equation}
			\item if $r \to \frac{d p}{k p + d}-$ then
			\begin{equation*}
			W^k L_{r,q}(\R^d) \hookrightarrow L_{r^\ast,p}(\R^d), \quad r^\ast = \frac{d r}{d - k r},
			\end{equation*}
			 with norm $\mathcal{O}(1)$, i.e., there exists $C > 0$, which is independent of $r$, such that
		\begin{equation}\label{SharpFrankeJawerth2*Rd}
			\|f\|_{L_{r^\ast,p}(\R^d)} \leq C \|f\|_{W^k L_{r,q}(\R^d)},
		\end{equation}
		\item
		\begin{equation*}
		q \leq p.
		\end{equation*}
	\end{enumerate}

\end{thm}

\begin{thm}\label{ThmStein3RdHom}
Let $k\in \mathbb{N}, d > k,  \frac{d}{d-k} < p < \infty, 1/\alpha = 1 - k/d$ and $0 < q \leq \infty$. The following statements are equivalent:
	\begin{enumerate}[\upshape(i)]
		\item
		\begin{equation*}
		(V^k L_{\frac{d p}{k p + d},q}(\R^d))_0 \hookrightarrow L_p(\R^d),
		\end{equation*}
		\item for $ f \in (V^k_1(\R^d))_0 + (V^k L_{\frac{d p}{k p + d},q}(\R^d))_0$ and $t \in (0,\infty)$, we have
		\begin{equation}\label{SharpKolyada2*RdHom}
	\left(\int_t^\infty \left(v^{1/p - 1/\alpha} \int_0^v u^{1/\alpha} f^\ast(u) \frac{du}{u} \right)^p \frac{dv}{v} \right)^{1/p} \lesssim \left(\int_t^\infty (v^{k/d + 1/p} |\nabla^k f|^{\ast \ast}(v))^q \frac{dv}{v} \right)^{1/q},
\end{equation}
			\item if  $r \to \frac{d p}{k p + d}-$ then
			\begin{equation*}
			(V^k L_{r,q}(\R^d))_0 \hookrightarrow L_{r^\ast,p}(\R^d), \quad r^\ast = \frac{d r}{d - k r},
			\end{equation*}
			 with norm $\mathcal{O}(1)$, i.e., there exists $C > 0$, which is independent of $r$, such that
		\begin{equation}\label{SharpFrankeJawerth2*RdHom}
			\|f\|_{L_{r^\ast,p}(\R^d)} \leq C \||\nabla^k f|\|_{L_{r,q}(\R^d)}
		\end{equation}
		for $f \in (V^k L_{r,q}(\R^d))_0$,
		\item
		\begin{equation*}
		q \leq p.
		\end{equation*}
	\end{enumerate}

\end{thm}

\begin{rem}
The limiting case $p=\infty$ of the embedding $W^k L_{\frac{d p}{k p + d},q}(\Omega) \hookrightarrow L_p(\Omega)$ will be investigated in Theorem \ref{ThmStein*} below.
\end{rem}

\begin{proof}[Proof of Theorem \ref{ThmStein3}]
	The equivalence between (i) and (iv) was already treated in Theorem \ref{ThmStein2}.
	
	(i) $\Longrightarrow$ (ii): According to (i) and (\ref{SobolevClassical*}), we have
	\begin{equation}\label{remLp1}
			K(t, f; L_{\alpha,1}(\Omega), L_p(\Omega)) \lesssim K(t,f; W^k_1(\Omega) , W^k L_{d p /(k p + d),q}(\Omega))
	\end{equation}
	for $f \in W^k_1(\Omega)$. Since (cf. (\ref{LemmaKFunctLorentz}))
	\begin{equation*}
	K(t^{1/\alpha - 1/p}, f; L_{\alpha,1}(\Omega), L_p(\Omega)) \asymp \int_0^{t} v^{1/\alpha} f^\ast(v) \frac{dv}{v} + t^{1/\alpha - 1/p} \left(\int_{t}^1 f^\ast(v)^p dv \right)^{1/p}
	\end{equation*}
	and (cf. (\ref{LemmaKFunctLorentzSobolev}) and (\ref{LemmaKFunctLorentz}))
	\begin{align*}
		K(t^{1/\alpha-1/p} ,f; W^k_1(\Omega) , W^k L_{d p /(k p + d),q}(\Omega)) \\
		& \hspace{-6cm} \asymp \sum_{l=0}^k \bigg[ \int_0^t |\nabla^l f|^\ast(v) dv + t^{1/\alpha-1/p} \left(\int_t^1 (v^{k/d + 1/p} |\nabla^l f|^\ast(v))^q \frac{dv}{v} \right)^{1/q} \bigg],
	\end{align*}
	it follows from (\ref{remLp1}) that
	\begin{align}
		t^{-1/\alpha+1/p} \int_0^{t} v^{1/\alpha} f^\ast(v) \frac{dv}{v} + \left(\int_{t}^1 f^\ast(v)^p dv \right)^{1/p} \nonumber \\
		 & \hspace{-6cm} \lesssim \sum_{l=0}^k \bigg[t^{-1/\alpha+1/p}  \int_0^t |\nabla^l f|^\ast(v) dv +  \left(\int_t^1 (v^{k/d + 1/p} |\nabla^l f|^\ast(v))^q \frac{dv}{v} \right)^{1/q} \bigg]. \label{remLp2}
	\end{align}
	Further, in view of (\ref{HardyInequal1***}) (noting that $p >d/(d-k)$), we have for each $t \in (0,1/2)$
	\begin{align}
	t^{-1/\alpha+1/p}  \int_0^t |\nabla^l f|^\ast(v) dv +  \left(\int_t^1 (v^{k/d + 1/p} |\nabla^l f|^\ast(v))^q \frac{dv}{v} \right)^{1/q} \nonumber  \\
	& \hspace{-9cm}\asymp \left( \int_t^1 (v^{k/d + 1/p} |\nabla^l f|^{\ast \ast} (v))^q \frac{dv}{v} \right)^{1/q}, \quad l \in \{0, \ldots, k\}, \label{remLp3}
	\end{align}
	and
	\begin{align}
		t^{-1/\alpha+1/p} \int_0^{t} v^{1/\alpha} f^\ast(v) \frac{dv}{v} + \left(\int_{t}^1 f^\ast(v)^p dv \right)^{1/p} \nonumber  \\
		& \hspace{-6cm} \asymp \left(\int_t^1 \left(v^{1/p - 1/\alpha} \int_0^v u^{1/\alpha} f^\ast(u) \frac{du}{u} \right)^p \frac{dv}{v} \right)^{1/p}. \label{remLp4}
	\end{align}
	Therefore, inserting (\ref{remLp3}) and (\ref{remLp4}) into (\ref{remLp2}), we arrive at (\ref{SharpKolyada2*}).
	
	(ii) $\Longrightarrow$ (i): Let $f \in W^k L_{\frac{d p}{k p + d},q}(\Omega)$. By (\ref{SharpKolyada2*}), we have
		\begin{equation*}
	\left(\int_t^1  f^\ast(v)^p dv \right)^{1/p} \lesssim \sum_{l=0}^k \left(\int_t^1 (v^{k/d + 1/p} |\nabla^l f|^{\ast \ast}(v))^q \frac{dv}{v} \right)^{1/q}, \quad t \in (0,1). 
\end{equation*}
Taking the supremum over all $t \in (0,1)$, we obtain
 \begin{align*}
	\|f\|_{L_p(\Omega)} &= \left(\int_0^1 f^\ast(v)^p dv \right)^{1/p} \lesssim \sum_{l=0}^k \left(\int_0^1 (v^{k/d + 1/p} |\nabla^l f|^{\ast \ast}(v))^q \frac{dv}{v} \right)^{1/q} \\
	& \asymp \|f\|_{W^k L_{\frac{d p}{k p + d},q}(\Omega)},
\end{align*}
where the last estimate is an immediate consequence of (\ref{HardyInequal1**}), which can also be applied even in the case $q < 1$ due to the monotonicity properties of rearrangements.
	
	(ii) $\Longrightarrow$ (iii): First we note that $q \leq p$ because we have already shown that (ii) $\iff$ (iv). Applying monotonicity properties, Minkowski's inequality, (\ref{SharpKolyada2*}) and Fubini's theorem, we derive
	\begin{align*}
		\|f\|_{L_{r^\ast,p}(\Omega)} & = \left(\int_0^1 (t^{1/r^\ast} f^\ast(t))^p \frac{dt}{t} \right)^{1/p} \\
		&\hspace{-2cm}\lesssim \left(  \int_0^1 t^{(1/r - k/d -1/p) p} \left( t^{1/p - 1/\alpha}\int_0^t u^{1/\alpha} f^\ast(u) \frac{du}{u} \right)^p \frac{dt}{t} \right)^{1/p}\\
		& \hspace{-2cm}  \asymp \left(\frac{1}{r} - \frac{k p + d}{dp}\right)^{1/q} \left(  \int_0^1 \left(\int_0^t \left(v^{1/r-k/d-1/p} t^{1/p - 1/\alpha}\int_0^t u^{1/\alpha} f^\ast(u) \frac{du}{u}  \right)^q \frac{dv}{v} \right)^{p/q} \frac{dt}{t} \right)^{1/p}\\
		& \hspace{-2cm} \leq \left(\frac{1}{r} - \frac{k p + d}{dp}\right)^{1/q} \left(  \int_0^1 v^{(1/r-k/d-1/p)q} \left(\int_v^1 \left( t^{1/p - 1/\alpha}\int_0^t u^{1/\alpha} f^\ast(u) \frac{du}{u}  \right)^p \frac{dt}{t} \right)^{q/p} \frac{dv}{v} \right)^{1/q}\\
		& \hspace{-2cm} \lesssim  \left(\frac{1}{r} - \frac{k p + d}{dp}\right)^{1/q} \sum_{l=0}^k  \left(\int_0^1  v^{(1/r - k/d -1/p) q} \int_v^1 (t^{k/d + 1/p} |\nabla^l f|^{\ast \ast}(t))^q \frac{dt}{t}  \frac{dv}{v}\right)^{1/q}  \\
		& \hspace{5cm} \asymp \sum_{l=0}^k \left(\int_0^1  (t^{1/r} |\nabla^l f|^{\ast \ast}(t))^q \frac{dt}{t} \right)^{1/q}.
	\end{align*}
	Therefore, to complete the proof of (\ref{SharpFrankeJawerth2*}) it suffices to show that
	\begin{equation}\label{Thm4.7H}
	 \sum_{l=0}^k \left(\int_0^1  (t^{1/r} |\nabla^l f|^{\ast \ast}(t))^q \frac{dt}{t} \right)^{1/q} \leq C \|f\|_{W^k L_{r,q}(\Omega)}, \quad  r \to \frac{d p}{k p + d}-,
	\end{equation}
	where $C > 0$ is independent of $r$.
	
	To verify  (\ref{Thm4.7H}), we first assume that $q \geq 1$. Since $r \to \frac{d p}{k p + d}-$ and $p > d/(d-k)$, we may assume without loss of generality that $1/r < 1$. Then, by Lemma \ref{LemmaHardySharp}, we have
	\begin{equation*}
			 \left(\int_0^1  (t^{1/r} |\nabla^l f|^{\ast \ast}(t))^q \frac{dt}{t} \right)^{1/q} \lesssim \left(1- \frac{1}{r} \right)^{-1} \||\nabla^l f|\|_{L_{r,q}(\Omega)}, \quad l=0, \ldots, k,
	\end{equation*}
	and so, (\ref{Thm4.7H}) follows because $1 - 1/r \to 1 - k/d + 1/p$ as $r \to \frac{d p}{k p + d}-$.
	
	Secondly, let $q < 1$. Applying monotonicity properties and Fubini's theorem, we obtain for each $l \in \{0, \ldots, k\}$,
	\begin{align*}
	\int_0^1  (t^{1/r} |\nabla^l f|^{\ast \ast}(t))^q \frac{dt}{t} & \asymp \sum_{j=0}^\infty 2^{-j(1/r -1) q} \left(\sum_{\nu=j}^\infty 2^{-\nu} |\nabla^l f|^\ast(2^{-\nu}) \right)^q \\
	&\hspace{-2cm} \leq \sum_{j=0}^\infty 2^{-j(1/r -1) q}  \sum_{\nu=j}^\infty (2^{-\nu} |\nabla^l f|^\ast(2^{-\nu}))^q \\
	&\hspace{-2cm}  \lesssim \left(1-\frac{1}{r} \right)^{-1} \sum_{\nu=0}^\infty (2^{-\nu/r} |\nabla^l f|^\ast(2^{-\nu}))^q \lesssim \||\nabla^l f |\|_{L_{r,q}(\Omega)}^q.
	\end{align*}
	The proof of (\ref{Thm4.7H}) is complete.
	
	To deal with the remaining implication (iii) $\Longrightarrow$ (i), it suffices to take limits as $r \to \frac{d p}{k p + d}-$ (or equivalently, $r^\ast \to p-$) in (\ref{SharpFrankeJawerth2*}) and apply the monotone convergence theorem.
\end{proof}

Proofs of Theorems \ref{ThmStein3Rd} and \ref{ThmStein3RdHom}
follow the same line of argument
as the proof of Theorem \ref{ThmStein3}.

We finish this section with some further generalizations of Theorems \ref{ThmSobRear}--\ref{ThmSobRearRdHom} and \ref{ThmStein3}--\ref{ThmStein3RdHom}.

Let $k \in \N, k < d$. The embeddings
	\begin{equation*}
			(V^k_p(\R^d))_0 \hookrightarrow L_{p^\ast,p}(\R^d), \quad 1 < p < \frac{d}{k}, \quad p^\ast = \frac{d p }{d - k p},
		\end{equation*}
		and
		\begin{equation*}
			(V^k L_{\frac{d p}{k p + d},p}(\R^d))_0 \hookrightarrow L_p(\R^d), \quad \frac{d}{d-k} < p < \infty, 
		\end{equation*}
		which were investigated in detail in Theorems \ref{ThmSobRearRdHom} and \ref{ThmStein3RdHom}, respectively, are special cases of the more general result due to Talenti \cite[(4.6)]{Talenti}: If $ 1 < p < \frac{d}{k}, \, p^\ast = \frac{d p}{d - k p}$ and $1 \leq q \leq \infty$ then
		\begin{equation}\label{FullS}
			(V^{k} L_{p,q}(\R^d))_0 \hookrightarrow L_{p^\ast,q}(\R^d).
		\end{equation}
		See also \cite{MilmanPustylnik}. Furthermore, the target space in (\ref{FullS}) is optimal among all r.i. spaces. In particular,
		\begin{equation*}
			(V^{k} L_{p,q_0}(\R^d))_0 \hookrightarrow L_{p^\ast,q_1}(\R^d) \iff q_0 \leq q_1.
		\end{equation*}
		
		Our method of proof of Theorems \ref{ThmSobRearRdHom} and \ref{ThmStein3RdHom} can easily be adapted to establish the corresponding characterizations of (\ref{FullS}). Namely, the following result holds.

		\begin{thm}\label{ThmSobRearFull}
Let $1 < p < \infty, 1 \leq q_0, q_1 \leq \infty$ and $k \in \mathbb{N}$. Assume that $k < d/p$. Let $p^\ast = d p /(d - k p)$ and $1/\alpha = 1 - k/d$. The following statements are equivalent:
	\begin{enumerate}[\upshape(i)]
		\item
		\begin{equation*}
			(V^k L_{p,q_0}(\R^d))_0 \hookrightarrow L_{p^\ast,q_1}(\R^d),
		\end{equation*}
		\item for $ f \in (V^k_1 (\R^d))_0+ (V^k L_{p,q_0}(\R^d))_0$ and $t \in (0,\infty)$, we have
		\begin{equation}\label{SharpKolyada2*'Full}
	\left(\int_t^\infty \left(v^{1/p-k/d - 1/\alpha} \int_0^v u^{1/\alpha} f^\ast(u) \frac{du}{u} \right)^{q_1} \frac{dv}{v} \right)^{1/q_1} \lesssim  \left(\int_t^\infty (v^{1/p} |\nabla^k f|^{\ast \ast}(v))^{q_0} \frac{dv}{v} \right)^{1/q_0},
\end{equation}
			\item if $r \to p-$ then
			\begin{equation*}
			(V^k L_{r,q_0}(\R^d))_0 \hookrightarrow L_{r^\ast,q_1}(\R^d), \quad r^\ast = \frac{d r}{d - k r},
			\end{equation*}
			 with norm $\mathcal{O}(1)$, i.e., there exists $C > 0$, which is independent of $r$, such that
		\begin{equation*}
			\|f\|_{L_{r^\ast,q_1}(\R^d)} \leq C \||\nabla^k f|\|_{L_{r,q_0}(\R^d)}
		\end{equation*}
		for $f \in (V^k L_{r,q_0}(\R^d))_0$,
		\item
		\begin{equation*}
		q_0 \leq q_1.
		\end{equation*}
	\end{enumerate}

\end{thm}
		
		Observe that the estimates (\ref{SharpKolyada2*'Full}) comprise (\ref{SharpKolyada2*'RdHom}) and (\ref{SharpKolyada2*RdHom}). More precisely, if $q_0 = q_1=p$ in (\ref{SharpKolyada2*'Full}) then we recover (\ref{SharpKolyada2*'RdHom}), i.e.,
			\begin{equation*}
	\left(\int_t^\infty \left(v^{1/p-k/d - 1/\alpha} \int_0^v u^{1/\alpha} f^\ast(u) \frac{du}{u} \right)^p \frac{dv}{v} \right)^{1/p} \lesssim \left(\int_t^\infty ( |\nabla^k f|^{\ast \ast}(v))^p dv \right)^{1/p},
\end{equation*}
and setting $r > d/(d-k), 1/p = k/d + 1/r$, and $q_0=q_1=r$ in (\ref{SharpKolyada2*'Full}) we obtain (\ref{SharpKolyada2*RdHom}), i.e.,
	\begin{equation*}
	\left(\int_t^\infty \left(v^{1/r - 1/\alpha} \int_0^v u^{1/\alpha} f^\ast(u) \frac{du}{u} \right)^r \frac{dv}{v} \right)^{1/r} \lesssim \left(\int_t^\infty (v^{k/d + 1/r} |\nabla^k f|^{\ast \ast}(v))^r \frac{dv}{v} \right)^{1/r}.
\end{equation*}

In a similar fashion as in Theorem \ref{ThmSobRearFull}, one can establish the corresponding results for $W^k L_{p,q_0}(\Omega) \hookrightarrow L_{p^\ast,q_1}(\Omega)$ (cf. Theorems \ref{ThmSobRear} and \ref{ThmStein3}) and for $W^k L_{p,q_0}(\R^d) \hookrightarrow L_{p^\ast,q_1}(\R^d)$ (cf. Theorems  \ref{ThmSobRearRd} and \ref{ThmStein3Rd}.) Further details are left to the reader.

\section{Critical case}\label{critical}
\subsection{Trudinger type embeddings with fixed target space}

Let $1 < p < \infty$ and $\frac{d}{p} \in \mathbb{N}$. The Trudinger embedding asserts that
\begin{equation}\label{AMT}
	W^{d/p}_p(\Omega) \hookrightarrow L_\infty (\log L)_{-1/p'}(\Omega).
\end{equation}
See \cite{Peetre, Pohozhaev, Strichartz, Trudinger, Yudovich}. Furthermore, this embedding is optimal within the class of Orlicz spaces. In particular, we have
\begin{equation}\label{AMTOptimal}
W^{d/p}_p(\Omega) \hookrightarrow L_\infty (\log L)_{-b}(\Omega) \iff b \geq 1/p'.
\end{equation}

The first goal of this section is to obtain characterizations of (\ref{AMT}) via growth estimates of rearrangements in terms of moduli of smoothness and extrapolation estimates. Namely, we have

\begin{thm}\label{ThmTruMos}
	Let $1 < p < \infty, \frac{d}{p} \in \mathbb{N}$ and $b \geq 0$. The following statements are equivalent:
	\begin{enumerate}[\upshape(i)]
		\item
		\begin{equation*}
		W^{d/p}_p(\Omega) \hookrightarrow L_\infty (\log L)_{-b}(\Omega),
		\end{equation*}
		\item for $f \in L_p(\Omega)$ and $t \in (0,1)$, we have
		\begin{align}\label{ThmTruMos1*}
			\left(\int_0^{t^d} f^*(u)^p du \right)^{1/p} +  t^{d/p} (1-\log t)^b \sup_{t^d \leq u \leq 1} (1-\log u)^{-b} f^\ast(u) & \\
			&\hspace{-5cm}\lesssim t^{d/p} (1-\log t)^b \|f\|_{L_p(\Omega)}  +(1 - \log t)^b \omega_{d/p}(f, t)_{p;\Omega}, \nonumber
		\end{align}
			\item if $q < \infty$ then
			\begin{equation*}
			W^{d/p}_p(\Omega) \hookrightarrow L_q(\Omega)
			\end{equation*}
			 with norm $\mathcal{O}(q^b)$, i.e., there exists $C > 0$, which is independent of $q$, such that
		\begin{equation*}
			\|f\|_{L_q(\Omega)} \leq C q^b \|f\|_{W^{d/p}_p(\Omega)},
		\end{equation*}
		\item
		\begin{equation*}
		b \geq 1/p'.
		\end{equation*}
	\end{enumerate}
\end{thm}


\begin{rem}
	(i) The two terms given in the left-hand side of \eqref{ThmTruMos1*} are independent of each other. More precisely, setting
	\begin{equation*}
		I:= \left(\int_0^{t^d} f^*(v)^p dv \right)^{1/p} \quad \text{and} \quad II:= t^{d/p} (1-\log t)^b \sup_{t^d \leq u \leq 1} (1-\log u)^{-b} f^\ast(u)
	\end{equation*}
	and taking $f^*(t) = t^{-1/p} (1-\log t)^{-\xi}, \, \xi > 1/p,$ it is plain to see that $II \lesssim I$. On the other hand, if $f^*(t) = 1$ then $I \lesssim II$.
	
	(ii) Estimate \eqref{ThmTruMos1*} (with $b=1/p'$) implies that
	\begin{equation}\label{ThmTruMos1*new+x}
		(1-\log t)^{1/p-1}\left(\int_0^{t^d} f^*(v)^p dv \right)^{1/p} \lesssim t^{d/p}  \|f\|_{L_p(\Omega)} + \omega_{d/p}(f, t)_{p;\Omega}
	\end{equation}
	and
	\begin{equation}\label{ThmTruMos1*new+}
		t^{d/p} (1-\log t)^{1/p-1} f^*(t^d) \lesssim t^{d/p} \|f\|_{L_p(\Omega)} + \omega_{d/p}(f, t)_{p;\Omega},
 	\end{equation}
{significantly improving Corollary 5 from Kolyada's paper \cite{Kolyada}.}
	
\end{rem}

\begin{proof}[Proof of Theorem \ref{ThmTruMos}]
	(i) $\Longrightarrow$ (ii): By (i), we have
	\begin{equation}\label{ThmTruMos1}
		K(t^{d/p}, f; L_p(\Omega), L_\infty (\log L)_{-b}(\Omega)) \lesssim K(t^{d/p}, f; L_p(\Omega), W^{d/p}_p(\Omega)) \asymp t^{d/p} \|f\|_{L_p(\Omega)} +  \omega_{d/p}(f,t)_{p;\Omega},
	\end{equation}
	where the last estimate follows from Lemma A.2 (iv).
	
	Next we compute $K(t, f; L_p(\Omega), L_\infty (\log L)_{-b}(\Omega))$. Using that $L_\infty (\log L)_{-b}(\Omega) = (L_p(\Omega), L_\infty(\Omega))_{(1,-b), \infty}$ (see (\ref{LorentzZygmundLimiting})), we can apply \eqref{LemmaHolmstedt2} and Lemma \ref{LemmaHolm4} together with elementary computations to establish
	\begin{align}
		K(t (1 - \log t)^b, f; L_p(\Omega), L_\infty (\log L)_{-b}(\Omega)) \nonumber \\
		&\hspace{-6cm} \asymp K(t (1-\log t)^b, f; L_p(\Omega), (L_p(\Omega), L_\infty(\Omega))_{(1,-b), \infty})  \nonumber \\
		& \hspace{-6cm} \asymp  \left(\int_0^{t^p} f^\ast(v)^p dv \right)^{1/p} +  t (1-\log t)^b \sup_{t^p \leq u \leq 1} u^{-1/p} (1-\log u)^{-b}  \left(\int_0^{u} f^\ast(v)^p dv \right)^{1/p}   \nonumber\\
		& \hspace{-6cm} \asymp  \left(\int_0^{t^p} f^\ast(v)^p dv \right)^{1/p}
		 + t (1-\log t)^b \sup_{t^p \leq u \leq 1} u^{-1/p} (1-\log u)^{-b}  \left(\int_{t^p}^u f^\ast(v)^p dv \right)^{1/p} \nonumber \\
		 & \hspace{-6cm} \asymp  \left(\int_0^{t^p} f^\ast(v)^p dv \right)^{1/p}  + t (1-\log t)^b \sup_{t^p \leq u \leq 1} (1-\log u)^{-b} f^\ast(u). \label{ThmTruMos2}
	\end{align}

	Now it follows from (\ref{ThmTruMos2}), (\ref{ThmTruMos1}) and (\ref{HomoMod}) that
	\begin{align*}
		\left(\int_0^{t^d} f^*(v)^p dv \right)^{1/p} +  t^{d/p} (1-\log t)^b \sup_{t^d \leq u \leq 1} (1-\log u)^{-b} f^\ast(u)  & \\
		& \hspace{-8cm}\asymp K(t^{d/p} (1 - \log t)^{b}, f; L_p(\Omega), L_\infty (\log L)_{-b}(\Omega)) \\
		& \hspace{-8cm} \lesssim t^{d/p} (1-\log t)^b \|f\|_{L_p(\Omega)} + \omega_{d/p}(f,t (1-\log t)^{b p/d})_{p;\Omega} \\
		& \hspace{-8cm}\lesssim  t^{d/p} (1-\log t)^b \|f\|_{L_p(\Omega)} +(1 - \log t)^b \omega_{d/p}(f, t)_{p;\Omega}.
	\end{align*}
	
		(ii) $\Longrightarrow$ (iii): Using the well-known fact that rearrangements preserve Lebesgue norms (see, e.g., \cite[Proposition 1.8, Chapter 2, page 43]{BennettSharpley}) together with (\ref{ThmTruMos1*}) (cf. also \eqref{ThmTruMos1*new+}) and (\ref{SobolevModuliD}), we obtain
	\begin{align*}
		\|f\|_{L_q(\Omega)} &\lesssim \left(\int_0^1 (-\log t)^{b q} d t\right)^{1/q} \|f\|_{L_p(\Omega)} + \left( \int_0^1 (t^{-1/p} (1-\log t)^b \omega_{d/p}(f, t^{1/d})_{p;\Omega})^q dt \right)^{1/q} \\
		& \lesssim \left(\int_0^1 (-\log t)^{b q} d t\right)^{1/q}  \left(\|f\|_{L_p(\Omega)} + \sup_{0 < t < 1} t^{-1/p}  \omega_{d/p}(f, t^{1/d})_{p;\Omega} \right) \\
		& \lesssim I_q \|f\|_{W^{d/p}_p(\Omega)},
	\end{align*}
	where $I_q := \left(\int_0^1 (-\log t)^{b q} d t\right)^{1/q}$. Next we show that $I_q \asymp q^b$, which leads to (iii). Indeed, employing Stirling's formula for the Gamma function we have $(\Gamma(\xi))^{1/\xi} \asymp \xi$ as $\xi \to \infty$, then
	\begin{equation*}
		I_q = \Gamma(b q + 1)^{1/q} = (b q)^{1/q} \Gamma (b q)^{1/q} \asymp \left(\Gamma(b q)^{1/(b q)} \right)^b \asymp q^b.
	\end{equation*}

	(iii) $\Longrightarrow$ (i): According to (iii) we have
	\begin{equation*}
		\sup_{j \geq 0} 2^{-j b} \|f\|_{L_{2^j d}(\Omega)} \lesssim \|f\|_{W^{d/p}_p(\Omega)},
	\end{equation*}
	and thus, the embedding $W^{d/p}_p(\Omega) \hookrightarrow L_\infty (\log L)_{-b}(\Omega)$ follows immediately from (\ref{ExtrapolationLZ}).

	The equivalence between (i) and (iv) was stated in (\ref{AMTOptimal}).
	
%
%
%
%

\end{proof}

\begin{rem}
It turns out that the statement (iii) in Theorem \ref{ThmTruMos} can be replaced by any  of the following extrapolation results:

$\text{(iii)}_a$ If $q < \infty$ then $W^{d/p}_p(\Omega) \hookrightarrow L_{q,\infty}(\Omega)$ with norm $\mathcal{O}(q^b)$. More precisely, there exists $C > 0$, which is independent of $q$, such that
		\begin{equation*}
			\|f\|_{L_{q,\infty}(\Omega)} \leq C q^b \|f\|_{W^{d/p}_p(\Omega)}.
		\end{equation*}
		
	$\text{(iii)}_b$ If $q < \infty$ then $B^{d/p-d/q}_{p,\infty}(\Omega) \hookrightarrow L_{q,\infty}(\log L)_{-b}(\Omega)$ with embedding constant uniformly bounded. More precisely, there exists $C > 0$, which is independent of $q$, such that
		\begin{equation*}
			\|f\|_{L_{q,\infty}(\log L)_{-b}(\Omega)} \leq C \|f\|_{B^{d/p-d/q}_{p,\infty}(\Omega), d/p}.
		\end{equation*}	
		

\end{rem}

\begin{proof}
		We start by proving that (ii) $\Longrightarrow \text{(iii)}_a$. By (ii) (cf. also \eqref{ThmTruMos1*new+}) and (\ref{SobolevModuliD}), we have for each $t \in (0,1)$,
		\begin{equation*}
			f^\ast(t) \lesssim (1-\log t)^{b} \Big( \|f\|_{L_p(\Omega)} + \sup_{0 < u < 1} u^{-1/p} \omega_{d/p}(f,u^{1/d})_{p;\Omega} \Big) \asymp (1-\log t)^{b}  \|f\|_{W^{d/p}_p(\Omega)},
		\end{equation*}
		which yields
		\begin{equation*}
			\|f\|_{L_{q,\infty}(\Omega)} = \sup_{0 < t < 1} t^{1/q} f^\ast(t) \lesssim  \|f\|_{W^{d/p}_p(\Omega)} \sup_{0 < t < 1} t^{1/q} (1-\log t)^{b} \asymp q^b \|f\|_{W^{d/p}_p(\Omega)}.
		\end{equation*}
		
		Next we show that $\text{(iii)}_a \Longrightarrow$ (i). According to $\text{(iii)}_a$, we have
		\begin{align*}
			\|f\|_{W^{d/p}_p(\Omega)} & \gtrsim \sup_{q > 1} q^{- b} \|f\|_{L_{q,\infty}(\Omega)}   = \sup_{0 < t < 1} f^\ast(t) \sup_{q > 1} e^{-(-\log t)/q} q^{-b}  \\
			& \asymp \sup_{0 < t < 1} (-\log t)^{-b} f^\ast(t) \asymp \|f\|_{L_\infty (\log L)_{-b}(\Omega)}.
		\end{align*}
		
		The implication (ii) $\Longrightarrow \text{(iii)}_b$ is obvious. Further, $\text{(iii)}_b \Longrightarrow$ (i) holds true. Indeed, it follows from $\text{(iii)}_b$ and (\ref{SobolevModuliD}) that
		\begin{align*}
			\|f\|_{L_{q,\infty} (\log L)_{-b}(\Omega)} & \lesssim \|f\|_{L_p(\Omega)} +
			 \sup_{q > 1} \sup_{0 < t < 1} t^{-d/p + d/q} \omega_{d/p}(f,t)_{p;\Omega} \\
			 &  = \|f\|_{L_p(\Omega)} + \sup_{0 < t < 1} t^{-d/p} \omega_{d/p}(f,t)_{p;\Omega} \sup_{q > 1} t^{d/q} \\
			 & \lesssim  \|f\|_{L_p(\Omega)} + \sup_{0 < t < 1} t^{-d/p} \omega_{d/p}(f,t)_{p;\Omega} \lesssim \|f\|_{W^{d/p}_p(\Omega)}.
		\end{align*}
		Passing to the limits $q \to \infty$, we obtain (i).
\end{proof}
		
		Some comments are in order. Assume $1 < p < \infty, k \in \N$, and $0 < s < d/p \leq k$. Let $1/r = 1 / p - s /d$. According to Karadzhov-Milman-Xiao \cite[Theorems 5 and 7]{KaradzhovMilmanXiao}, the following sharp version of the Sobolev inequality for Besov spaces holds
		\begin{equation}\label{KMX0}
			\|f\|_{L_{r,\infty}(\R^d)} \leq C (d - s p)^{-1} \|f\|_{\dot{B}^{s}_{p,\infty}(\R^d), k}, \quad f \in \dot{B}^{s}_{p,\infty}(\R^d),
		\end{equation}
		where $C > 0$ is independent of $s$. Here, we remark that $d/p$ is not necessarily positive integer. Note that the embedding $\dot{B}^s_{p,\infty}(\R^d) \hookrightarrow L_{r,\infty}(\R^d)$ is optimal within the context of r.i. spaces (cf. \cite{Martin} and \cite{Netrusov89}) and in virtue of (\ref{KMX0}) its norm blows up as $s \to d/p-$ (according to the facts that $\dot{W}^{d/p}_p(\R^d) \not \hookrightarrow L_\infty(\R^d)$ and \eqref{SobolevModuliDPer'}; cf. also \eqref{NonSobBoundedIntro}). This raises the natural question whether it is possible to obtain (non-trivial) free-smoothness Sobolev inequalities for Besov spaces as $s$ approaches the critical value $d/p$. To be more precise, the question is  about the existence of some r.i. space $X$ such that $X \subsetneq L_p(\R^d)$ and $\|f\|_X \leq C \|f\|_{B^{s}_{p,\infty}(\R^d), k}$ with $C$ independent of $f$ and the smoothness parameter $s$. In this respect, Theorem \ref{ThmTruMos} with $\text{(iii)}_b$ provides a positive answer if $k = d/p$ for Besov spaces on domains. Indeed, if $s \to d/p-$, then
		\begin{equation*}
		 B^s_{p,\infty}(\Omega) \hookrightarrow L_{r,\infty}(\log L)_{-1/p'}(\Omega) \quad \text{with uniform norm with respect to $s$}.
		 \end{equation*}
		 Furthermore, this result is sharp in the following sense. There exists $C>0$ independent of $s$ such that
		 \begin{equation*}
		 \|f\|_{L_{r,\infty}(\log L)_{-b}(\Omega)} \leq C \|f\|_{B^{s}_{p,\infty}(\Omega), d/p}
		 \end{equation*}
		 if and only if
		 \begin{equation*}
		 	 b \geq 1/p'.
		 \end{equation*}		
		Note that $L_{r,\infty}(\Omega) \subsetneq L_{r,\infty}(\log L)_{-1/p'}(\Omega) \subsetneq L_p(\Omega)$.

Let $1 < p < \infty$ and $\frac{d}{p} \in \mathbb{N}$. In Theorem \ref{ThmTruMos} we were concerned with characterizations of the embedding
\begin{equation*}
	W^{d/p}_p(\Omega) \hookrightarrow L_\infty (\log L)_{-1/p'}(\Omega)
\end{equation*}
(see (\ref{AMT})). As already mentioned in \eqref{AMTOptimal} this embedding is sharp within the class of Orlicz spaces (and, in particular, Zygmund spaces). However, it can be sharpened  in  the scale of  Lorentz-Zygmund spaces. Namely, in light of the Maz'ya-Hansson-Br\'ezis-Wainger embedding \eqref{BW+}, we have
\begin{equation}\label{BW}
	W^{d/p}_p(\Omega) \hookrightarrow L_{\infty,p} (\log L)_{-1}(\Omega)
\end{equation}
and
\begin{equation}\label{BWNew}
	L_{\infty,p} (\log L)_{-1}(\Omega) \subsetneq  L_\infty (\log L)_{-1/p'}(\Omega).
\end{equation}
This does not contradict (\ref{AMTOptimal}) because $L_{\infty,p} (\log L)_{-1}(\Omega)$ does not fit into the scale of Orlicz spaces. Furthermore, the target space in (\ref{BW}) is the best possible among the class of r.i. spaces (see \cite{Hansson} and \cite{CwikelPustylnik}). In particular, we have
\begin{equation}\label{BWOptimal}
	W^{d/p}_p(\Omega) \hookrightarrow L_{\infty,p} (\log L)_{-b}(\Omega) \iff b \geq 1.
\end{equation}


We now turn to the counterpart of Theorem \ref{ThmTruMos} for (\ref{BW}).

\begin{thm}\label{ThmBreWain}
	Let $1 < p < \infty, \frac{d}{p} \in \mathbb{N}$ and $b > 1/p$. The following statements are equivalent:
	\begin{enumerate}[\upshape(i)]
		\item
		\begin{equation*}
		W^{d/p}_p(\Omega) \hookrightarrow L_{\infty,p} (\log L)_{-b}(\Omega),
		\end{equation*}
		\item for $f \in L_p(\Omega)$ and $t \in (0,1)$, we have
		\begin{align}
			(1-\log t)^{1/p}\left(\int_0^{t^d} f^*(u)^p du \right)^{1/p} + t^{d/p} (1-\log t)^{b} \left(\int_{t^d}^1(1 - \log u)^{- b p} f^\ast(u)^p \frac{du}{u} \right)^{1/p} & \nonumber\\
			& \hspace{-8cm} \lesssim t^{d/p} (1-\log t)^{b} \|f\|_{L_p(\Omega)} + (1-\log t)^{b}  \omega_{d/p}(f, t)_{p;\Omega}, \label{ThmTruMos1*++}
		\end{align}
		\item we have
		\begin{equation*}
			B^{d/p-\lambda}_{p,p}(\Omega) \hookrightarrow L_{d/\lambda,p}(\log L)_{-b}(\Omega)
		\end{equation*}
		with norm $\mathcal{O}(\lambda^{1/p})$ as $\lambda \to 0+$. More precisely, if $0 < \lambda  < d/p$ then there exists $C > 0$, which is independent of $\lambda$, such that
		\begin{equation}\label{SharpSobolevLog*For}
			\|f\|_{L_{d/\lambda,p}(\log L)_{-b}(\Omega)} \leq C \lambda^{1/p} \|f\|_{B^{d/p-\lambda}_{p,p}(\Omega), d/p},
		\end{equation}
		\item
		\begin{equation*}
		b \geq 1.
		\end{equation*}
	\end{enumerate}
\end{thm}

Before we proceed with the proof of this theorem, some comments are in order.

\begin{rem}\label{RemEta}
	(i) To the best of our knowledge,  inequality (\ref{SharpSobolevLog*For}) with $b=1$, i.e.,
		\begin{equation}\label{SharpSobolevLog}
			\|f\|_{L_{d/\lambda,p}(\log L)_{-1}(\Omega)} \leq C \lambda^{1/p} \|f\|_{B^{d/p-\lambda}_{p,p}(\Omega), d/p}, \quad \lambda \to 0+,
		\end{equation}
	 is new. We now pursue a further study of this interesting estimate for its own sake. It can be considered a  sharp form of Sobolev embedding theorem as the smoothness parameter (i.e., $\frac{d}{p} - \lambda$) approaches the  critical value (i.e., $\frac{d}{p}$). This research area attracted a lot of attention starting from the celebrated work of Bourgain, Br\'ezis and Mironescu \cite{BourgainBrezisMironescu}. Among other results, they showed the following: Let $Q$ be a cube in $\R^d$ and assume $p \in [1,\infty), s \in [1/2,1)$ and $s < d/p$. Then there exists a positive constant $c_d$, which depends only on $d$, such that
	\begin{equation}\label{BBM}
		\Big\|f - \displaystyle\stackinset{c}{}{c}{}{-\mkern4mu}{\displaystyle\int_Q} f\Big\|_{L_{\frac{dp}{d - sp}}(Q)} \leq c_d ((1-s) (d - sp)^{1-p})^{1/p} \left(\int_Q \int_Q \frac{|f(x) - f(y)|^p}{|x - y|^{d + s p}} \,  {d} x \, {d} y \right)^{1/p}
	\end{equation}
	whenever the right-hand side is finite (recall that this functional defines a seminorm which is referred as Gagliardo seminorm). 
Here,
$\displaystyle\stackinset{c}{}{c}{}{-\mkern4mu}{\displaystyle\int_Q} f = \frac{1}{|Q|_d} \int_Q f(x) \, {d} x$. The counterpart of \eqref{BBM} for functions $f \in C^\infty_0(\R^d)$ also holds true.
Afterwards, Maz'ya and Shaposhnikova \cite{MazyaShaposhnikova} extended this result to the whole range $s \in (0,1)$. Namely, if $p \in [1,\infty), s \in (0,1)$ and $s < d/p$ then there exists a positive constant $c_{d,p}$, which depends only on $d$ and $p$, such that
		\begin{equation}\label{MS}
		\|f\|_{L_{\frac{dp}{d - sp}}(\R^d)} \leq c_{d,p} (s (1-s) (d - sp)^{1-p})^{1/p} \left(\int_{\R^d} \int_{\R^d} \frac{|f(x) - f(y)|^p}{|x - y|^{d + s p}} \, {d}x \, {d}y \right)^{1/p}
	\end{equation}
	for all $f \in C^\infty_0(\R^d)$. Further, we remark that
	\begin{equation*}
		\left(\int_{\R^d} \int_{\R^d} \frac{|f(x) - f(y)|^p}{|x - y|^{d + s p}} \, {d}x \, {d}y \right)^{1/p} \asymp \|f\|_{\dot{B}^s_{p,p}(\R^d),1} = \left(\int_0^\infty t^{-s p} \omega_1(f,t)_{p;\R^d}^p \frac{dt}{t} \right)^{1/p}
	\end{equation*}
 with equivalence constants independent of $s \in (0,1)$ (see \cite[Proposition 2.3]{KolyadaLerner} and \cite[Section 10.3]{DominguezTikhonov}).
	
	The corresponding sharp inequalities obtained by replacing the Lebesgue norms in (\ref{BBM}), (\ref{MS}) by Lorentz norms were studied in \cite{KaradzhovMilmanXiao}, \cite{KolyadaLerner}, \cite{EdmundsEvansKaradzhov}, and \cite{EdmundsEvansKaradzhov2}. In particular, it was shown by Karadzhov, Milman and Xiao \cite[Theorem 5]{KaradzhovMilmanXiao} that
	\begin{equation}\label{KMX}
		\|f\|_{L_{d/\lambda,p}(\R^d)} \leq C \lambda^{1/p - 1} \|f\|_{\dot{B}^{d/p - \lambda}_{p,p}(\R^d), d/p}, \quad \lambda \to 0+.
	\end{equation}
	
	Comparing (\ref{SharpSobolevLog}) and the counterpart of (\ref{KMX}) for domains we see that the target space in (\ref{KMX}) is smaller than the corresponding one in (\ref{SharpSobolevLog}), i.e.,
	\begin{equation*}
		\|f\|_{L_{d/\lambda,p}(\log L)_{-1}(\Omega)} \leq C \|f\|_{L_{d/\lambda,p}(\Omega)}
	\end{equation*}
	with $C > 0$ independent of $\lambda$. However, the crucial point, which makes our argument rather optimal, is the fundamental  difference between the behavior of the constants in (\ref{SharpSobolevLog}) and (\ref{KMX}) with respect to $\lambda$.
  Namely, the embedding constant in (\ref{KMX}) (i.e., $\lambda^{1/p-1}$)  blows up as $\lambda \to 0+$. This links to the fact that the corresponding target space $L_{d/\lambda,p}$ becomes trivial if $\lambda = 0$ (cf. Section \ref{SectionFunctionSpaces}). On the other hand, the embedding constant in (\ref{SharpSobolevLog}) remains bounded as $\lambda \to 0+$. In particular, taking limits as $\lambda \to 0 +$ in (\ref{SharpSobolevLog}) we find that
	\begin{equation*}
		B^{d/p}_{p,p}(\Omega) \hookrightarrow L_{\infty,p}(\log L)_{-1}(\Omega),
	\end{equation*}
	which is optimal in the setting of r.i. spaces (see \cite[Theorem 7]{Martin}).

	(ii) It is an open question whether the Br\'ezis-Wainger embedding (\ref{BW}) can be obtained from Sobolev embeddings via extrapolation methods. An explicit formulation was stated in Mart\'in and Milman \cite{MartinMilman}. In that  paper the authors observed that limits of Lorentz norms in the following Talenti inequality (cf. \cite[(4.6)]{Talenti} for $k =1$ and iterating for $k > 1$; see also (\ref{FullS}))
	\begin{equation}\label{Tal}
	\|f\|_{L_{r,q}(\Omega)} \leq C r \|f\|_{W^k L_{p,q}(\Omega)},  \quad 1 <p < \frac{d}{k}, \quad \frac{1}{r} = \frac{1}{p}-\frac{k}{d}, \quad 1 \leq q \leq \infty,
	\end{equation}
	are not useful to achieve (\ref{BW}). Here, $C$ is uniformly bounded as $p \to d/k -$. However, they showed that some limiting embeddings into $L_\infty(\Omega)$ follow from (\ref{Tal}) by extrapolation means.
	
	Theorem \ref{ThmBreWain} provides a positive answer to the above question and shows that (\ref{BW}) follows from
		\begin{equation*}
			\|f\|_{L_{r,p}(\log L)_{-1}(\Omega)} \leq C r^{-1/p} \|f\|_{B^{s}_{p,p}(\Omega),d/p}, \quad  0 < s < \frac{d}{p}, \quad \frac{1}{r} = \frac{1}{p}-\frac{s}{d},
		\end{equation*}
	(see (\ref{SharpSobolevLog})) by means of taking limits as $s \to d/p-$ (or equivalently, $r \to \infty$). Hence, we circumvent the problems  on
the limits of Lorentz norms in  embeddings (\ref{Tal}), as well as their  counterparts for the Besov space
	\begin{equation*}
		\|f\|_{L_{r,p}(\R^d)} \leq C r^{1 - 1/p} \|f\|_{\dot{B}^{s}_{p,p}(\R^d), d/p}
	\end{equation*}
	(see (\ref{KMX})),
	 by switching the target space from $L_{r,p}(\Omega)$ to the bigger Lorentz-Zygmund space $L_{r,p}(\log L)_{-1}(\Omega)$ but substantially improving the sharp exponent of the asymptotic decay of the norm from $ r^{1 - 1/p}$ to $r^{-1/p}$.

	 (iii) Estimate \eqref{ThmTruMos1*++} is stronger than \eqref{ThmTruMos1*} because
	 \begin{equation*}
	 	\left(\int_0^{t^d} f^*(u)^p du \right)^{1/p} \lesssim (1-\log t)^{1/p}\left(\int_0^{t^d} f^*(u)^p du \right)^{1/p}
	 \end{equation*}
	 and
	 \begin{equation*}
	 	\sup_{t^d \leq u \leq 1} (1-\log u)^{-b} f^\ast(u)  \lesssim \left(\int_{t^d}^1(1 - \log u)^{- b p} f^\ast(u)^p \frac{du}{u} \right)^{1/p}.
	 \end{equation*}
	 This fact can also be observed from \eqref{BWNew} and the equivalence between Sobolev inequalities and pointwise inequalities provided by Theorems \ref{ThmTruMos}(i), (ii) and \ref{ThmBreWain}(i), (ii).
	
	 (iv)  The two terms given in the left-hand side of \eqref{ThmTruMos1*++} are independent of each other. This can be seen by taking the functions $f_1(t) = 1$ and $f_2(t) = t^{-1/p} (1-\log t)^{-\xi}, \, \xi > 1/p$.

	 (v) It follows from \eqref{ThmTruMos1*++} (with $b=1$) that
	 \begin{equation*}
	  (1-\log t)^{1/p-1}\left(\int_0^{t^d} f^*(u)^p du \right)^{1/p} \lesssim t^{d/p} \|f\|_{L_p(\Omega)} +  \omega_{d/p}(f, t)_{p;\Omega}
	 \end{equation*}
	 (see also \eqref{ThmTruMos1*new+x}) and
	 \begin{equation}\label{ThmTruMos1*new+p}
	 t^{d/p} \left(\int_{t^d}^1(1 - \log u)^{-  p} f^\ast(u)^p \frac{du}{u} \right)^{1/p} \lesssim  t^{d/p} \|f\|_{L_p(\Omega)} +   \omega_{d/p}(f, t)_{p;\Omega}.
	 \end{equation}
	 It turns out that the left-hand sides of \eqref{ThmTruMos1*new+} and \eqref{ThmTruMos1*new+p} are not comparable.
This can be easily checked by taking
$f_1$ and $f_2$ such that
$f_1^\ast(t) = t^{-\varepsilon}$ for some $\varepsilon > 0$
and
$f_2^\ast(t) = (1-\log t)^{1/p'}$. Note that this does not contradict the fact that
\eqref{AMT} is weaker than \eqref{BW}. Here we observe that if $t \to 0+$ then the left-hand side of \eqref{ThmTruMos1*new+p} (after multiplication by $t^{-d/p}$) dominates the corresponding one for \eqref{ThmTruMos1*new+} (cf. \eqref{BWNew}).

%
%
%

\end{rem}

\begin{proof}[Proof of Theorem \ref{ThmBreWain}]
	(i) $\Longrightarrow$ (ii): Assume $W^{d/p}_p(\Omega) \hookrightarrow L_{\infty,p} (\log L)_{-b}(\Omega)$. Then, by Lemma A.2 (iv), we have
	\begin{eqnarray}\label{ThmBreWain1}\nonumber
		K(t^{d/p}, f; L_p(\Omega), L_{\infty,p} (\log L)_{-b}(\Omega)) &\lesssim& K(t^{d/p}, f; L_p(\Omega), W^{d/p}_p(\Omega))\\&&\qquad \asymp t^{d/p} \|f\|_{L_p(\Omega)} + \omega_{d/p}(f,t)_{p;\Omega}.
	\end{eqnarray}
	
	To estimate $K(t, f; L_p(\Omega), L_{\infty,p} (\log L)_{-b}(\Omega))$, we first note that $L_{\infty,p} (\log L)_{-b}(\Omega) = (L_p(\Omega), L_\infty(\Omega))_{(1,-b), p}$ (see (\ref{LorentzZygmundLimiting})), and then (\ref{LemmaHolmstedt2}), Lemma \ref{LemmaHolm4} and Fubini's theorem allow us to obtain
	\begin{align}
		K(t (1 - \log t)^{b - 1/p}, f; L_p(\Omega), L_{\infty,p} (\log L)_{-b}(\Omega)) \nonumber \\
		&\hspace{-7cm} \asymp K(t (1-\log t)^{b-1/p}, f; L_p(\Omega), (L_p(\Omega), L_\infty(\Omega))_{(1,-b), p}) \nonumber \\
		& \hspace{-7cm} \asymp  \left(\int_0^{t^p} f^\ast(u)^p du \right)^{1/p} \nonumber \\
		& \hspace{-6cm}+  t (1-\log t)^{b-1/p} \left(\int_{t}^1 u^{-p} (1-\log u)^{-b p}  \int_0^{u^p} f^\ast(v)^p dv \frac{du}{u}\right)^{1/p}  \nonumber\\
		& \hspace{-7cm} \asymp  \left(\int_0^{t^p} f^\ast(u)^p du \right)^{1/p} \nonumber \\
		& \hspace{-6cm} +  t (1-\log t)^{b-1/p} \left(\int_{t^p}^1 u^{-1} (1-\log u)^{-b p}  \int_{t^p}^{u} f^\ast(v)^p dv \frac{du}{u}\right)^{1/p} \nonumber \\
		& \hspace{-7cm} \asymp  \left(\int_0^{t^p} f^\ast(u)^p du \right)^{1/p}
+  t (1-\log t)^{b-1/p} \left(\int_{t^p}^1 (1-\log u)^{-b p} f^\ast(u)^p     \frac{du}{u} \right)^{1/p}.
 \label{ThmBreWain2}
	\end{align}
	
	Collecting (\ref{ThmBreWain2}) and (\ref{ThmBreWain1}), we find that
	\begin{align*}
	  \left(\int_0^{t^p} f^\ast(u)^p du \right)^{1/p}+  t (1-\log t)^{b-1/p} \left(\int_{t^p}^1 (1-\log u)^{-b p} f^\ast(u)^p     \frac{du}{u} \right)^{1/p} & \\
		 & \hspace{-8cm}\lesssim K(t (1 - \log t)^{b - 1/p}, f; L_p(\Omega), W^{d/p}_p(\Omega)) \\
		 & \hspace{-8cm}\asymp t (1-\log t)^{b-1/p} \|f\|_{L_p(\Omega)} + \omega_{d/p}(f,t^{p/d} (1-\log t)^{(b p-1)/d})_{p;\Omega} \\
		 & \hspace{-8cm} \lesssim  t (1-\log t)^{b-1/p} \|f\|_{L_p(\Omega)} + (1 - \log t)^{b-1/p} \omega_{d/p}(f,t^{p/d})_{p;\Omega},
	\end{align*}
	where we have used (\ref{HomoMod}) in the last step. Therefore (ii) follows.
	
	
(ii) $\Longrightarrow$ (iii): Applying (ii),  changing variables and taking into account \eqref{DefBesovInh}, we arrive at
\begin{align}
	\int_0^1 t^{\lambda p/d} \int_{t}^1(1 - \log u)^{- b p} f^\ast(u)^p \frac{du}{u} \frac{dt}{t} &\lesssim \int_0^1 t^{\lambda p/d} \frac{dt}{t} \|f\|_{L_p(\Omega)}^p + \int_0^1 t^{(\lambda -d/p) p} \omega_{d/p}(f,t)_{p;\Omega}^p \frac{dt}{t} \nonumber \\
	& \hspace{-4cm}\asymp \lambda^{-1} \|f\|^p_{L_p(\Omega)} + \vertiii{f}_{B^{d/p-\lambda}_{p,p}(\Omega), d/p}^p \asymp  \|f\|_{B^{d/p-\lambda}_{p,p}(\Omega), d/p}^p. \label{ThmBreWain3}
\end{align}
On the other hand, we have
\begin{align}
	\int_0^1 t^{\lambda p/d} \int_{t}^1(1 - \log u)^{- b p} f^\ast(u)^p \frac{du}{u} \frac{dt}{t}
&
\nonumber \\
	& \hspace{-5cm}\asymp \lambda^{-1} \int_0^1 u^{\lambda p/d} (1-\log u)^{-b p} f^\ast(u)^p \frac{du}{u} = \lambda^{-1} \|f\|_{L_{d/\lambda,p}(\log L)_{-b}(\Omega)}^p. \label{ThmBreWain4}
\end{align}
Combining (\ref{ThmBreWain3}) and (\ref{ThmBreWain4}), we obtain \eqref{SharpSobolevLog*For}.

(iii) $\Longrightarrow$ (i): By \eqref{DefBesovInh}, we have
\begin{align}
	 \|f\|_{B^{d/p-\lambda}_{p,p}(\Omega),d/p} &= \lambda^{-1/p} \|f\|_{L_p(\Omega)} + \left(\int_0^1 t^{(\lambda -d/p) p} \omega_{d/p}(f,t)_{p;\Omega}^p \frac{dt}{t}\right)^{1/p} \nonumber\\
	  &\leq  \lambda^{-1/p} \|f\|_{L_p(\Omega)} + \left( \int_0^1 t^{\lambda p} \frac{dt}{t}\right)^{1/p} \sup_{0 < t < 1} t^{-d/p} \omega_{d/p}(f,t)_{p;\Omega} \nonumber \\
	& \asymp \lambda^{-1/p} \Big(\|f\|_{L_p(\Omega)} +  \sup_{0 < t < 1} t^{-d/p} \omega_{d/p}(f,t)_{p;\Omega} \Big) \asymp \lambda^{-1/p} \|f\|_{W^{d/p}_p(\Omega)}, \label{ThmBreWain5}
\end{align}
where the last estimate follows from (\ref{SobolevModuliD}). Hence (iii) and (\ref{ThmBreWain5}) imply
\begin{equation*}
	 \|f\|_{L_{d/\lambda,p}(\log L)_{-b}(\Omega)}  \leq C \|f\|_{W^{d/p}_p(\Omega)}, \quad \lambda \in (0, d/p).
\end{equation*}
We may now let $\lambda \to 0+$ to establish the desired estimate (i).
	
	For (i) $\iff$ (iv), see (\ref{BWOptimal}).

\end{proof}

As application of Theorems \ref{ThmTruMos} and \ref{ThmBreWain} we obtain in Corollary \ref{CorET} below the exact behaviour of the embedding constant of
\begin{equation*}
	W^{d/p}_p(\Omega) \hookrightarrow L_{q,p}(\Omega),
\end{equation*}
where $p$ is fixed and $q \to \infty$ (see (\ref{CriticalLorentzIntro})). Related questions when the target space $L_{q,p}(\Omega)$ is replaced by the bigger Lebesgue space $L_q(\Omega)$ were investigated in \cite[Section 2.7.2, pages 89--92]{EdmundsTriebel}. Such questions are of great interest in connection to extrapolation theory; see \cite[page 92]{EdmundsTriebel} and the references within.

\begin{cor}\label{CorET}
Let $1 < p < \infty$ and $\frac{d}{p} \in \mathbb{N}$. Then
  $W^{d/p}_p(\Omega) \hookrightarrow L_{q,p}(\Omega)$ with norm $\mathcal{O}(q)$ as $q \to \infty$, that is, there exists $C > 0$, which is independent of $q$, such that
		\begin{equation*}
			\|f\|_{L_{q,p}(\Omega)} \leq C q \|f\|_{W^{d/p}_p(\Omega)}.
		\end{equation*}
		Furthermore, this result is sharp in the following sense
		 \begin{equation}\label{SobolevEmbCrit}
			\|f\|_{L_{q,p}(\Omega)} \leq C q^b \|f\|_{W^{d/p}_p(\Omega)} \iff b \geq 1.
		\end{equation}
\end{cor}	
\begin{proof}
		For $b \in \mathbb{R}$, we have
	\begin{align*}
		\left(\int_t^1 u^{p/q} f^\ast(u)^p \frac{du}{u}\right)^{1/p} & \leq \left(\int_t^1 (-\log u)^{-b p} f^\ast(u)^p \frac{du}{u}\right)^{1/p}  \sup_{t \leq u \leq 1} u^{1/q} (- \log u)^b \\
		& = q^b \left(\int_t^1 (-\log u)^{-b p} f^\ast(u)^p \frac{du}{u}\right)^{1/p}  \sup_{t^{1/q} \leq u \leq 1} u (- \log u)^b \\
		& \lesssim q^b \left(\int_t^1 (-\log u)^{-b p} f^\ast(u)^p \frac{du}{u}\right)^{1/p}.
	\end{align*}
	Thus, applying the inequality given in Theorem \ref{ThmBreWain}(ii) with $b=1$ (see also \eqref{ThmTruMos1*new+p}), we obtain
	\begin{equation*}
			\left(\int_t^1 u^{p/q} f^\ast(u)^p \frac{du}{u}\right)^{1/p} \lesssim q ( \|f\|_{L_p(\Omega)} +  t^{-1/p} \omega_{d/p}(f, t^{1/d})_p).
	\end{equation*}
	Taking the supremum over all $t \in (0,1)$ on both sides and using a simple change of variables, we derive
	\begin{equation*}
		\|f\|_{L_{q,p}(\Omega)} \lesssim q \big(\|f\|_{L_p(\Omega)} + \sup_{0 < t < 1} t^{-d/p} \omega_{d/p}(f,t)_p\big) \lesssim q \|f\|_{W^{d/p}_p(\Omega)},
	\end{equation*}
	where the last estimate follows from (\ref{SobolevModuliD}).
	
	It remains to show the only-if part in (\ref{SobolevEmbCrit}). Assume that there exists $b$ such that
	\begin{equation*}
			\|f\|_{L_{q,p}(\Omega)} \leq C q^b \|f\|_{W^{d/p}_p(\Omega)}.
		\end{equation*}
		Combining this inequality together with the estimate (cf. \cite[p. 192]{SteinWeiss})
		\begin{equation*}
			\|f\|_{L_q(\Omega)} \leq \left(\frac{p}{q} \right)^{1/p} \|f\|_{L_{q,p}(\Omega)}, \quad p < q < \infty,
		\end{equation*}
		we arrive at
		\begin{equation*}
			\|f\|_{L_{q}(\Omega)} \leq C q^{b-1/p} \|f\|_{W^{d/p}_p(\Omega)}.
		\end{equation*}
		Then, according to Theorem \ref{ThmTruMos}, we have $b -1/p \geq 1/p'$. 
\end{proof}

\subsection{Sobolev embeddings into $L_\infty$}
We start with the following 
 Sobolev embeddings into $L_\infty(\Omega)$: if $k < d$, then
\begin{equation}\label{Stein}
	W^k L_{d/k,1}(\Omega) \hookrightarrow L_\infty(\Omega),
\end{equation}
see \cite{Stein81} and \cite{DeVoreSharpley}. Furthermore, this embedding is sharp, that is,
	\begin{equation}\label{SharpStein}
		W^k L_{d/k,q}(\Omega) \hookrightarrow L_\infty(\Omega) \iff  q \leq 1.
	\end{equation}

%

Next we complement Theorems \ref{ThmTruMos} and \ref{ThmBreWain} by characterizations of (\ref{Stein}) in terms of sharp inequalities of the $L_\infty$-moduli of smoothness, as well as extrapolation estimates of Jawerth-Franke type embeddings for Lorentz-Sobolev spaces. More specifically, we establish the following three theorems.

\begin{thm}\label{ThmStein}
Let $k \in \mathbb{N}, d > k$ and $0 < q \leq \infty$. The following statements are equivalent:
	\begin{enumerate}[\upshape(i)]
		\item
		\begin{equation*}
		W^k L_{d/k,q}(\Omega) \hookrightarrow L_\infty(\Omega),
		\end{equation*}
		\item for $ f \in W^k L_{d/k,q}(\Omega)$ and $t \in (0,1)$, we have
		\begin{equation}\label{SharpKolyada}
	t^k \|f\|_{L_\infty(\Omega)} + \omega_k(f,t)_{\infty;\Omega} \lesssim \sum_{l=0}^k \left(\int_0^{t^d} (u^{k/d} |\nabla^l f|^\ast(u))^q  \frac{du}{u} \right)^{1/q},
\end{equation}
			\item if $r \to \frac{d}{k}+$ then
			\begin{equation}\label{FrankeJawerth}
			W^k L_{r,q}(\Omega) \hookrightarrow B^{k - d/r}_{\infty,q}(\Omega)
			\end{equation}
			 with norm $\mathcal{O}((r-\frac{d}{k})^{-1/q})$, i.e., for any $m \in \N$ there exists $C > 0$, which is independent of $r$, such that
		\begin{equation}\label{SharpFrankeJawerth}
			\|f\|_{B^{k - d/r}_{\infty,q}(\Omega), m} \leq C \Big(r - \frac{d}{k}\Big)^{-1/q} \|f\|_{W^k L_{r,q}(\Omega)},
		\end{equation}
		\item
		\begin{equation*}
		q \leq 1.
		\end{equation*}
	\end{enumerate}
\end{thm}

\begin{thm}\label{ThmSteinRd}
Let $k \in \mathbb{N}, d > k$ and $0 < q \leq \infty$. The following statements are equivalent:
	\begin{enumerate}[\upshape(i)]
		\item
		\begin{equation*}
		W^k L_{d/k,q}(\R^d) \hookrightarrow L_\infty(\R^d),
		\end{equation*}
		\item for $ f \in W^k L_{d/k,q}(\R^d) + W^k_\infty(\R^d)$ and $t \in (0,\infty)$, we have
		\begin{equation}\label{SharpKolyadaRd}
	\min\{1,t^k\} \|f\|_{L_\infty(\R^d)} + \omega_k(f,t)_{\infty;\R^d} \lesssim \sum_{l=0}^k \left(\int_0^{t^d} (u^{k/d} |\nabla^l f|^\ast(u))^q  \frac{du}{u} \right)^{1/q},
\end{equation}
			\item if $r \to \frac{d}{k}+$ then
			\begin{equation}\label{FrankeJawerthRd}
			W^k L_{r,q}(\R^d) \hookrightarrow B^{k - d/r}_{\infty,q}(\R^d)
			\end{equation}
			 with norm $\mathcal{O}((r-\frac{d}{k})^{-1/q})$, i.e., for any $m \in \N$ there exists $C > 0$, which is independent of $r$, such that
		\begin{equation}\label{SharpFrankeJawerthRd}
			\|f\|_{B^{k - d/r}_{\infty,q}(\R^d), m} \leq C \Big(r - \frac{d}{k}\Big)^{-1/q} \|f\|_{W^k L_{r,q}(\R^d)},
		\end{equation}
		\item
		\begin{equation*}
		q \leq 1.
		\end{equation*}
	\end{enumerate}
\end{thm}

\begin{thm}\label{ThmSteinRdHom}
Let $k \in \mathbb{N}, d > k$ and $0 < q \leq \infty$. The following statements are equivalent:
	\begin{enumerate}[\upshape(i)]
		\item
		\begin{equation*}
		(V^k L_{d/k,q}(\R^d))_0 \hookrightarrow L_\infty(\R^d),
		\end{equation*}
		\item for $ f \in (V^k L_{d/k,q}(\R^d))_0 + (V^k_\infty(\R^d))_0$ and $t \in (0,\infty)$, we have
		\begin{equation}\label{SharpKolyadaRdHom}
 \omega_k(f,t)_{\infty;\R^d} \lesssim \left(\int_0^{t^d} (u^{k/d} |\nabla^k f|^\ast(u))^q  \frac{du}{u} \right)^{1/q},
\end{equation}
			\item $r \to \frac{d}{k}+$ then
			\begin{equation}\label{FrankeJawerthRdHom}
			(V^k L_{r,q}(\R^d))_0 \hookrightarrow B^{k - d/r}_{\infty,q}(\R^d)
			\end{equation}
			 with norm $\mathcal{O}((r-\frac{d}{k})^{-1/q})$, i.e., for any $m \in \N$ there exists $C > 0$, which is independent of $r$, such that
		\begin{equation}\label{SharpFrankeJawerthRdHom}
			\|f\|_{\dot{B}^{k - d/r}_{\infty,q}(\R^d), m} \leq C \Big(r - \frac{d}{k}\Big)^{-1/q} \| |\nabla^k f|\|_{L_{r,q}(\R^d)}
		\end{equation}
		for $f \in (V^k L_{r,q}(\R^d))_0$,
		\item
		\begin{equation*}
		q \leq 1.
		\end{equation*}
	\end{enumerate}
\end{thm}

\begin{rem}\label{Remark ThmStein}
(i) Inequality (\ref{SharpKolyadaRdHom}) with $q=1$, i.e.,
	\begin{equation*}
	\omega_k(f,t)_\infty \lesssim \int_0^t u^k |\nabla^k f|^\ast(u^d)  \frac{du}{u}
\end{equation*}
was shown by DeVore and Sharpley \cite[Lemma 2]{DeVoreSharpley} if $k=1$ and functions $f$ defined on the unit cube in $\R^d$ and by Kolyada and P\'erez L\'azaro \cite[(1.6)]{KolyadaPerez} for higher-order derivatives and functions $f$ on $\R^d$. It plays a central role in the theory of function spaces as can be seen in \cite{Haroske}, \cite{GogatishviliMouraNevesOpic}, and the references given there. 

(ii) Recall that the Jawerth-Franke embedding for Lorentz-Sobolev spaces asserts that if $1 < r < p < \infty, 0 < q \leq \infty$ and $k > d \Big(\frac{1}{r} - \frac{1}{p} \Big)$, then
\begin{equation}\label{SeegerTrebels*}
	W^k L_{r,q}(\R^d) \hookrightarrow B^{k - d/r + d/p}_{p,q}(\R^d);
\end{equation}
cf. (\ref{SeegerTrebels}). However, the interesting case $p = \infty$ in (\ref{SeegerTrebels*}) was left open in \cite{SeegerTrebels}. As a byproduct of (\ref{FrankeJawerthRd}), we can cover this case and, in addition, obtain sharp estimates of the rates of blow-up of the corresponding embedding constant. More specifically, we derive
 \begin{equation}\label{newFJ}
	(V^k L_{r,q}(\R^d))_0 \hookrightarrow \dot{B}^{k - d/r}_{\infty, q}(\R^d), \quad q \leq 1, \quad d/r < k < d,
\end{equation}
with norm $\mathcal{O}((r-\frac{d}{k})^{-1/q})$. In particular, this shows that (\ref{newFJ}) does not hold in the limiting case $r=d/k$ in the case  when the Besov spaces are defined in terms of  moduli of smoothness. This matches the fact that $\dot{B}^0_{\infty,q}(\R^d)$ endowed with $\|f\|_{\dot{B}^0_{\infty,q}(\Omega), m} = \left(\int_0^\infty \omega_m(f,t)_\infty^q \frac{dt}{t} \right)^{1/q}, \, m \in \N,$ (see \eqref{DefBesov}) becomes trivial.

(iii) The difference  between the subcritical and critical cases given in Theorems \ref{ThmStein2} and \ref{ThmStein}, respectively, is the sharp exponent $q$. More specifically, in the subcritical case we obtain the optimal index $q = p$ (see Theorem \ref{ThmStein2}(iv)), whereas $q=1$ in the critical case (see Theorem \ref{ThmStein}(iv)).
%
\end{rem}

\begin{proof}[Proof of Theorem \ref{ThmStein}]
	(i) $\Longrightarrow$ (ii): By (\ref{LemmaKFunctLorentzSobolev}) and (\ref{LemmaKFunctLorentzLinfty}) we have
	\begin{align*}
		t^k \|f\|_{L_\infty(\Omega)}& + \omega_k(f,t)_{\infty;\Omega} \asymp K(t^k, f ; L_\infty(\Omega), W^k_\infty(\Omega))  \nonumber \\
		& \lesssim K(t^k, f; W^k L_{d/k,q}(\Omega), W^k_\infty(\Omega))
   \asymp \sum_{l=0}^k \left(\int_0^{t^{d}} (u^{k/d} |\nabla^l f|^\ast(u))^q  \frac{du}{u} \right)^{1/q}.
	\end{align*}
	
	
	(ii) $\Longrightarrow$ (iii): Since $r \to d/k+$, we may assume, without loss of generality, that $k - d/r < 1$. In light of (\ref{JacksonInequal}), it is enough to show (\ref{SharpFrankeJawerth}) with $m=1$. Following (\ref{ThmStein2.2prev}) and (\ref{ThmStein2.2prev1}) (with $p=\infty$), we have
	\begin{equation*}
		\vertiii{f}_{B^{k-d/r}_{\infty,q}(\Omega), 1} \leq C  \left( \int_0^1 t^{-(k-d/r) q} \omega_k(f, t)_{\infty;\Omega}^q \frac{dt}{t}\right)^{1/q} + C \|f\|_{L_\infty(\Omega)},
	\end{equation*}
	where $C > 0$ does not depend on $r$. Therefore, applying (ii) and a simple change of variables, we obtain
	\begin{align}
		\vertiii{f}_{B^{k-d/r}_{\infty,q}(\Omega), 1}
		& \lesssim \|f\|_{L_\infty(\Omega)} + \sum_{l=0}^k \left( \int_0^1 t^{-(k-d/r) q}  \int_0^t (u^k |\nabla^l f|^\ast(u^d))^q  \frac{du}{u} \frac{dt}{t}\right)^{1/q} \nonumber\\
		& \asymp \|f\|_{L_\infty(\Omega)} +  (k r-d)^{-1/q} \sum_{l=0}^k \left(\int_0^1 (t^{1/r} |\nabla^l f|^\ast(t))^q \frac{dt}{t} \right)^{1/q}. \label{6new}
	\end{align}
	It follows from \eqref{DefBesovInh} and \eqref{6new} that
	\begin{align*}
		\|f\|_{B^{k-d/r}_{\infty,q}(\Omega), 1} &\asymp \Big(k-\frac{d}{r}\Big)^{-1/q} \|f\|_{L_\infty(\Omega)} + \vertiii{f}_{B^{k-d/r}_{\infty,q}(\Omega), 1} \\
		& \lesssim  \Big(k-\frac{d}{r}\Big)^{-1/q}  (\|f\|_{L_\infty(\Omega)} + \|f\|_{W^k L_{r,q}(\Omega)}).
	\end{align*}
	Thus the proof of (iii) will be complete  if we show that
	\begin{equation}\label{p+p}
		\|f\|_{L_\infty(\Omega)} \leq C \|f\|_{W^k L_{r,q}(\Omega)},
	\end{equation}
	where $C$ is independent of $r$. This inequality can be derived as follows. Using (ii) we obtain
	\begin{align*}
	 t^k \|f\|_{L_\infty(\Omega)} &\lesssim \sum_{l=0}^k \left(\int_0^{t^d} (u^{k/d} |\nabla^l f|^\ast(u))^q  \frac{du}{u} \right)^{1/q} \\
	 & \leq  \sum_{l=0}^k \left(\int_0^{t^d} (u^{1/r} |\nabla^l f|^\ast(u))^q  \frac{du}{u} \right)^{1/q}  \sup_{0 < u < t^d} u^{k/d-1/r} \\
	 & = t^{k-d/r} \sum_{l=0}^k \left(\int_0^{t^d} (u^{1/r} |\nabla^l f|^\ast(u))^q  \frac{du}{u} \right)^{1/q}
	\end{align*}
	for all $t \in (0,1)$. Thus,
	\begin{equation*}
		t^{d/r} \|f\|_{L_\infty(\Omega)} \lesssim  \sum_{l=0}^k \left(\int_0^{t^d} (u^{1/r} |\nabla^l f|^\ast(u))^q  \frac{du}{u} \right)^{1/q}, \quad t \in (0,1).
	\end{equation*}
	Taking the supremum over all $t \in (0,1)$ in the previous inequality we arrive at \eqref{p+p}.

	(iii) $\Longrightarrow$ (i): Assume $r > d/k$. According to \eqref{DefBesovInh}, we have
	\begin{equation*}
		\|f\|_{L_\infty(\Omega)} \leq C \Big( k - \frac{d}{r}\Big)^{1/q} \|f\|_{B^{k-d/r}_{\infty,q}(\Omega), m} \asymp C  \Big(r-\frac{d}{k} \Big)^{1/q} \|f\|_{B^{k-d/r}_{\infty,q}(\Omega), m},
	\end{equation*}
	where $C > 0$ is independent of $r$. Invoking now (iii) we infer that
	\begin{equation*}
		\|f\|_{L_\infty(\Omega)} \leq C \|f\|_{W^k L_{r,q}(\Omega)},
	\end{equation*}
	where $C > 0$ does not depend on $r$. Thus from monotone convergence we deduce that $W^k L_{d/k,q}(\Omega) \hookrightarrow L_\infty(\Omega)$.

	The equivalence (i) $\iff$ (iv) was already stated in (\ref{SharpStein}).
\end{proof}

Proof of Theorem \ref{ThmSteinRd}  follows line by line the proof of Theorem \ref{ThmStein}.

\begin{proof}[Proof of Theorem \ref{ThmSteinRdHom}]
	(i) $\Longrightarrow$ (ii): Given any decomposition $f=g+h$ with $g \in L_\infty(\R^d)$ and $h \in (V^k_\infty(\R^d))_0$, it follows from the triangle inequality and \eqref{DerMod} that
	\begin{equation*}
		\omega_k(f,t)_{\infty;\R^d} \leq \omega_k(g,t)_{\infty;\R^d} + \omega_k(h,t)_{\infty;\R^d} \lesssim \|g\|_{L_\infty(\R^d)} + t^k \| |\nabla^k h| \|_{L_\infty(\R^d)}.
	\end{equation*}
	Thus taking the infimum over all decompositions of $f$, we derive
	\begin{equation}\label{New24}
	\omega_k(f,t)_{\infty;\R^d} \lesssim K(t^k,f ; L_\infty(\R^d), (V^k_\infty(\R^d))_0).
	\end{equation}
	 By \eqref{New24}, (i), (\ref{LemmaKFunctLorentzSobolev}) and (\ref{LemmaKFunctLorentzLinfty}) we have
	\begin{equation*}
		 \omega_k(f,t)_{\infty;\R^d}  \lesssim K(t^k, f; (V^k L_{d/k,q}(\R^d))_0, (V^k_\infty(\R^d))_0) \asymp  \left(\int_0^{t^{d}} (u^{k/d} |\nabla^k f|^\ast(u))^q  \frac{du}{u} \right)^{1/q}.
	\end{equation*}

The proof of 	(ii) $\Longrightarrow$ (iii) is similar to the proof in Theorem \ref{ThmStein}.
	
	(iii) $\Longrightarrow$ (i): Without loss of generality we may assume that $k - \frac{d}{r} < m$. Then
	\begin{equation*}
		 \|f\|_{\dot{B}^{k - d/r}_{\infty,q}(\R^d), m} \asymp \|f\|_{(L_\infty(\R^d), \dot{W}^m_\infty(\R^d))_{\frac{k r-d}{m r}}, q}
	\end{equation*}
	with equivalence constants independent of $r$. This enables us to rewrite (iii) as
	\begin{equation*}
		 \Big(k - \frac{d}{r}\Big)^{1/q} \|f\|_{(L_\infty(\R^d), \dot{W}^m_\infty(\R^d))_{\frac{1}{m} (k - \frac{d}{r})}, q} \leq C  \| |\nabla^k f|\|_{L_{r,q}(\R^d)}
	\end{equation*}
	where $C$ does not depend on $r$. After the change of variables $\theta = \frac{1}{m} \big(k-\frac{d}{r} \big)$, we can take limits in the previous inequality as $r \to \frac{d}{k}+$ (or equivalently, $\theta \to 0+$) to find that
	\begin{equation*}
		\lim_{\theta \to 0+}  \theta^{1/q}  \|f\|_{(L_\infty(\R^d), \dot{W}^m_\infty(\R^d))_{\theta, q}}  \leq C \||\nabla^k f|\|_{L_{d/k, q}(\R^d)}, \quad f \in (V^k L_{d/k,q}(\R^d))_0.
	\end{equation*}
	The limit of the interpolation norms given in the left-hand side can be expressed in terms of the moduli of smoothness by invoking \cite[Theorem 1]{Milman05}, namely,
	\begin{equation*}
	\lim_{\theta \to 0+}  \theta^{1/q}  \|f\|_{(L_\infty(\R^d), \dot{W}^m_\infty(\R^d))_{\theta, q}}  = \lim_{t \to \infty} K(t,f; L_\infty(\R^d), \dot{W}^m_\infty(\R^d)) \asymp  \lim_{t \to \infty} \omega_m(f,t)_{\infty;\R^d}.
	\end{equation*}
	Then the proof will be complete  if we show that
	\begin{equation}\label{Linfty}
		 \lim_{t \to \infty} \omega_m(f,t)_{\infty;\R^d} \asymp \|f\|_{L_\infty(\R^d)}
	\end{equation}
	whenever $\{x \in \R^d : |f(x)| > \lambda\}$ is of finite measure for each $\lambda > 0$ (note that however  formula \eqref{Linfty}  does not hold for any  function from  $L_\infty(\R^d)$). Clearly, we have $\omega_m(f,t)_{\infty;\R^d} \lesssim \|f\|_{L_\infty(\R^d)}, \, t > 0,$ and thus $\lim_{t \to \infty} \omega_m(f,t)_{\infty;\R^d} \lesssim \|f\|_{L_\infty(\R^d)}$. To show the converse inequality we shall first assume $m=1$.  Note that
	\begin{equation}\label{Linfty2}
		\omega_1(f,t)_{\infty;\R^d}  = \sup_{|h| \leq t} \sup_{x \in \R^d} |f(x+h)-f(x)| = \sup_{|x-y| \leq t} |f(y)-f(x)|.
	\end{equation}
	Let $\varepsilon > 0$, there exists $x_\varepsilon \in \R^d$ such that $\|f\|_{L_\infty(\R^d)} < |f(x_\varepsilon)| + \varepsilon$. On the other hand, there exists $y_\varepsilon \in \R^d, \, y_\varepsilon \neq x_\varepsilon,$ such that $|f(y_\varepsilon)| \leq \varepsilon$. Otherwise, the level set $\{x \in \R^d : |f(x)| > \varepsilon\} = \R^d$, which is not possible. Consequently, by \eqref{Linfty2}, we have
	\begin{align*}
		\sup_{t > 0}\omega_1(f,t)_{\infty;\R^d} & \geq \omega_1(f,|x_\varepsilon - y_\varepsilon|)_{\infty;\R^d} \geq |f(x_\varepsilon)-f(y_\varepsilon) | \\
		& \geq |f(x_\varepsilon)| - |f(y_\varepsilon)| \geq \|f\|_{L_\infty(\R^d)} - 2 \varepsilon.
	\end{align*}
	Taking limits as $\varepsilon \to 0+$ and using the monotonicity properties of the moduli of smoothness, we arrive at
	\begin{equation}\label{666}
		\lim_{t \to \infty} \omega_1(f,t)_{\infty;\R^d} = \sup_{t > 0}\omega_1(f,t)_{\infty;\R^d} \geq \|f\|_{L_\infty(\R^d)}.
	\end{equation}
	If $m > 1$ and $t > 0$ then, by \eqref{MarchaudInequal},
	\begin{equation*}
			\omega_1 (f,t)_{\infty;\R^d} \lesssim t \int_t^\infty \frac{\omega_m(f,u)_{\infty; \R^d}}{u} \frac{du}{u} \leq \sup_{u > 0} \omega_m(f,u)_{\infty;\R^d},
 \end{equation*}
 so that
 \begin{equation}\label{6661}
 	\sup_{t > 0} \omega_1 (f,t)_{\infty;\R^d} \lesssim \sup_{t > 0} \omega_m(f,t)_{\infty;\R^d}.
 \end{equation}
 Now the desired estimate $\lim_{t \to \infty} \omega_m(f,u)_{\infty;\R^d} \gtrsim \|f\|_{L_\infty(\R^d)}$ follows from the monotonicity properties of the moduli of smoothness, \eqref{6661} and \eqref{666}.

	The equivalence (i) $\iff$ (iv) was already stated in (\ref{SharpStein}).
\end{proof}

Next we deal with the counterpart of Theorem \ref{ThmStein} for estimates involving only rearrangements.

\begin{thm}\label{ThmStein*}
Let $k \in \mathbb{N}, d > k, 1/\alpha = 1 - k/d$, and $0 < q \leq \infty$. The following statements are equivalent:
	\begin{enumerate}[\upshape(i)]
		\item
		\begin{equation*}
		W^k L_{d/k,q}(\Omega) \hookrightarrow L_\infty(\Omega),
		\end{equation*}
		\item for $ f \in W^k_1(\Omega)$ and $t \in (0,1/2)$, we have
\begin{equation}\label{SharpKolyada**}
	 t^{-1/\alpha} \int_0^t v^{1/\alpha} f^\ast(v) \frac{dv}{v}  \lesssim \sum_{l=0}^k \left(\int_{t}^1 (v^{k/d} |\nabla^l f|^{\ast \ast}(v))^q \frac{dv}{v} \right)^{1/q},
\end{equation}
	\item if $r \to \frac{d}{k}-$ then
			\begin{equation}\label{FrankeJawerth**}
			W^k L_{r,q}(\Omega) \hookrightarrow L_{r^\ast, q}(\Omega), \quad r^\ast = \frac{d r}{d - k r},
			\end{equation}
			 with norm $\mathcal{O}((r^\ast)^{1/q})$, i.e., there exists $C > 0$, which is independent of $r$, such that
		\begin{equation}\label{SharpFrankeJawerth**}
			\|f\|_{L_{r^\ast, q}(\Omega)} \leq C (r^\ast)^{1/q} \|f\|_{W^k L_{r,q}(\Omega)},
		\end{equation}
		\item
		\begin{equation*}
		q \leq 1.
		\end{equation*}
	\end{enumerate}
\end{thm}

\begin{thm}\label{ThmStein*Rd}
Let $k \in \mathbb{N}, d > k, 1/\alpha = 1 - k/d$, and $0 < q \leq \infty$. The following statements are equivalent:
	\begin{enumerate}[\upshape(i)]
		\item
		\begin{equation*}
		W^k L_{d/k,q}(\R^d) \hookrightarrow L_\infty(\R^d),
		\end{equation*}
		\item for $ f \in W^k_1(\R^d) + W^k L_{d/k,q}(\R^d)$ and $t \in (0,\infty)$, we have
\begin{equation}\label{SharpKolyada**Rd}
	 t^{-1/\alpha} \int_0^t v^{1/\alpha} f^\ast(v) \frac{dv}{v}  \lesssim \sum_{l=0}^k \left(\int_{t}^\infty (v^{k/d} |\nabla^l f|^{\ast \ast}(v))^q \frac{dv}{v} \right)^{1/q},
\end{equation}
	\item if $r \to \frac{d}{k}-$ then 
			\begin{equation}\label{FrankeJawerth**Rd}
			W^k L_{r,q}(\R^d) \hookrightarrow L_{r^\ast, q}(\R^d), \quad r^\ast = \frac{d r}{d - k r},
			\end{equation}
			 with norm $\mathcal{O}((r^\ast)^{1/q})$, i.e., there exists $C > 0$, which is independent of $r$, such that
		\begin{equation}\label{SharpFrankeJawerth**Rd}
			\|f\|_{L_{r^\ast, q}(\R^d)} \leq C (r^\ast)^{1/q} \|f\|_{W^k L_{r,q}(\R^d)},
		\end{equation}
		\item
		\begin{equation*}
		q \leq 1.
		\end{equation*}
	\end{enumerate}
\end{thm}

\begin{thm}\label{ThmStein*RdHom}
Let $k \in \mathbb{N}, d > k, 1/\alpha = 1 - k/d$, and $0 < q \leq \infty$. The following statements are equivalent:
	\begin{enumerate}[\upshape(i)]
		\item
		\begin{equation*}
		(V^k L_{d/k,q}(\R^d))_0 \hookrightarrow L_\infty(\R^d),
		\end{equation*}
		\item for $ f \in (V^k_1(\R^d))_0 + (V^k L_{d/k,q}(\R^d))_0$ and $t \in (0,\infty)$, we have
\begin{equation}\label{SharpKolyada**RdHom}
	 t^{-1/\alpha} \int_0^t v^{1/\alpha} f^\ast(v) \frac{dv}{v}  \lesssim \left(\int_{t}^\infty (v^{k/d} |\nabla^k f|^{\ast \ast}(v))^q \frac{dv}{v} \right)^{1/q},
\end{equation}
	\item if $r \to \frac{d}{k}-$ then
			\begin{equation}\label{FrankeJawerth**RdHom}
			(V^k L_{r,q}(\R^d))_0 \hookrightarrow L_{r^\ast, q}(\R^d), \quad r^\ast = \frac{d r}{d - k r},
			\end{equation}
			 with norm $\mathcal{O}((r^\ast)^{1/q})$, i.e., there exists $C > 0$, which is independent of $r$, such that
		\begin{equation}\label{SharpFrankeJawerth**RdHom}
			\|f\|_{L_{r^\ast, q}(\R^d)} \leq C (r^\ast)^{1/q} \||\nabla^k f|\|_{L_{r,q}(\R^d)}
		\end{equation}
		for $f \in (V^k L_{r,q}(\R^d))_0$,
		\item
		\begin{equation*}
		q \leq 1.
		\end{equation*}
	\end{enumerate}
\end{thm}

The following set of comments pertains to Theorems \ref{ThmStein*}--\ref{ThmStein*RdHom}.

\begin{rem}\label{RemEps}
	(i)  Let $q=1$. Inequality (\ref{SharpKolyada**RdHom}) reads as follows
	\begin{equation}\label{NewK}
	 t^{-1/\alpha} \int_0^t v^{1/\alpha} f^\ast(v) \frac{dv}{v}  \lesssim \int_{t}^\infty v^{k/d} |\nabla^k f|^{\ast \ast}(v) \frac{dv}{v}.
\end{equation}
	It turns out that (\ref{NewK}) provides an improvement of the following estimate given in \cite[Corollary 3.2]{KolyadaNafsa}
	\begin{equation*}
	f^{\ast \ast}(t)  \lesssim \int_{t}^\infty v^{k/d} |\nabla^k f|^{\ast \ast}(v) \frac{dv}{v}, \quad f \in C^\infty_0(\R^d).
\end{equation*}
This follows from 
\begin{equation*}
\int_0^t f^\ast(v) dv \leq t^{1 -1/\alpha} \int_0^t v^{1/\alpha} f^\ast(v) \frac{dv}{v}.
\end{equation*}
In addition, it is clear  that the terms $f^{\ast \ast}(t)$ and $ t^{-1/\alpha} \int_0^t v^{1/\alpha} f^\ast(v) \frac{dv}{v}$ are not comparable (consider, e.g., the function $f^\ast(t) = t^{-1/\alpha} (1-\log t)^{\varepsilon}, \,\varepsilon < -1$, and invoke \cite[Chapter 2, Corollary 7.8, p. 86]{BennettSharpley}).

(ii) Setting $q=1$ in (\ref{SharpFrankeJawerth**}) we recover the sharp blow-up of the norm of the embedding (\ref{FrankeJawerth**}) as $r \to (d/k)-$ obtained by Talenti (see (\ref{Tal})).

(iii)  Theorem \ref{ThmStein*} provides the limiting version of Theorem \ref{ThmStein3} with $p=\infty$. Note that there are significant distinctions between these two theorems. For instance, in virtue of (\ref{SharpKolyada2*}) (with $q=p$) and (\ref{SharpKolyada**}) (with $q=1$), we have
	\begin{equation}\label{SharpKolyada2*New}
	\left(\int_t^1 \left(v^{1/p - 1/\alpha} \int_0^v u^{1/\alpha} f^\ast(u) \frac{du}{u} \right)^p \frac{dv}{v} \right)^{1/p} \lesssim \sum_{l=0}^k \left(\int_t^1 (v^{k/d + 1/p} |\nabla^l f|^{\ast \ast}(v))^p \frac{dv}{v} \right)^{1/p}
\end{equation}
and
\begin{equation}\label{SharpKolyada**new}
	 t^{-1/\alpha} \int_0^t v^{1/\alpha} f^\ast(v) \frac{dv}{v}  \lesssim \sum_{l=0}^k \int_{t}^1 v^{k/d} |\nabla^l f|^{\ast \ast}(v) \frac{dv}{v},
\end{equation}
respectively. Here  the left-hand side of (\ref{SharpKolyada**new}) corresponds to the one given in (\ref{SharpKolyada2*New}) with $p=\infty$, however the right-hand sides involve different norms.

Concerning the extrapolation estimates given in Theorems \ref{ThmStein3}(iii) and \ref{ThmStein*}(iii), we make the following observations. The optimal inequalities read as follows (see (\ref{SharpFrankeJawerth2*}) and (\ref{SharpFrankeJawerth**}))
	\begin{equation}\label{SharpFrankeJawerth2*new}
			\|f\|_{L_{r^\ast,p}(\Omega)} \leq C \|f\|_{W^k L_{r,p}(\Omega)}, \quad \frac{d}{d-k} < p < \infty, \quad  r \to \frac{d p}{k p + d}-,
		\end{equation}
		and
			\begin{equation}\label{SharpFrankeJawerth**new}
			\|f\|_{L_{r^\ast, 1}(\Omega)} \leq C r^\ast \|f\|_{W^k L_{r,1}(\Omega)}, \quad  r \to \frac{d}{k}-.
		\end{equation}
Here $r^\ast = \frac{d r}{d - k r}$ and $d > k$. The inequality \eqref{SharpFrankeJawerth**new} can be considered as the limiting version of \eqref{SharpFrankeJawerth2*new} with $p=\infty$ (formally speaking, $\frac{d p}{k p + d} = \frac{d}{k}$ if $p=\infty$). However, there are some distinctions between these two inequalities. Firstly, we switch from the second index $p$ in Lorentz spaces in (\ref{SharpFrankeJawerth2*new}) to the index $1$ in (\ref{SharpFrankeJawerth**new}). Secondly, the uniform behaviour of the embedding constant in (\ref{SharpFrankeJawerth2*new}) with respect to $r$ breaks down for $p=\infty$ where we obtain the blow up $r^\ast$ (see (\ref{SharpFrankeJawerth**new})). In particular, \eqref{SharpFrankeJawerth**new} fails to be true if $r=d/k$ (and so, $r^*=\infty$). Indeed, we observe that in this case the Lorentz space $L_{r*,1}(\Omega) = \{0\}$ (cf. Section \ref{SectionFunctionSpaces}.)

\end{rem}

\begin{proof}[Proof of Theorem \ref{ThmStein*}]
	(i) $\Longrightarrow$ (ii): It follows from (i) and the embedding
	\begin{equation*}
	W^k_1(\Omega) \hookrightarrow L_{\alpha,1}(\Omega)
	\end{equation*}
	 (cf. (\ref{SobolevClassical*})) that
	\begin{equation}\label{ThmStein*1}
		K(t, f; L_{\alpha,1}(\Omega), L_\infty(\Omega)) \lesssim K(t,f; W^k_1(\Omega) , W^k L_{d/k,q}(\Omega)).
	\end{equation}
	
	By the Holmstedt's formula (\ref{LemmaKFunctLorentzLinfty}),
	\begin{equation}\label{ThmStein*2}
	K(t, f; L_{\alpha,1}(\Omega), L_\infty(\Omega)) \asymp \int_0^{t^\alpha} v^{1/\alpha} f^\ast(v) \frac{dv}{v}.
	\end{equation}
	On the other hand, applying (\ref{DS2})   together with (\ref{LemmaKFunctLorentz}), we arrive at
	\begin{align}
	K(t,f; W^k_1(\Omega) , W^k L_{d/k,q}(\Omega)) & \asymp \sum_{l=0}^k K(t, |\nabla^l f|; L_1(\Omega),  L_{d/k,q}(\Omega)) \nonumber \\
	& \hspace{-3cm} \asymp \sum_{l=0}^k \bigg[ \int_0^{t^\alpha} |\nabla^l f|^\ast(v) dv + t \left( \int_{t^\alpha}^1 (v^{k/d} |\nabla^l f|^\ast(v))^q \frac{dv}{v}\right)^{1/q} \bigg]. \label{ThmStein*3}
	\end{align}
	Therefore, by (\ref{ThmStein*1}), (\ref{ThmStein*2}) and  (\ref{ThmStein*3}), we have
		\begin{equation*}
	t^{-1/\alpha} \int_0^t v^{1/\alpha} f^\ast(v) \frac{dv}{v} \lesssim \sum_{l=0}^k \bigg[ t^{-1/\alpha}  \int_0^{t} |\nabla^l f|^\ast(v) dv + \left( \int_{t}^1 (v^{k/d} |\nabla^l f|^\ast(v))^q \frac{dv}{v}\right)^{1/q} \bigg].
\end{equation*}
Hence, we complete the proof of (\ref{SharpKolyada**}) by invoking (\ref{HardyInequal1***}).

(ii) $\Longrightarrow$ (iii): Applying monotonicity properties and (\ref{SharpKolyada**}), we have
\begin{align*}
	\int_0^1 (t^{1/r^\ast} f^\ast(t))^q \frac{dt}{t} & \asymp \int_0^{1/2} (t^{1/r^\ast} f^\ast(t))^q \frac{dt}{t}  \\
	& \lesssim \int_0^{1/2} \left( t^{1/r^\ast - 1/\alpha} \int_0^t v^{1/\alpha} f^\ast(v) \frac{dv}{v}\right)^q \frac{dt}{t} \\
	 &\lesssim  \sum_{l=0}^k \int_0^1 t^{q/r^\ast} \int_{t}^1 (v^{k/d} |\nabla^l f|^{\ast \ast}(v))^q \frac{dv}{v} \frac{dt}{t} \\
	& \asymp r^\ast \sum_{l=0}^k \int_0^1 (v^{1/r} |\nabla^l f|^{\ast \ast} (v))^q \frac{dv}{v}.
\end{align*}
Therefore, to complete the proof of (iii), it is enough to apply the following inequality
\begin{equation}\label{ThmStein*7}
\sum_{l=0}^k	\left(\int_0^1 (v^{1/r} |\nabla^l f|^{\ast \ast} (v))^q \frac{dv}{v} \right)^{1/q} \leq C \|f\|_{W^k L_{r,q}(\Omega)}, \quad r \to \frac{d}{k}-,
\end{equation}
where $C > 0$ is independent of $r$. The proof of (\ref{ThmStein*7}) is completely analogous to that in the proof of (\ref{Thm4.7H}), so we omit further details.

(iii) $\Longrightarrow$ (i): Firstly, we claim that
\begin{equation}\label{ThmStein*8}
	\lim_{r^\ast \to \infty} (r^\ast)^{-1/q} \|f\|_{L_{r^\ast,q}(\Omega)} \asymp \|f\|_{L_\infty(\Omega)}.
\end{equation}
To prove this formula, we follow some ideas given in \cite{Milman05} in the abstract setting of interpolation spaces.  Since $\lim_{t \to 0+} f^\ast(t) = \|f\|_{L_\infty(\Omega)}$, given any $\varepsilon > 0$ there exists $\delta \in (0,1)$ such that
\begin{equation}\label{ThmStein*9}
	 \|f\|_{L_\infty(\Omega)}^q - \varepsilon  \leq f^\ast(t)^q \leq \|f\|_{L_\infty(\Omega)}^q + \varepsilon, \quad t \in (0, \delta).
\end{equation}
Since $r^\ast \to \infty$, we may assume that $r^\ast > 2$. Then
\begin{equation*}
	\frac{1}{r^\ast}  \int_\delta^1 t^{q/r^\ast} f^\ast(t)^q \frac{dt}{t}  = \frac{\delta^{q/r^\ast}}{r^\ast}  \int_\delta^1 \Big(\frac{t}{\delta}\Big)^{q/r^\ast} f^\ast(t)^q \frac{dt}{t} \leq  \frac{\delta^{q(1/r^\ast - 1/2)}}{r^\ast}  \int_0^1 t^{q/2} f^\ast(t)^q \frac{dt}{t}.
\end{equation*}
Now, taking limits as $r^\ast \to \infty$, we derive
\begin{equation}\label{ThmStein*10}
 \lim_{r^\ast \to \infty} \frac{1}{r^\ast}  \int_\delta^1 t^{q/r^\ast} f^\ast(t)^q \frac{dt}{t} = 0.
 \end{equation}

  On the other hand, it follows from (\ref{ThmStein*9}) that
\begin{align*}
	\frac{1}{q} \delta^{q/r^\ast}  (\|f\|_{L_\infty(\Omega)}^q - \varepsilon) &=  \frac{1}{r^\ast}  (\|f\|_{L_\infty(\Omega)}^q - \varepsilon) \int_0^{\delta} t^{q/r^\ast} \frac{dt}{t} \\
	 & \leq \frac{1}{r^\ast}  \int_0^{\delta} t^{q/r^\ast} f^\ast(t)^q \frac{dt}{t} \\
	 & \leq \frac{1}{r^\ast}  (\|f\|_{L_\infty(\Omega)}^q + \varepsilon) \int_0^{\delta} t^{q/r^\ast} \frac{dt}{t} = \frac{1}{q} \delta^{q/r^\ast}  (\|f\|_{L_\infty(\Omega)}^q + \varepsilon),
\end{align*}
which yields
\begin{equation*}
	\frac{1}{q} (\|f\|_{L_\infty(\Omega)}^q - \varepsilon) \leq \lim_{r^\ast \to \infty}  \frac{1}{r^\ast}  \int_0^{\delta} t^{q/r^\ast} f^\ast(t)^q \frac{dt}{t} \leq \frac{1}{q}  (\|f\|_{L_\infty(\Omega)}^q + \varepsilon).
\end{equation*}
Using that $\varepsilon > 0$ was arbitrary, we obtain
\begin{equation}\label{ThmStein*11}
 \lim_{r^\ast \to \infty} \frac{1}{r^\ast}  \int_0^{\delta} t^{q/r^\ast} f^\ast(t)^q \frac{dt}{t} = \frac{1}{q} \|f\|_{L_\infty(\Omega)}^q.
\end{equation}
Therefore, (\ref{ThmStein*8}) follows from (\ref{ThmStein*10}) and (\ref{ThmStein*11}).

Now we are ready to show the implication (iii) $\Longrightarrow$ (i). Taking limits in (\ref{SharpFrankeJawerth**}) (noting that $r^\ast \to \infty$ if and only if $r \to (d/k)-$) and using (\ref{ThmStein*8}), we have
\begin{align*}
	\|f\|_{L_\infty(\Omega)} \asymp \lim_{r^\ast \to \infty} (r^\ast)^{-1/q}\|f\|_{L_{r^\ast, q}(\Omega)} \lesssim \lim_{r \to (d/k)-}  \|f\|_{W^k L_{r,q}(\Omega) } = \|f\|_{W^k L_{d/k,q}(\Omega)}.
\end{align*}

For the equivalence (i) $\iff$ (iv), see (\ref{SharpStein}).

\end{proof}

The proofs of Theorems \ref{ThmStein*Rd} and \ref{ThmStein*RdHom} are similar to
the one of Theorem \ref{ThmStein*}.

\section{Supercritical case}\label{supercritical}

\subsection{Br\'ezis-Wainger inequalities in terms of Lipschitz conditions}

A famous result by Br\'ezis and Wainger \cite{BrezisWainger} asserts that functions in the Sobolev space $\dot{H}^{1 + d/p}_p(\T^d), \, 1 < p < \infty,$ are Lipschitz-continuous up the logarithmic term. More precisely, they proved that
\begin{equation}\label{3333}
	\dot{H}^{1 + d/p}_p(\T^d) \hookrightarrow \text{Lip}^{(1, -1/p')}_{\infty, \infty}(\T^d).
\end{equation}
This embedding has been extensively studied in the literature and has found deep applications in the PDE's and the theory of function spaces (see, e.g., \cite{DominguezTikhonov19, Haroske} and the references given there).  
In particular, \eqref{3333} was recently extended in \cite{DominguezTikhonov19} to any Sobolev space in the supercritical case, that is,
\begin{equation}\label{BrezisWaingerClassic}
	\dot{H}^{\alpha + d/p}_p(\T^d) \hookrightarrow \text{Lip}^{(\alpha, -1/p')}_{\infty, \infty}(\T^d), \qquad \alpha > 0.
\end{equation}
Furthermore, if $d=1$ then the previous embedding is optimal in the following sense
\begin{equation}\label{BrezisWaingerClassic2}
	\dot{H}^{\alpha + 1/p}_p(\T) \hookrightarrow \text{Lip}^{(\alpha, -b)}_{\infty, \infty}(\T) \iff b \geq 1/p',
\end{equation}
see \cite{DominguezTikhonov19}. We also remark that the counterpart of \eqref{BrezisWaingerClassic} for function spaces on $\R^d$ holds. The arguments given in \cite{DominguezTikhonov19}, originally developed in the inhomogeneous situation, carry over to the homogeneous case as well.

We will see that \eqref{BrezisWaingerClassic} is closely connected to the so-called Ulyanov  inequalities for moduli of smoothness. Recall that the classical Ulyanov inequality states that (cf. \cite{Ulyanov})
\begin{equation}\label{UlyanovClassic}
	\omega_k(f,t)_{q;\T} \lesssim \left(\int_0^t (u^{-\theta} \omega_k(f,u)_{p;\T})^{q^\ast} \frac{du}{u} \right)^{1/q^\ast}
\end{equation}
where
 \begin{equation*}
		k \in \N, \quad 1 \leq p < q \leq \infty, \quad \theta = \frac{1}{p} - \frac{1}{q} \quad \text{and} \quad q^\ast = \left\{\begin{array}{cl} q,& q < \infty ,  \\
	1 , & q=\infty.
		       \end{array}
                        \right.
	\end{equation*}
	The sharp version of this inequality for $1<p<q<\infty$ was obtained in \cite{SimonovTikhonov, Trebels} with the help of the fractional moduli of smoothness. If $\alpha > 0$ then
\begin{equation}\label{UlyanovTik2}
		\omega_\alpha (f, t)_{q;\T} \lesssim \left( \int_0^t (u^{-\theta}  \omega_{\alpha + \theta} (f, u)_{p;\T})^{q} \frac{du}{u}\right)^{1/q}, \quad f \in L_p(\T).
	\end{equation}
Note that \eqref{UlyanovTik2} is a stronger estimate than \eqref{UlyanovClassic} (see \eqref{JacksonInequal}). The  inequality \eqref{UlyanovTik2} remains valid for functions $f \in L_p(\T^d)$ and $f \in L^p(\R^d)$. In the limiting cases, if $p=1$ or $q=\infty$ then for any $t \in (0,1)$, we have (cf. {\cite[Theorem 1]{Tikhonov}}, \cite{KolomoitsevTikhonov})
\begin{equation}\label{UlyanovTik}
		\omega_\alpha (f, t)_{q;\T} \lesssim \left( \int_0^t (u^{-\theta} (1 - \log u)^{1/\min\{p',q\}} \omega_{\alpha + \theta} (f, u)_{p;\T})^{q^\ast} \frac{du}{u}\right)^{1/q^\ast}, \quad f \in L_p(\T).
	\end{equation}

Our first concern in this section is to carry out the programme developed in the previous sections for (\ref{BrezisWaingerClassic}). Namely, we obtain the following





\begin{thm}\label{UlyanovBrezisWainger}
	Let $1 <p < \infty, \alpha > 0$ and $b \geq 0$. The following statements are equivalent:
	\begin{enumerate}[\upshape(i)]
	\item \begin{equation}\label{BrezisWainger}
		\dot{H}^{\alpha + 1/p}_p(\T) \hookrightarrow \emph{Lip}^{(\alpha, -b)}_{\infty, \infty}(\T),
	\end{equation}
	\item for $f \in \dot{B}^{1/p}_{p,1}(\T)$ and $t \in (0,1)$, we have
		\begin{equation}\label{UlyanovTikSharp}
		\omega_\alpha (f, t)_{\infty;\T} \lesssim  \int_0^{t (1 - \log t)^{b/\alpha}} u^{-1/p} \omega_{\alpha + 1/p} (f, u)_{p;\T} \frac{du}{u},
	\end{equation}
		\item if $\alpha_0 \to \alpha -$ then\footnote{We recall that
$	 \vertiii{f}_{\mathcal{C}^{\alpha_0}(\T),\alpha} :=  \sup_{0 < t < 1} (t^{-{\alpha_0}} \omega_{\alpha} (f,t)_{\infty;\T}).$}
		\begin{equation*}
		\dot{H}^{\alpha + 1/p}_p(\T) \hookrightarrow \dot{\mathcal{C}}^{\alpha_0}(\T)
		\end{equation*}
		 with norm $\mathcal{O}((\alpha - \alpha_0)^{-b})$, i.e., there exists $C > 0$, which is independent of $\alpha_0$, such that
		\begin{equation}\label{UlyanovBrezisWaingerSharp**<<}
		\vertiii{f}_{\mathcal{C}^{\alpha_0}(\T), \alpha} \leq C (\alpha- \alpha_0)^{-b} \|f\|_{\dot{H}^{\alpha + 1/p}_p(\T)}
		\end{equation}
		where $C$ is a positive constant independent of $\alpha_0$,
	\item \begin{equation*}
	b \geq 1/p'.
	\end{equation*}
	\end{enumerate}
\end{thm}

\begin{rem}\label{RemUlyanovBrezisWainger}
(i) Inequality (\ref{UlyanovTikSharp}) with $b=1/p'$, i.e.,
	\begin{equation}\label{UlyanovTikSharp*}
		\omega_\alpha (f, t)_{\infty;\T} \lesssim  \int_0^{t (1 - \log t)^{1/\alpha p'}} u^{-1/p} \omega_{\alpha + 1/p} (f, u)_{p;\T} \frac{du}{u}
	\end{equation}
	sharpens the known estimate (see (\ref{UlyanovTik}))
\begin{equation}\label{UlyanovTik*}
		\omega_\alpha (f, t)_{\infty;\T} \lesssim \int_0^t u^{-1/p} (1 - \log u)^{1/p'} \omega_{\alpha + 1/p} (f, u)_{p;\T} \frac{du}{u}.
	\end{equation}
 Indeed, by monotonicity properties of the moduli of smoothness (cf. Section \ref{SectionModuli}) and elementary estimates, we derive that
	\begin{align}
		\int_t^{t (1-\log t)^{1/\alpha p'}} u^{-1/p} \omega_{\alpha + 1/p}(f,u)_{p;\T} \frac{du}{u} & \lesssim  t^{-1/p-\alpha} \omega_{\alpha+1/p}(f,t)_{p;\T} \int_t^{t (1-\log t)^{1/\alpha p'}} u^\alpha \frac{du}{u} \nonumber \\
		& \hspace{-4cm}\asymp   t^{-1/p} (1-\log t)^{1/p'} \omega_{\alpha+1/p}(f,t)_{p;\T}  \nonumber \\
		& \hspace{-4cm} \lesssim \int_0^{t} u^{-1/p} (1-\log u)^{1/p'} \omega_{\alpha + 1/p}(f,u)_{p;\T} \frac{du}{u} \label{13}.
	\end{align}
Thus the right-hand side of (\ref{UlyanovTikSharp*}) is smaller than the corresponding one of (\ref{UlyanovTik*}).
		Note that  inequality (\ref{UlyanovTikSharp*}) is valid for a bigger class of functions than (\ref{UlyanovTik*}). For instance, take 
 $f \in L_p(\T)$ such that $\omega_{\alpha + 1/p}(f, u)_{p;\T} \asymp u^{1/p} (1-\log u)^{-\varepsilon}$, \,$1 < \varepsilon < 1 + 1/p'$ (see \cite{Tikhonovreal}).

(ii) We observe that the sharp estimate (\ref{UlyanovBrezisWaingerSharp**<<}) involves the semi-norm $\vertiii{f}_{\mathcal{C}^{\alpha_0}(\T), \alpha}$ rather than the usual one $\|f\|_{\dot{\mathcal{C}}^{\alpha_0}(\T), \alpha}= \sup_{t > 0} (t^{-\alpha_0} \omega_\alpha(f,t)_{\infty;\T})$. This technical issue is required in order to characterize the space $\L^{(\alpha,-b)}_{\infty,\infty}(\T)$ (see (\ref{DefLipschitz})) in terms of extrapolation of the scale $\dot{\mathcal{C}}^{\alpha_0}(\T), \, \alpha_0 < \alpha$.

(iii) The assumption $ b \geq 0$ in Theorem \ref{UlyanovBrezisWainger} is imposed to avoid trivial spaces. Recall that $\text{Lip}^{(\alpha, -b)}_{\infty, \infty}(\T) = \{0\}$ if $b < 0$ (see Section \ref{SectionFunctionSpaces}).

(iv)	Despite the fact that (\ref{BrezisWainger}) is optimal within the class $\text{Lip}^{(\alpha, -b)}_{\infty, \infty}(\T)$, we will prove in Theorem \ref{ThmBWRef1} below that (\ref{BrezisWainger}) can be improved with the help of the finer scale $\text{Lip}^{(\alpha, -b)}_{\infty, q}(\T), \, q < \infty$. See also Remark \ref{RemThmBWRef1}(ii) below.

(v) In the previous theorem we restrict our attention to univariate periodic functions.
 A comment for functions on $\T^d$ and $\R^d$ will be given in Remark \ref{Remark6.3} below.
\end{rem}

\begin{proof}[Proof of Theorem \ref{UlyanovBrezisWainger}]
(i) $\Longrightarrow$ (ii): Let $k \in \N$. Recall that $L_\infty(\T) + \infty \dot{W}^k_\infty(\T)$ means the \emph{Gagliardo completion} of $L_\infty(T)$ in $L_\infty(\T) + \dot{W}^k_\infty(\T)$ and
\begin{equation*}
	\|f\|_{L_\infty(\T) + \infty \dot{W}^k_\infty(\T)} := \sup_{t > 0} K(t, f; L_\infty(\T), \dot{W}^k_\infty(\T));
\end{equation*}
cf. \cite[p. 295]{BennettSharpley}.

According to Lemma \ref{LemmaKfunctLS} and \eqref{UlyanovClassic} with $q=\infty$, we have
\begin{equation}\label{BBounded}
	\|f\|_{L_\infty(\T) + \infty \dot{W}^k_\infty(\T)} \asymp \sup_{t > 0} \omega_k(f,t)_{\infty;\T} \lesssim \|f\|_{\dot{B}^{1/p}_{p,1}(\T),k}.
\end{equation}

On the other hand, we claim that
\begin{equation}\label{6664}
	\|f\|_{L_\infty(\T) + \infty \dot{W}^k_\infty(\T)} \asymp \|f\|_{L_\infty(\T) + \infty \dot{H}^\alpha_\infty(\T)}
\end{equation}
where
\begin{equation*}
	\|f\|_{L_\infty(\T) + \infty \dot{H}^\alpha_\infty(\T)} := \sup_{t > 0} K(t,f; L_\infty(\T), \dot{H}^\alpha_\infty(\T)),
\end{equation*}
Indeed, note that \eqref{6664} can be rewritten in terms of moduli of smoothness as (cf. Lemma \ref{LemmaKfunctLS})
\begin{equation}\label{0009}
	\sup_{t > 0} \omega_k(f,t)_{\infty;\T} \asymp \sup_{t > 0} \omega_\alpha(f,t)_{\infty;\T}
\end{equation}
and the validity of this formula is a simple consequence of \eqref{JacksonInequal} and \eqref{MarchaudInequal}.

Combining \eqref{BBounded} with \eqref{6664} we obtain
\begin{equation}\label{6665}
\|f\|_{L_\infty(\T) + \infty \dot{H}^\alpha_\infty(\T)}  \lesssim \|f\|_{\dot{B}^{1/p}_{p,1}(\T),k}, \quad f \in \dot{B}^{1/p}_{p,1}(\T). 
\end{equation}

It follows from \eqref{6665} and (i) that
	\begin{equation}\label{7}
		 K(t,f; L_\infty(\T) + \infty \dot{H}^\alpha_\infty(\T), \text{Lip}^{(\alpha,-b)}_{\infty,\infty}(\T)) \lesssim K(t,f; \dot{B}^{1/p}_{p,1}(\T), \dot{H}^{\alpha + 1/p}_p (\T)).
	\end{equation}
	Next, we proceed to estimate these $K$-functionals. According to \eqref{BInter}, we have
	\begin{equation*}
	\dot{B}^{1/p}_{p,1}(\T) = (L_p(\T), \dot{H}^{\alpha +1/p}_p(\T))_{\frac{1}{1+ \alpha p}, 1}
	\end{equation*}
	and thus, by (\ref{LemmaHolmstedt1*}), we obtain
	\begin{align*}
		K(t^{1-\frac{1}{1+ \alpha p}}, f; \dot{B}^{1/p}_{p,1}(\T) ,  \dot{H}^{\alpha +1/p}_p(\T)) & \asymp \int_0^t u^{-\frac{1}{1+ \alpha p}} K(u,f; L_p(\T), \dot{H}^{\alpha+1/p}_p(\T)) \frac{du}{u}	\\
		&\hspace{-4.5cm} \asymp \int_0^{t^{\frac{p}{1+ \alpha p}}} u^{-1/p} K(u^{\alpha + 1/p}, f; L_p(\T), \dot{H}^{\alpha +1/p}_p(\T)) \frac{du}{u}.
	\end{align*}
	Hence, a simple change of variables and Lemma \ref{LemmaKfunctLS} allow us to derive
	\begin{equation}\label{8}
		K(t^\alpha,f ; \dot{B}^{1/p}_{p,1}(\T) ,  \dot{H}^{\alpha+1/p}_p(\T)) \asymp \int_0^t u^{-1/p} \omega_{\alpha + 1/p}(f,u)_{p;\T} \frac{du}{u}.
	\end{equation}
	
	Next we estimate $K(t,f; L_\infty(\T) + \infty \dot{H}^\alpha_\infty(\T), \text{Lip}^{(\alpha,-b)}_{\infty,\infty}(\T))$. To do this, we first recall the well-known fact (cf. \cite[Theorem 1.5, Chapter 5, p. 297]{BennettSharpley})
	\begin{equation}\label{9099999}
		K(t,f;L_\infty(\T) + \infty \dot{H}^\alpha_\infty(\T),  \dot{H}^\alpha_\infty(\T)) = K(t,f;L_\infty(\T), \dot{H}^\alpha_\infty(\T)).
	\end{equation}
	This implies (cf. (\ref{LipLimInter}))
	\begin{equation}\label{909}
		\text{Lip}^{(\alpha,-b)}_{\infty,\infty}(\T) = (L_\infty(\T) + \infty \dot{H}^\alpha_\infty(\T), \dot{H}^\alpha_\infty(\T))_{(1,-b),\infty}.
	\end{equation}

	It follows from \eqref{909}, (\ref{LemmaHolmstedt2}) and Lemma \ref{LemmaKfunctLS} that
	\begin{align}
		K(t^\alpha (1- \log t)^{b},f; L_\infty(\T) + \infty \dot{H}^\alpha_\infty(\T), \text{Lip}^{(\alpha,-b)}_{\infty,\infty}(\T)) \nonumber \\
		&  \hspace{-7.8cm} \asymp K(t^\alpha (1-\log t)^{b},f; L_\infty(\T)+ \infty \dot{H}^\alpha_\infty(\T),  (L_\infty(\T)+ \infty \dot{H}^\alpha_\infty(\T), \dot{H}^\alpha_\infty(\T))_{(1,-b),\infty}) \nonumber \\
		& \hspace{-7.8cm} \asymp K(t^\alpha,f; L_\infty(\T)+ \infty \dot{H}^\alpha_\infty(\T), \dot{H}^\alpha_\infty(\T)) \nonumber \\
		&\hspace{-7.3cm} + t^\alpha (1- \log t)^{b} \sup_{t^\alpha < u < 1} u^{-1} (1-\log u)^{-b} K(u,f;L_\infty(\T)+ \infty \dot{H}^\alpha_\infty(\T), \dot{H}^\alpha_\infty(\T)) \nonumber \\
		& \hspace{-7.8cm} \gtrsim   K(t^\alpha,f; L_\infty(\T)+ \infty \dot{H}^\alpha_\infty(\T), \dot{H}^\alpha_\infty(\T)) \asymp \omega_\alpha(f,t)_{\infty;\T}. \label{KFunctLip2}
	\end{align}
	
	Plugging the estimates (\ref{KFunctLip2}) and (\ref{8}) into (\ref{7}), we arrive at
	\begin{equation*}
		 \omega_\alpha(f,t)_{\infty;\T} \lesssim  \int_0^{t (1-\log t)^{b/\alpha}} u^{-1/p} \omega_{\alpha+ 1/p}(f,u)_{p;\T} \frac{du}{u}.
	\end{equation*}
	
	(ii) $\Longrightarrow$ (iii): 	According to (ii), we have
	\begin{align*}
		\omega_\alpha (f, t)_{\infty;\T} & \lesssim  \int_0^{t (1 - \log t)^{b/\alpha}} u^{-1/p} \omega_{\alpha + 1/p} (f, u)_{p;\T} \frac{du}{u} \\
		& \lesssim t^\alpha (1-\log t)^b \sup_{u> 0} u^{-\alpha - 1/p} \omega_{\alpha + 1/p}(f,u)_{p;\T} \\
		& \asymp t^\alpha (1-\log t)^b \|f\|_{\dot{H}^{\alpha + 1/p}_p(\T)},
	\end{align*}
	where we have also used \eqref{SobolevModuliDPer'} in the last estimate. Therefore, applying a simple change of variables, we establish
	\begin{align*}
		\vertiii{f}_{\mathcal{C}^{\alpha_0}(\T), \alpha}&\asymp \sup_{0 < t < 1/2} t^{-\alpha_0} \omega_\alpha(f,t)_{\infty;\T}  \lesssim \|f\|_{\dot{H}^{\alpha + 1/p}_p(\T)} \sup_{0 < t < 1/2} t^{\alpha- \alpha_0} (-\log t)^b \\
		& \leq  (\alpha - \alpha_0)^{-b} \|f\|_{\dot{H}^{\alpha + 1/p}_p(\T)}  \sup_{0 < t < 1} t (- \log t)^b \\
		& \lesssim (\alpha - \alpha_0)^{-b} \|f\|_{\dot{H}^{\alpha + 1/p}_p(\T)}.
	\end{align*}

%
	
	(iii) $\Longrightarrow$ (i): Let $j_0 \in \N_0$ be such that $2^{-j_0} < \alpha$ and set $\alpha_j = \alpha - 2^{-j}, \, j \geq j_0$. By (iii), we have
	\begin{equation*}
		 2^{-j b}  \sup_{0 < t < 1/2} t^{-\alpha_j} \omega_\alpha(f,t)_{\infty;\T}  \leq C \|f\|_{\dot{H}^{\alpha + 1/p}_p(\T)}, \quad j \geq j_0,
	\end{equation*}
	where $C > 0$ is independent of $j$. Hence, taking the supremum over all $j \geq j_0$, we derive
	\begin{equation*}
		 \sup_{j \geq j_0} 2^{-j b}  \sup_{0 < t < 1/2} t^{-\alpha_j} \omega_\alpha(f,t)_{\infty;\T} \leq C \|f\|_{\dot{H}^{\alpha + 1/p}_p(\T)}
	\end{equation*}
	and thus, we will show (i) if we show that
	\begin{equation}\label{extrapol}
		  \sup_{j \geq j_0} 2^{-j b}  \sup_{0 < t < 1/2} t^{-\alpha_j} \omega_\alpha(f,t)_{\infty;\T}  \asymp \|f\|_{\text{Lip}^{(\alpha,-b)}_{\infty, \infty}(\T)}.
	\end{equation}
	Applying Fubini's theorem and elementary computations, we obtain
	\begin{align*}
	  \sup_{j \geq j_0} 2^{-j b}  \sup_{0 < t < 1/2} t^{-\alpha_j} \omega_\alpha(f,t)_{\infty;\T} & = \sup_{0 < t < 1/2} t^{-\alpha} \omega_\alpha(f,t)_{\infty;\T} \sup_{j \geq j_0} 2^{-j b} t^{2^{-j}}  \\
	  & \hspace{-3cm} \asymp \sup_{0 < t < 1/2} t^{-\alpha} (1-\log t)^{-b} \omega_\alpha(f,t)_{\infty;\T} \asymp \|f\|_{\text{Lip}^{(\alpha,-b)}_{\infty, \infty}(\T)},
	\end{align*}
	where the penultimate estimate follows from the change of variables
	\begin{equation*}
		\sup_{j \geq j_0} 2^{-j b} t^{2^{-j}} \asymp (-\log t)^{-b} \sup_{j \geq j_0 - \log (-\log t)} 2^{-2^{-j}} 2^{-j b}
	\end{equation*}
	together with the fact that
	\begin{equation*}
	2^{-2^{-j_0}} 2^{-j_0 b} \leq \sup_{j \geq j_0 - \log (-\log t)} 2^{-2^{-j}} 2^{-j b} \leq \sup_{x > 0} 2^{-x} x^b < \infty.
	\end{equation*}
	Hence (\ref{extrapol}) holds.
	
	The equivalence between (i) and (iv) was already stated in \eqref{BrezisWaingerClassic2}.

\end{proof}

\begin{rem}\label{Remark6.3}
	Imitating the proof of the implication (i) $\Longrightarrow$ (ii) in Theorem \ref{UlyanovBrezisWainger},  embedding \eqref{BrezisWaingerClassic} allows us to derive the following Ulyanov type inequality for functions $f \in L_p(\T^d), \, 1 < p < \infty$. Namely, if $\alpha > 0$ then
	 \begin{equation*}
		\omega_\alpha (f, t)_{\infty;\T^d} \lesssim  \int_0^{t (1 - \log t)^{1/\alpha p'}} u^{-d/p} \omega_{\alpha + d/p} (f, u)_{p;\T^d} \frac{du}{u}, \quad t \in (0,1),
	\end{equation*}
	whenever the right-hand side is finite. The corresponding inequality for functions $f \in L_p(\R^d)$ also holds true.
\end{rem}

Before going further, we briefly recall the definition of \emph{Besov spaces of logarithmic smoothness}. Let $0 < s < \alpha, 1 \leq p \leq \infty, 0 < q \leq \infty,$ and $-\infty < b < \infty$. Then the (periodic) space $\dot{B}^{s, b}_{p,q}(\T)$ is formed by all $f \in L_p(\T)$ such that
	 \begin{equation}\label{DefBesovLog}
	\|f\|_{\dot{B}^{s,b}_{p,q}(\T),\alpha} = \left(\int_0^\infty (t^{-s} (1 + |\log t|)^{b} \omega_{\alpha} (f,t)_{p;\T})^q \frac{dt}{t}\right)^{1/q} < \infty
\end{equation}
(with the usual change when $q=\infty$). Frequently, the integral in (\ref{DefBesovLog}) is taken over $(0,1)$, that is,
 \begin{equation}\label{DefHolZygLog}
	\vertiii{f}_{B^{s,b}_{p,q}(\T),\alpha} = \left(\int_0^1 (t^{-s} (1 -\log t)^{b} \omega_{\alpha} (f,t)_{p;\T})^q \frac{dt}{t}\right)^{1/q}.
\end{equation}
The space $B^{s,b}_{p,q}(\T)$ is formed by all those $f \in L_p(\T)$ such that \eqref{DefHolZygLog} is finite. 
In particular, setting $b=0$ in $\|f\|_{\dot{B}^{s,b}_{p,q}(\T),\alpha}$ (respectively, $\vertiii{f}_{B^{s,b}_{p,q}(\T),\alpha}$) we recover $\|f\|_{\dot{B}^s_{p,q}(\T),\alpha}$, see (\ref{DefBesov}) (respectively, $\vertiii{f}_{B^s_{p,q}(\T),\alpha}$, see (\ref{DefHolZyg})). For more details on function spaces of logarithmic smoothness, we refer the reader to \cite{DominguezTikhonov}.

Next we establish other sharpness assertions for \eqref{UlyanovTikSharp*} which complement that given by (ii) $\iff$ (iv) in Theorem \ref{UlyanovBrezisWainger}.

\begin{rem}\label{RemShaprnessAss}
	Inequality (\ref{UlyanovTikSharp*}) obtained in Theorem \ref{UlyanovBrezisWainger} is optimal in the following senses
	\begin{equation}\label{SharpnessAssUT}
		\omega_\alpha (f, t)_{\infty;\T} \lesssim  \left(\int_0^{t (1 - \log t)^{1/\alpha p'}} (u^{-1/p} \omega_{\alpha + 1/p} (f, u)_{p;\T})^q \frac{du}{u}\right)^{1/q} \iff q \leq 1
	\end{equation}
	and
	\begin{equation}\label{SharpnessAssUT2}
		\omega_\alpha (f, t)_{\infty;\T} \lesssim  \int_0^{t (1 - \log t)^{1/\alpha p'}} u^{-1/p} (1+| \log u|)^b \omega_{\alpha + 1/p} (f, u)_{p;\T} \frac{du}{u} \iff b \geq 0.
	\end{equation}
\end{rem}
	
\begin{proof}[Proof of Remark \ref{RemShaprnessAss}]	We show that the inequality
	\begin{equation*}
	\omega_\alpha (f, t)_{\infty;\T} \lesssim  \left(\int_0^{t (1 - \log t)^{1/\alpha p'}} (u^{-1/p} \omega_{\alpha + 1/p} (f, u)_{p;\T})^q \frac{du}{u}\right)^{1/q}
	\end{equation*}
	yields $q \leq 1$. Indeed, taking $t$ sufficiently large in the previous estimate we obtain $B^{1/p}_{p,q}(\T) \hookrightarrow L_\infty(\T)$. Since
	\begin{equation}\label{SickelTriebel}
	B^{1/p}_{p,q}(\T) \hookrightarrow L_\infty(\T) \iff q \leq 1,
	\end{equation}
	 see \cite[Theorem 3.3.1]{SickelTriebel}, (\ref{SharpnessAssUT}) follows.
	
	Next, we show that if
	\begin{equation}\label{SharpnessAssUT2*-}
		\omega_\alpha (f, t)_{\infty;\T} \lesssim  \int_0^{t (1 - \log t)^{1/\alpha p'}} u^{-1/p} (1+|\log u|)^b \omega_{\alpha + 1/p} (f, u)_{p;\T} \frac{du}{u}
	\end{equation}
	holds then $b \geq 0$. We will proceed by contradiction, that is, assume that there exists $b < 0$ such that (\ref{SharpnessAssUT2*-}) holds. Then, taking $t$ sufficiently large in (\ref{SharpnessAssUT2*-}), we derive
	\begin{equation}\label{SharpnessAssUT2**}
	B^{1/p, b}_{p,1}(\T) \hookrightarrow L_\infty(\T).
	\end{equation}
 Let us distinguish two possible cases. Assume first that $b \in (-1,0)$. Let $q \in \left(1, \frac{1}{b+1}\right)$. Then we have $B^{1/p}_{p,q}(\T) \hookrightarrow B^{1/p, b}_{p,1}(\T)$ (see \cite[Proposition 6.1]{DominguezTikhonov}) and, by (\ref{SharpnessAssUT2**}), $B^{1/p}_{p,q}(\T) \hookrightarrow L_\infty(\T)$, which is not true for $q > 1$ (see (\ref{SickelTriebel})). If $b \leq -1$ then the proof follows from the trivial embeddings $B^{1/p,b_0}_{p,1}(\T) \hookrightarrow B^{1/p,b}_{p,1}(\T), \, b_0 > b,$ and the previous case.
\end{proof}

\subsection{Br\'ezis-Wainger inequalities in terms of integral Lipschitz conditions}

According to \eqref{BrezisWaingerClassic2} the Br\'ezis-Wainger embedding
\begin{equation}\label{BWRef}
		\dot{H}^{\alpha + 1/p}_p(\T) \hookrightarrow \text{Lip}^{(\alpha, -1/p')}_{\infty, \infty}(\T)
	\end{equation}
	is the best possible among the class of the logarithmic Lipschitz spaces $\text{Lip}^{(\alpha, -b)}_{\infty, \infty}(\T)$. However, as already mentioned in Remark \ref{RemUlyanovBrezisWainger}(iv), (\ref{BWRef}) is not optimal within the broader scale of the spaces $\text{Lip}^{(\alpha, -b)}_{\infty, q}(\T)$.
	
	\begin{thm}\label{ThmBWRef1}
		Let $\alpha > 0, 1 < p < \infty$ and $b > 1/p$. Then we have
		\begin{equation}\label{ThmBWRef2}
		\dot{H}^{\alpha + 1/p}_p(\T) \hookrightarrow \emph{Lip}^{(\alpha, -b)}_{\infty, p}(\T) \iff b \geq 1.
	\end{equation}
	\end{thm}
	
	\begin{rem}\label{RemThmBWRef1}
		(i) For function spaces over $\R^d$ and $\alpha = 1$, a local version of the embedding given in (\ref{ThmBWRef2}) was obtained by Triebel \cite[(1.242), page 49]{Triebel06} and Haroske \cite[(6.18), page 96]{Haroske} as a part of the computation of the so-called continuity envelope of Sobolev spaces. However, their arguments do not allow to consider the fractional setting, that is, $\alpha > 0$. Below, we will give a new approach which will enable us to establish (\ref{ThmBWRef2}) for all $\alpha > 0$.
	
		(ii) Embedding (\ref{ThmBWRef2}) sharpens (\ref{BWRef}). More precisely, we will show that
		\begin{equation}\label{ThmBWRef3}
			 \text{Lip}^{(\alpha, -1)}_{\infty, p}(\T) \hookrightarrow  \text{Lip}^{(\alpha, -1/p')}_{\infty, \infty}(\T).
		\end{equation}
		Let $t \in (0,1)$. Using monotonicity properties of the moduli of smoothness (see Section \ref{SectionModuli}), we derive
		\begin{align*}
			\|f\|_{ \text{Lip}^{(\alpha, -1)}_{\infty, p}(\T)} & \geq \left(\int_0^t (u^{-\alpha} (1 - \log u)^{-1} \omega_\alpha(f,u)_{\infty;\T})^p \frac{du}{u} \right)^{1/p} \\
			& \gtrsim t^{-\alpha} \omega_\alpha(f,t)_{\infty;\T} \left( \int_0^t (1 - \log u)^{-p} \frac{du}{u}\right)^{1/p} \\
			& \asymp t^{-\alpha} (1 - \log t)^{-1/p'} \omega_\alpha(f,t)_{\infty;\T}.
		\end{align*}
		Now the embedding (\ref{ThmBWRef3}) follows by taking the supremum over all $t \in (0,1)$.
		
		Moreover, it is not hard to see that $\text{Lip}^{(\alpha, -1)}_{\infty, p}(\T) \neq \text{Lip}^{(\alpha, -1/p')}_{\infty, \infty}(\T).$
	\end{rem}
	
	\begin{proof}[Proof of Theorem \ref{ThmBWRef1}] We will make use of the fractional counterpart of (\ref{MarchaudInequal}) which states that given $\alpha, \gamma > 0$ there exists $t_0 \in (0,1)$ sufficiently small such that 
	\begin{equation}\label{Marchaud}
		\omega_\alpha(f,t)_{\infty;\T} \lesssim t^\alpha \int_t^1 \frac{\omega_{\gamma + \alpha}(f,u)_{\infty;\T}}{u^\alpha} \frac{du}{u}, \quad t \in (0,t_0).
	\end{equation}
	See \cite[Theorem 4.4]{KolomoitsevTikhonov}.
	
	We start by showing that
	\begin{equation}\label{ThmBWRef1.1}
	\dot{B}^\alpha_{\infty,p}(\T) \hookrightarrow \text{Lip}^{(\alpha, -1)}_{\infty, p}(\T).
	\end{equation}
	Applying (\ref{Marchaud}) and (\ref{HardyInequal4}), we have
	\begin{align*}
		\|f\|_{ \text{Lip}^{(\alpha, -1)}_{\infty, p}(\T)} &\asymp \left(\int_0^{t_0} (t^{-\alpha} (1 - \log t)^{-1} \omega_\alpha(f,t)_{\infty;\T})^p \frac{dt}{t} \right)^{1/p} \\
		& \lesssim \left( \int_0^{t_0}\left((1 - \log t)^{-1}   \int_t^1 \frac{\omega_{\alpha + \gamma}(f,u)_{\infty;\T}}{u^\alpha} \frac{du}{u} \right)^p \frac{dt}{t}\right)^{1/p} \\
		& \lesssim \left( \int_0^{t_0}\left((1 - \log t)^{-1}   \int_t^{t_0} \frac{\omega_{\alpha + \gamma}(f,u)_{\infty;\T}}{u^\alpha} \frac{du}{u} \right)^p \frac{dt}{t}\right)^{1/p} \\
		&\hspace{1cm} +  \int_{t_0}^1 \frac{\omega_{\alpha + \gamma}(f,u)_{\infty;\T}}{u^\alpha} \frac{du}{u}\\
		& \lesssim \left(\int_0^{t_0} (t^{-\alpha} \omega_{\alpha + \gamma}(f,t)_{\infty;\T})^p \frac{dt}{t} \right)^{1/p} + \omega_{\alpha + \gamma}(f,t_0)_{\infty;\T} \lesssim \|f\|_{\dot{B}^\alpha_{\infty,p}(\T), \alpha + \gamma}.
	\end{align*}
	
	 Combining (\ref{ThmBWRef1.1}) with the Jawerth-Franke embedding $\dot{H}^{\alpha + 1/p}_p(\T) \hookrightarrow \dot{B}^\alpha_{\infty,p}(\T)$ (see (\ref{FrankeMarschall}) below), we arrive at the embedding stated in (\ref{ThmBWRef2}).
	
	 Next we show the only-if part. We will proceed by contradiction. Assume that there exists $b \in (1/p,1)$ such that $\dot{H}^{\alpha + 1/p}_p(\T) \hookrightarrow \L^{(\alpha, -b)}_{\infty, p}(\T)$. It is plain to see that $\L^{(\alpha, -b)}_{\infty,p}(\T) \hookrightarrow \L^{(\alpha,-b+1/p)}_{\infty, \infty}(\T)$. Hence
	 \begin{equation*}
	 	\dot{H}^{\alpha + 1/p}_p(\T) \hookrightarrow \L^{(\alpha,-b+1/p)}_{\infty, \infty}(\T) \quad \text{for some} \quad b < 1.
	 \end{equation*}
	 However, this contradicts (\ref{BrezisWaingerClassic2}).
	\end{proof}

	\begin{rem}\label{Remark6.7}
	Similarly one can  show that
	\begin{equation}\label{RemarkNew1}
			\dot{H}^{\alpha + d/p}_p(\mathcal{X}) \hookrightarrow \text{Lip}^{(\alpha, -1)}_{\infty, p}(\mathcal{X}),\qquad \mathcal{X}\in \{\T^d,\R^d\}, \quad\alpha > 0, 1 < p < \infty.
		\end{equation}
	\end{rem}

The characterization of the refinement of the Br\'ezis-Wainger embedding, i.e., 
$
		\dot{H}^{\alpha + 1/p}_p(\T) \hookrightarrow \text{Lip}^{(\alpha, -1)}_{\infty, p}(\T)
 $, 
 	takes the following form.
	
\begin{thm}\label{UlyanovBrezisWaingerSharp}
	Let $1 <p < \infty, \alpha > 0$ and $b > 1/p$. The following statements are equivalent:
	\begin{enumerate}[\upshape(i)]
	\item \begin{equation*}
		\dot{H}^{\alpha + 1/p}_p(\T) \hookrightarrow \emph{Lip}^{(\alpha, -b)}_{\infty, p}(\T),
	\end{equation*}
		\item for $f \in \dot{B}^{1/p}_{p,1}(\T)$ and $t \in (0,1)$, we have
		\begin{align}
		\omega_\alpha (f, t)_{\infty;\T} + t^\alpha (1 - \log t)^{b-1/p} \left(\int_t^1 (u^{-\alpha} (1 - \log u)^{-b} \omega_\alpha (f,u)_{\infty;\T})^p \frac{du}{u} \right)^{1/p} \nonumber\\
		& \hspace{-10cm}\lesssim \int_0^{t (1 - \log t)^{(b-1/p)/\alpha}} u^{-1/p} \omega_{\alpha + 1/p} (f, u)_{p;\T} \frac{du}{u}, \label{UlyanovBrezisWaingerSharp<}
	\end{align}
		\item if $\lambda \to 0+$ then
		\begin{equation*}
		 \dot{B}^{\alpha + 1/p - \lambda}_{p,p}(\T) \hookrightarrow \dot{B}^{\alpha - \lambda, -b}_{\infty, p}(\T)
		 \end{equation*}
		  with norm $\mathcal{O}(\lambda^{1/p})$, i.e., there exists $C > 0$, which is independent of $\lambda$, such that
		\begin{equation}\label{UlyanovBrezisWaingerSharp**}
			\vertiii{f}_{B^{\alpha - \lambda, -b}_{\infty, p}(\T), \alpha} \leq C \lambda^{1/p}  \vertiii{f}_{B^{\alpha + 1/p - \lambda}_{p,p}(\T),  \alpha + 1/p },
		\end{equation}
	\item \begin{equation*}
	b \geq 1.
	\end{equation*}
	\end{enumerate}
\end{thm}

\begin{rem}\label{RemarkUlyanovBrezisWaingerSharp}
(i) The two terms given in the left-hand side of \eqref{UlyanovBrezisWaingerSharp<} are not comparable. Indeed, take $f \in L_\infty(\T)$ with $\omega_\alpha(f,t)_{\infty;\T} \asymp t^\alpha (1-\log t)^{1/p'} (1 + \log (1-\log t))^{-\varepsilon}$ where $\varepsilon < 1/p$ (cf. \cite{Tikhonovreal}).  Elementary computations show that $\omega_\alpha (f, t)_{\infty;\T} \lesssim t^\alpha (1 - \log t)^{1/p'} \left(\int_t^1 (u^{-\alpha} (1 - \log u)^{-1} \omega_\alpha (f,u)_{\infty;\T})^p \frac{du}{u} \right)^{1/p}$. On the other hand, consider $g \in L_\infty(\T)$ such that $\omega_\alpha(g,t)_{\infty;\T} \asymp t^\eta$ for some $\eta \in (0, \alpha)$ (cf. \cite{Tikhonovreal}).  It is readily seen that $t^\alpha (1 - \log t)^{1/p'} \left(\int_t^1 (u^{-\alpha} (1 - \log u)^{-1} \omega_\alpha (g,u)_{\infty;\T})^p \frac{du}{u} \right)^{1/p} \lesssim \omega_\alpha (g, t)_{\infty;\T}$.

(ii) The Ulyanov-type inequality (\ref{UlyanovBrezisWaingerSharp<}) with $b=1$, that is,
	 \begin{align}
		\omega_\alpha (f, t)_{\infty;\T} + t^\alpha (1 - \log t)^{1/p'} \left(\int_t^1 (u^{-\alpha} (1 - \log u)^{-1} \omega_\alpha (f,u)_{\infty;\T})^p \frac{du}{u} \right)^{1/p} \nonumber
\\
		& \hspace{-10cm}\lesssim \int_0^{t (1 - \log t)^{1/\alpha p'}} u^{-1/p} \omega_{\alpha + 1/p} (f, u)_{p;\T} \frac{du}{u}
\label{vspom}	
\end{align}
	sharpens the estimate
		\begin{equation*}
		\omega_\alpha (f, t)_{\infty;\T} \lesssim  \int_0^{t (1 - \log t)^{1/\alpha p'}} u^{-1/p} \omega_{\alpha + 1/p} (f, u)_{p;\T} \frac{du}{u}
	\end{equation*}
 given in Theorem \ref{UlyanovBrezisWainger}(ii) with $b = 1/p'$ (see also (\ref{UlyanovTikSharp*})). In particular, (\ref{vspom}) sharpens the Ulyanov inequality (\ref{UlyanovTik*}).
	To be more precise, let
	\begin{equation*}
		I(t) = \omega_\alpha (f, t)_{\infty;\T} + t^\alpha (1 - \log t)^{1/p'} \left(\int_t^1 (u^{-\alpha} (1 - \log u)^{-1} \omega_\alpha (f,u)_{\infty;\T})^p \frac{du}{u} \right)^{1/p}
	\end{equation*}
	and
$
		J(t) = \omega_\alpha (f, t)_{\infty;\T}.
	$
	Obviously, $J(t) \leq I(t)$. Setting $f \in L_\infty(\T)$ with $J(t) = \omega_\alpha (f, t)_{\infty;\T} \asymp t^\alpha (1-\log t)^{1/p'} (1 + \log (1 - \log t))^{-\varepsilon}, \, \varepsilon < 1/p$ (see \cite{Tikhonovreal}), it is easy to check that $I(t) \asymp t^\alpha (1-\log t)^{1/p'} (1 + \log (1 - \log t))^{-\varepsilon + 1/p}$. Thus $J(t)$ and $I(t)$ are not equivalent.
Such an improvement is consistent with the fact that the embedding stated in Theorem \ref{UlyanovBrezisWaingerSharp}(i), i.e,
	 \begin{equation*}
		\dot{H}^{\alpha + 1/p}_p(\T) \hookrightarrow \text{Lip}^{(\alpha, -1)}_{\infty, p}(\T),
	\end{equation*}
	is a refinement of that given in Theorem \ref{UlyanovBrezisWainger}(i)
	\begin{equation*}
		\dot{H}^{\alpha + 1/p}_p(\T) \hookrightarrow \text{Lip}^{(\alpha, -1/p')}_{\infty, \infty}(\T);
	\end{equation*}	
	see Remark \ref{RemThmBWRef1}(ii).
	
	(iii) The sharp norm estimates of the classical Sobolev embeddings for Besov spaces
	\begin{equation}\label{LernerKolyada}
	 \dot{B}^{k - \lambda}_{p,r}(\T) \hookrightarrow \dot{B}^{k - 1/p + 1/q- \lambda}_{q, r}(\T), \quad 1 \leq p < q \leq \infty, \quad 0 < r \leq \infty,  \quad k \in \mathbb{N},
	 \end{equation}
	  as  $\lambda \to 0+$ or $\lambda \to (k - 1/p + 1/q)-$ were settled in \cite{KolyadaLerner} and \cite{Dominguez}. Here both Besov (semi-)norms in (\ref{LernerKolyada}) are defined in terms of the corresponding modulus of smoothness with fixed order $k$.
This is in sharp contrast with (\ref{UlyanovBrezisWaingerSharp**})
where the semi-norms $\vertiii{f}_{B^{\alpha - \lambda, -b}_{\infty, p}(\T), \alpha}$ and $\vertiii{f}_{B^{\alpha + 1/p - \lambda}_{p,p}(\T),  \alpha + 1/p }$ involve the moduli of smoothness of order $\alpha$ and $\alpha + 1/p$, respectively. See (\ref{DefHolZygLog}). Furthermore, it turns out that the  sharp estimates for (\ref{LernerKolyada}) obtained in \cite{KolyadaLerner} and \cite{Dominguez} and those given in (\ref{UlyanovBrezisWaingerSharp**}) with $b=1$ are independent of each other. Indeed, assume $k= \alpha + 1/p \in \mathbb{N}$. According to \cite[Remark 3.3]{Dominguez}, if $q = \infty$ and $r=p$ in (\ref{LernerKolyada}) then
	  	\begin{equation}\label{UlyanovBrezisWaingerSharp**1}
			\|f\|_{\dot{B}^{\alpha - \lambda}_{\infty, p}(\T), k} \leq C \lambda^{1/p}  \|f\|_{\dot{B}^{\alpha + 1/p - \lambda}_{p,p}(\T),k }, \quad \lambda \to 0+.
		\end{equation}
	At the same time,  it follows from (\ref{UlyanovBrezisWaingerSharp**}) that
		\begin{equation}\label{UlyanovBrezisWaingerSharp**2new}
			\vertiii{f}_{B^{\alpha - \lambda, -1}_{\infty, p}(\T), \alpha} \leq C \lambda^{1/p}  \|f\|_{\dot{B}^{\alpha + 1/p - \lambda}_{p,p}(\T), k }, \quad \lambda \to 0+.
		\end{equation}
		Notice that (\ref{UlyanovBrezisWaingerSharp**2new}) is not an immediate consequence of  (\ref{UlyanovBrezisWaingerSharp**1}). To be more precise, invoking Marchaud inequality (\ref{Marchaud}) (noting that $\alpha < k = \alpha + 1/p$) and Hardy's inequality (\ref{HardyIneq}), it is plain to check that
		\begin{equation}\label{UlyanovBrezisWaingerSharp**2}
			\vertiii{f}_{B^{\alpha - \lambda, -1}_{\infty, p}(\T), \alpha} \leq \left(\int_0^\infty t^{-(\alpha - \lambda) p}  \omega_\alpha(f,t)_{\infty;\T}^p \frac{dt}{t} \right)^{1/p} \lesssim \lambda^{-1} \|f\|_{\dot{B}^{\alpha - \lambda}_{\infty, p}(\T),k}
		\end{equation}
		and so, by (\ref{UlyanovBrezisWaingerSharp**1}),
		\begin{equation*}
			\vertiii{f}_{B^{\alpha - \lambda, -1}_{\infty, p}(\T), \alpha} \lesssim \lambda^{1/p-1}  \|f\|_{\dot{B}^{\alpha + 1/p - \lambda}_{p,p}(\T), k}.
		\end{equation*}
		Here, the embedding constant blows up as $\lambda \to 0+$, which is not the case in (\ref{UlyanovBrezisWaingerSharp**2new}).


	(iv) Notice that a remark parallel to Remark \ref{RemUlyanovBrezisWainger}(ii) applies on (\ref{UlyanovBrezisWaingerSharp**}).
	
	(v) The assumption $ b > 1/p$ in Theorem \ref{UlyanovBrezisWaingerSharp} is imposed to avoid trivial spaces. Recall that $\text{Lip}^{(\alpha, -b)}_{\infty, p}(\T) = \{0\}$ if $b \leq 1/p$ (see Section \ref{SectionFunctionSpaces}).
\end{rem}

\begin{proof}[Proof of Theorem \ref{UlyanovBrezisWaingerSharp}]
(i) $\Longrightarrow$ (ii): By (i) and \eqref{6665}, we obtain
\begin{equation}\label{KFunctUlyanovBrezisWaingerSharp}
	K(t, f ; L_\infty(\T) + \infty \dot{H}^\alpha_\infty(\T),  \text{Lip}^{(\alpha, -b)}_{\infty, p}(\T)) \lesssim K(t, f; \dot{B}^{1/p}_{p,1}(\T), \dot{H}^{\alpha + 1/p}_p(\T)).
\end{equation}
We have (see (\ref{8}))
\begin{equation}\label{8*}
		K(t^\alpha,f ; \dot{B}^{1/p}_{p,1}(\T) ,  \dot{H}^{\alpha+1/p}_p(\T)) \asymp \int_0^t u^{-1/p} \omega_{\alpha + 1/p}(f,u)_{p;\T} \frac{du}{u}.
	\end{equation}
	Next we compute $K(t, f ; L_\infty(\T)+ \infty \dot{H}^\alpha_\infty(\T),  \text{Lip}^{(\alpha, -b)}_{\infty, p}(\T))$. Since $\L^{(\alpha,-b)}_{\infty,p}(\T) = (L_\infty(\T)+ \infty \dot{H}^\alpha_\infty(\T), \dot{H}^\alpha_\infty(\T))_{(1,-b),p}$ (see (\ref{LipLimInter}) and \eqref{9099999}), we can apply (\ref{LemmaHolmstedt2}) to establish 
	\begin{align*}
		K(t (1 - \log t)^{b - 1/p}, f ;  L_\infty(\T)+ \infty \dot{H}^\alpha_\infty(\T),  \text{Lip}^{(\alpha, -b)}_{\infty, p}(\T)) \\
		& \hspace{-8.5cm}\asymp K(t (1 - \log t)^{b - 1/p}, f ;  L_\infty(\T)+ \infty \dot{H}^\alpha_\infty(\T),  (L_\infty(\T)+ \infty \dot{H}^\alpha_\infty(\T), \dot{H}^\alpha_\infty(\T))_{(1,-b),p}) \\
		& \hspace{-8.5cm} \asymp K(t,f ; L_\infty(\T)+ \infty \dot{H}^\alpha_\infty(\T), \dot{H}^\alpha_\infty(\T)) \\
		& \hspace{-8cm} + t (1 - \log t)^{b - 1/p} \left(\int_t^1 (u^{-1} (1-\log u)^{-b} K(u,f ; L_\infty(\T)+ \infty \dot{H}^\alpha_\infty(\T), \dot{H}^\alpha_\infty(\T)))^p \frac{du}{u} \right)^{1/p}.
	\end{align*}
	Therefore, by \eqref{9099999} and Lemma \ref{LemmaKfunctLS},
	\begin{align}
		K(t^\alpha (1 - \log t)^{b - 1/p}, f ;  L_\infty(\T)+ \infty \dot{H}^\alpha_\infty(\T),  \text{Lip}^{(\alpha, -b)}_{\infty, p}(\T))& \asymp \omega_\alpha(f,t)_{\infty;\T} \nonumber \\
		& \hspace{-6cm}  + t^\alpha (1 - \log t)^{b - 1/p} \left(\int_t^1 (u^{-\alpha} (1-\log u)^{-b} \omega_\alpha(f,u)_{\infty;\T})^p \frac{du}{u} \right)^{1/p}. \label{UlyanovBrezisWaingerSharp.a}
	\end{align}
	Putting together (\ref{KFunctUlyanovBrezisWaingerSharp}), (\ref{8*}) and (\ref{UlyanovBrezisWaingerSharp.a}), it follows that
	\begin{align*}
		 \omega_\alpha(f,t)_{\infty;\T} + t^\alpha (1 - \log t)^{b - 1/p} \left(\int_t^1 (u^{-\alpha} (1-\log u)^{-b} \omega_\alpha(f,u)_{\infty;\T})^p \frac{du}{u} \right)^{1/p} \\
		 &\hspace{-10cm} \lesssim  \int_0^{t (1 - \log t)^{(b-1/p)/\alpha}} u^{-1/p} \omega_{\alpha + 1/p} (f, u)_{p;\T} \frac{du}{u}.
	\end{align*}
	
	(ii) $\Longrightarrow$ (iii): Firstly, we will show that (ii) implies

	\begin{align}
		 \left(\int_t^1 (u^{-\alpha} (1 - \log u)^{-b} \omega_\alpha (f,u)_{\infty;\T})^p \frac{du}{u} \right)^{1/p} \nonumber \\
		& \hspace{-5cm} \lesssim t^{-\alpha} (1 - \log t)^{-b+1/p} \int_0^t u^{-1/p} (1 - \log u)^{b - 1/p} \omega_{\alpha + 1/p}(f,u)_{p;\T} \frac{du}{u} \label{UlyanovBrezisWaingerSharp.b}
	\end{align}
	for $t \in (0,1)$. Indeed, using monotonicity properties of the moduli of smoothness (see Section \ref{SectionModuli}), we write
	\begin{align*}
		 \int_0^{t (1 - \log t)^{(b-1/p)/\alpha}} u^{-1/p} \omega_{\alpha + 1/p} (f, u)_{p;\T} \frac{du}{u}
&\lesssim \int_0^t u^{-1/p} (1 - \log u)^{b-1/p} \omega_{\alpha + 1/p}(f,u)_{p;\T} \frac{du}{u}
 \\
		 & \hspace{-1cm} + \omega_{\alpha + 1/p}(f,t)_{p;\T} t^{-\alpha - 1/p} \int_t^{t (1 - \log t)^{(b-1/p)/\alpha}} u^\alpha \frac{du}{u} \\
		 & \hspace{-1cm} \lesssim \int_0^t u^{-1/p} (1 - \log u)^{b-1/p} \omega_{\alpha + 1/p}(f,u)_{p;\T} \frac{du}{u}.
	\end{align*}
	Therefore, by (ii) we derive 
	\begin{align*}
		 t^\alpha (1 - \log t)^{b-1/p} \left(\int_t^1 (u^{-\alpha} (1 - \log u)^{-b} \omega_\alpha (f,u)_{\infty;\T})^p \frac{du}{u} \right)^{1/p} \\
		 & \hspace{-9cm} \lesssim \int_0^{t (1 - \log t)^{(b-1/p)/\alpha}} u^{-1/p} \omega_{\alpha + 1/p} (f, u)_{p;\T} \frac{du}{u} \\
		 & \hspace{-9cm} \lesssim \int_0^t u^{-1/p} (1 - \log u)^{b-1/p} \omega_{\alpha + 1/p}(f,u)_{p;\T} \frac{du}{u}.
	\end{align*}
	This yields (\ref{UlyanovBrezisWaingerSharp.b}).
	
	Let $0 < \lambda < \alpha$. Applying H\"older's inequality,

	\begin{align}
\label{UlyanovBrezisWaingerSharp.c}	&\left( \int_0^t u^{-1/p} (1 - \log u)^{b - 1/p} \omega_{\alpha + 1/p}(f,u)_{p;\T} \frac{du}{u}\right)^p  \\
	&
 \lesssim (\alpha - \lambda)^{-p/p'}  t^{(\alpha-\lambda) p/2}  \int_0^t \left(u^{- (\alpha-\lambda)/2-1/p} (1 - \log u)^{b - 1/p} \omega_{\alpha + 1/p}(f,u)_{p;\T} \right)^p \frac{du}{u}. \nonumber
	\end{align}
	In view of (\ref{DefHolZygLog}), (\ref{UlyanovBrezisWaingerSharp.b}) and (\ref{UlyanovBrezisWaingerSharp.c}) and applying Fubini's theorem twice, we have
	\begin{align*}
		\vertiii{f}_{B^{\alpha-\lambda, -b}_{\infty,p}(\T), \alpha}^p & = \int_0^1 u^{-\alpha p + \lambda p} (1 - \log u)^{-b p} \omega_\alpha(f, u)_{\infty;\T}^p \frac{du}{u} \nonumber\\
		&\hspace{-2cm} \asymp \lambda \int_0^1 u^{-\alpha p} (1 - \log u)^{-b p}  \omega_\alpha(f, u)_{\infty;\T}^p  \int_0^u t^{\lambda p} \frac{dt}{t} \frac{du}{u} \nonumber \\
		& \hspace{-2cm}  = \lambda \int_0^1 t^{\lambda p} \int_t^1  u^{-\alpha p} (1 - \log u)^{-b p}  \omega_\alpha(f, u)_{\infty;\T}^p \frac{du}{u} \frac{dt}{t} \nonumber \\
		&\hspace{-2cm} \lesssim \lambda \int_0^1 t^{\lambda p- \alpha p} (1 - \log t)^{-b p + 1}  \left( \int_0^t u^{-1/p} (1 - \log u)^{b - 1/p} \omega_{\alpha + 1/p}(f,u)_{p;\T} \frac{du}{u}\right)^p \frac{dt}{t} \\
		& \hspace{-2cm} \lesssim \lambda (\alpha - \lambda)^{-p/p'} \int_0^1 t^{(\lambda - \alpha + (\alpha-\lambda)/2 )p} (1 - \log t)^{-b p + 1} \\
		&\hspace{2cm} \int_0^t \left(u^{- (\alpha-\lambda)/2-1/p} (1 - \log u)^{b - 1/p} \omega_{\alpha + 1/p}(f,u)_{p;\T} \right)^p \frac{du}{u} \frac{dt}{t} \\
		& \hspace{-2cm} = \lambda (\alpha - \lambda)^{-p/p'} \int_0^1 \left(u^{- (\alpha-\lambda)/2-1/p} (1 - \log u)^{b - 1/p} \omega_{\alpha + 1/p}(f,u)_{p;\T} \right)^p \\
		& \hspace{2cm} \int_u^1 t^{-( \alpha - \lambda)p/2} (1 - \log t)^{-b p + 1} \frac{dt}{t} \frac{du}{u} \\
		& \hspace{-2cm} \asymp  \lambda (\alpha - \lambda)^{-p/p' + b p - 2} \int_0^1 \left(u^{- (\alpha-\lambda)/2-1/p} (1 - \log u)^{b - 1/p} \omega_{\alpha + 1/p}(f,u)_{p;\T} \right)^p \\
		& \hspace{2cm} \int_{u^{\alpha-\lambda}}^1 t^{-p/2} (- \log t)^{-b p + 1} \frac{dt}{t} \frac{du}{u} \\
		& \hspace{-2cm} \lesssim  \lambda (\alpha - \lambda)^{-p/p' - 1} \int_0^1 \left(u^{- \alpha+\lambda-1/p}  \omega_{\alpha + 1/p}(f,u)_{p;\T} \right)^p  \frac{du}{u} \asymp \lambda \vertiii{f}_{B^{\alpha + 1/p - \lambda}_{p,p}(\T), \alpha + 1/p}^p,
	\end{align*}
	where the last estimate follows from the fact that $ (\alpha - \lambda)^{-p/p' - 1}$ is uniformly bounded with $\lambda \to 0+$.
	
	(iii) $\Longrightarrow$ (i): We claim that there is a positive constant $C$, which does not depend on $\lambda$, satisfying
	\begin{equation}\label{UlyanovBrezisWaingerSharp.d}
		\vertiii{f}_{B^{\alpha + 1/p - \lambda}_{p,p}(\T), \alpha + 1/p} \leq C \lambda^{-1/p} \|f\|_{\dot{H}^{\alpha + 1/p}_p(\T)}, \quad \lambda \to 0+.
	\end{equation}
	Indeed, we have
	\begin{align*}
		\vertiii{f}_{B^{\alpha + 1/p - \lambda}_{p,p}(\T), \alpha + 1/p} &=  \left(\int_0^1 t^{-(\alpha + 1/p -\lambda)p} \omega_{\alpha + 1/p}(f,t)_{p;\T}^p \frac{dt}{t} \right)^{1/p}  \\
		& \leq  \left(\int_0^1 t^{\lambda p} \frac{dt}{t} \right)^{1/p} \sup_{0 < t < 1} t^{-(\alpha + 1/p)} \omega_{\alpha + 1/p}(f,t)_{p;\T} \\
		& \lesssim \lambda^{-1/p} \|f\|_{\dot{H}^{\alpha + 1/p}_p(\T)},
	\end{align*}
	where we have used Lemma \ref{Lemma3.6} in the last estimate.
	
	According to (iii) and (\ref{UlyanovBrezisWaingerSharp.d}), we obtain
	\begin{equation*}
		\vertiii{f}_{B^{\alpha - \lambda, -b}_{\infty, p}(\T), \alpha} \lesssim \lambda^{1/p}  \vertiii{f}_{B^{\alpha + 1/p - \lambda}_{p,p}(\T), \alpha + 1/p } \lesssim \|f\|_{\dot{H}^{\alpha + 1/p}_p(\T)},
	\end{equation*}
	where the hidden constant is independent of $\lambda$. Then the embedding given in (i) follows by passing to the limit $\lambda \to 0+$ and applying the monotone convergence theorem (see (\ref{DefLipschitz}) and (\ref{DefHolZygLog})).
	
	The equivalence (i) $\iff$ (iv) was already shown in Theorem \ref{ThmBWRef1}.
	
%
\end{proof}

\begin{rem}
Repeating the proof of the implication (i) $\Longrightarrow$ (ii) in Theorem \ref{UlyanovBrezisWaingerSharp} line by line but now using \eqref{RemarkNew1}, one can establish the multivariate counterpart of \eqref{vspom}. Namely, if $1 < p < \infty$ and $\alpha > 0$ then
	\begin{align*}
		\omega_\alpha (f, t)_{\infty;\T^d} + t^\alpha (1 - \log t)^{1/p'} \left(\int_t^1 (u^{-\alpha} (1 - \log u)^{-1} \omega_\alpha (f,u)_{\infty;\T^d})^p \frac{du}{u} \right)^{1/p} \nonumber\\
		& \hspace{-10cm}\lesssim \int_0^{t (1 - \log t)^{1/\alpha p'}} u^{-d/p} \omega_{\alpha + d/p} (f, u)_{p;\T^d} \frac{du}{u}, \quad f \in L_p(\T^d),
	\end{align*}
	whenever the right-hand side is finite. The corresponding inequality for functions $f \in L_p(\R^d)$ also holds true.
\end{rem}

\begin{rem}
	The Ulyanov-type inequality (\ref{UlyanovBrezisWaingerSharp<}) (with $b=1$) obtained in Theorem \ref{UlyanovBrezisWaingerSharp} is optimal in the following senses
	\begin{align*}
		\omega_\alpha (f, t)_{\infty;\T} + t^\alpha (1 - \log t)^{1/p'} \left(\int_t^1 (u^{-\alpha} (1 - \log u)^{-1} \omega_\alpha (f,u)_{\infty;\T})^p \frac{du}{u} \right)^{1/p} \\
		& \hspace{-10cm}\lesssim \left(\int_0^{t (1 - \log t)^{1/\alpha p'}} (u^{-1/p} \omega_{\alpha + 1/p} (f, u)_{p;\T})^q \frac{du}{u} \right)^{1/q} \iff q \leq 1,
	\end{align*}
	and
	\begin{align*}
		\omega_\alpha (f, t)_{\infty;\T} + t^\alpha (1 - \log t)^{1/p'} \left(\int_t^1 (u^{-\alpha} (1 - \log u)^{-1} \omega_\alpha (f,u)_{\infty;\T})^p \frac{du}{u} \right)^{1/p} \\
		& \hspace{-10cm}\lesssim \int_0^{t (1 - \log t)^{1/\alpha p'}} u^{-1/p} (1-\log u)^b \omega_{\alpha + 1/p} (f, u)_{p;\T} \frac{du}{u} \iff b \geq 0.
	\end{align*}
	The proofs of these assertions proceed in complete analogy with those given to show (\ref{SharpnessAssUT}) and (\ref{SharpnessAssUT2}) and they are left to the reader.
\end{rem}

\subsection{Embeddings of Jawerth-Franke}
The Jawerth-Franke embeddings establish relations between Besov spaces and Triebel-Lizorkin spaces with different metrics. See \cite{Jawerth} and \cite{Franke} (cf. also \cite{Marschall} and \cite{Vybiral}). In particular, for the Sobolev spaces the result reads as follows.

\begin{thm}\label{ThmFrankeJawerthRecall}
	Let $1 \leq p_0 < p < p_1 \leq \infty$ and $\alpha \geq 0$. Then
\begin{equation}\label{FrankeMarschall}
		\dot{B}^{\alpha + d(1/p_0 -1/p)}_{p_0,p}(\T^d) \hookrightarrow \dot{H}^{\alpha}_p(\T^d) \hookrightarrow \dot{B}^{\alpha + d(1/p_1 - 1/p)}_{p_1,p}(\T^d).
	\end{equation}
	The previous embeddings also hold true for function spaces over $\R^d$.
\end{thm}

Note that working with Fourier-analytically defined function spaces, these embeddings can be extended to any $\alpha \in \R$.

Applying the relationships between Lipschitz spaces and Besov spaces obtained in \cite{DominguezTikhonov19}, we observe that the Br\'ezis-Wainger-type embeddings (cf. Theorem \ref{ThmBWRef1} and Remark \ref{Remark6.7}) can be strengthened by the Jawerth-Franke embeddings. To be more precise, it follows from
\begin{equation*}
\dot{B}^\alpha_{\infty,p}(\T^d) \hookrightarrow \text{Lip}^{(\alpha, -1)}_{\infty, p}(\T^d), \quad \alpha > 0, \quad 1 < p < \infty,
\end{equation*}
that
\begin{equation}\label{BWVSFJ2}
		\dot{H}^{\alpha + d/p}_p(\T^d) \hookrightarrow \dot{B}^\alpha_{\infty,p}(\T^d)
	\end{equation}	
consists of an improvement of
\begin{equation*}
		\dot{H}^{\alpha + d/p}_p(\T^d) \hookrightarrow \text{Lip}^{(\alpha, -1)}_{\infty, p}(\T^d).
	\end{equation*}

	Our next goal is to study the estimates in terms of the moduli of smoothness and extrapolation inequalities related to (\ref{BWVSFJ2}).

\begin{thm}\label{ThmJawerth}
	Let $1 <p < \infty, \alpha > 0$ and $0 < q \leq \infty$. The following statements are equivalent:
	\begin{enumerate}[\upshape(i)]
	\item \begin{equation*}
		\dot{H}^{\alpha + d/p}_{p}(\T^d) \hookrightarrow \dot{B}^\alpha_{\infty,q}(\T^d),
	\end{equation*}
		\item for $f \in \dot{B}^{d/p}_{p,1}(\T^d)$ and $t > 0$, we have
		 \begin{equation}\label{ThmJawerth1}
		t^\alpha \left(\int_t^\infty (u^{-\alpha} \omega_{\alpha + d/p} (f,u)_{\infty;\T^d})^q \frac{du}{u} \right)^{1/q} \lesssim \int_0^{t} u^{-d/p} \omega_{\alpha + d/p} (f, u)_{p;\T^d} \frac{du}{u},
	\end{equation}
	\item if  $\lambda \to 0+$ then
	\begin{equation*}
	\dot{B}^{\alpha + d/p - \lambda}_{p,q}(\T^d) \hookrightarrow \dot{B}^{\alpha - \lambda}_{\infty, q}(\mathbb{T}^d)
	\end{equation*}
	 with norm $\mathcal{O}(\lambda^{1/q})$, i.e., there exists $C > 0$, which is independent of $\lambda$, such that
		\begin{equation}\label{ThmJawerth1*}
			\|f\|_{\dot{B}^{\alpha - \lambda}_{\infty, q}(\T^d), \alpha + d/p} \leq C \lambda^{1/q}  \|f\|_{\dot{B}^{\alpha + d/p - \lambda}_{p,q}(\T^d), \alpha + d/p},
		\end{equation}
	\item \begin{equation*}
	q \geq p.
	\end{equation*}
	\end{enumerate}
	The corresponding result also holds true for $\R^d$.
\end{thm}

\begin{rem}
(i)
Inequality (\ref{ThmJawerth1}) with $q=p$ is actually the fractional counterpart of the following Kolyada's inequality \cite{Kolyada2} (see also \cite{Netrusov}):
$$
		t^{k-d/p} \left(\int_t^\infty (u^{d/p-k} \omega_{k} (f,u)_{\infty;\R^d})^p \frac{du}{u} \right)^{1/p} \lesssim \int_0^{t} u^{-d/p} \omega_k (f, u)_{p;\R^d} \frac{du}{u},\qquad k > d/p.
$$
	
	(ii) Let $d=1$. Inequality (\ref{ThmJawerth1}) with $q=p$ implies
	\begin{equation}\label{KolyadaJaw2}
		t^\alpha \left(\int_t^1 (u^{-\alpha} \omega_{\alpha + 1/p} (f,u)_{\infty;\T})^p \frac{du}{u} \right)^{1/p} \lesssim \int_0^{t} u^{-1/p} \omega_{\alpha + 1/p} (f, u)_{p;\T} \frac{du}{u}, \quad t \in (0,1).
	\end{equation}
	This is stronger than (\ref{UlyanovBrezisWaingerSharp<}) with $b=1$ given by
	 \begin{align}
		\omega_\alpha (f, t)_{\infty;\T} + t^\alpha (1 - \log t)^{1/p'} \left(\int_t^1 (u^{-\alpha} (1 - \log u)^{-1} \omega_\alpha (f,u)_{\infty;\T})^p \frac{du}{u} \right)^{1/p} \nonumber \\
		& \hspace{-10cm}\lesssim \int_0^{t (1 - \log t)^{1/\alpha p'}} u^{-1/p} \omega_{\alpha + 1/p} (f, u)_{p;\T} \frac{du}{u}. \label{UlyanovBrezisWaingerSharp<*}
	\end{align}
	Indeed, assume that (\ref{KolyadaJaw2}) holds true. Then, applying the Marchaud inequality (\ref{Marchaud}) together with Hardy's inequality (\ref{HardyInequal4}), we derive
	\begin{align*}
		 \left(\int_{t (1-\log t)^{1/\alpha p'}}^1 (u^{-\alpha} (1 - \log u)^{-1} \omega_\alpha (f,u)_{\infty;\T})^p \frac{du}{u} \right)^{1/p} \\
		 & \hspace{-7cm} \lesssim \left(\int_{t (1-\log t)^{1/\alpha p'}}^1 \left( (1-\log u)^{-1} \int_u^1 \frac{\omega_{\alpha + 1/p}(f,v)_{\infty;\T}}{v^\alpha} \frac{dv}{v}\right)^p \frac{du}{u} \right)^{1/p} \\
		 & \hspace{-7cm} \lesssim \left( \int_{t (1-\log t)^{1/\alpha p'}}^1 (u^{-\alpha} \omega_{\alpha + 1/p}(f,u)_{\infty;\T})^p \frac{du}{u}\right)^{1/p}.
	\end{align*}
	Hence, it follows from (\ref{KolyadaJaw2}) that
	\begin{align}
	t^{\alpha} (1-\log t)^{1/p'} \left(\int_{t (1-\log t)^{1/\alpha p'}}^1 (u^{-\alpha} (1 - \log u)^{-1} \omega_\alpha (f,u)_{\infty;\T})^p \frac{du}{u} \right)^{1/p} \nonumber \\
	  & \hspace{-10cm} \lesssim  \int_0^{t(1-\log t)^{1/\alpha p'}} u^{-1/p} \omega_{\alpha + 1/p} (f, u)_{p;\T} \frac{du}{u}. \label{KolyadaJaw3}
	\end{align}
	On the other hand, invoking again (\ref{Marchaud}), we have
	\begin{equation}\label{KolyadaJaw4}
		\omega_\alpha(f,t)_{\infty;\T}  \lesssim  t^\alpha \int_{t}^1 u^{-\alpha} \omega_{\alpha + 1/p}(f,u)_{\infty;\T} \frac{du}{u} =  I + II,
	\end{equation}
	where
	\begin{equation*}
		I := t^\alpha \int_{t}^{t (1-\log t)^{1/\alpha p'} } u^{-\alpha} \omega_{\alpha + 1/p}(f,u)_{\infty;\T} \frac{du}{u}
	\end{equation*}
	and
	\begin{equation*}
	 II := t^\alpha \int_{t (1-\log t)^{1/\alpha p'} }^1 u^{-\alpha} \omega_{\alpha + 1/p}(f,u)_{\infty;\T} \frac{du}{u}.
	\end{equation*}
	By H\"older's inequality, we obtain 
\begin{align}
	II & \leq t^\alpha  \left(\int_{t (1-\log t)^{1/\alpha p'} }^1 (u^{-\alpha} \omega_{\alpha + 1/p}(f,u)_{\infty;\T})^p \frac{du}{u}\right)^{1/p} \left(\int_{t (1-\log t)^{1/\alpha p'}}^1 \frac{du}{u}\right)^{1/p'} \nonumber \\
	& \asymp t^\alpha (1-\log t)^{1/p'} \left(\int_{t (1-\log t)^{1/\alpha p'} }^1 (u^{-\alpha} \omega_{\alpha + 1/p}(f,u)_{\infty;\T})^p \frac{du}{u}\right)^{1/p} \nonumber \\
	& \lesssim \int_0^{t (1-\log t)^{1/\alpha p'} } u^{-1/p} \omega_{\alpha + 1/p} (f, u)_{p;\T} \frac{du}{u} \label{KolyadaJaw5}
\end{align}
where we have used (\ref{KolyadaJaw2}) in the last step.

Next we estimate $I$. Using the monotonicity properties of the moduli of smoothness (see Section \ref{SectionModuli}),
\begin{align}
	I &\leq t^\alpha \omega_{\alpha + 1/p}(f,t(1-\log t)^{1/\alpha p'})_{\infty;\T}  \int_{t}^{t (1-\log t)^{1/\alpha p'} } u^{-\alpha}  \frac{du}{u} \nonumber \\
	& \lesssim  \omega_{\alpha + 1/p}(f,t(1-\log t)^{1/\alpha p'})_{\infty;\T} \lesssim \frac{ \omega_{\alpha + 1/p}(f,t(1-\log t)^{1/\alpha p'})_{\infty;\T} }{(t (1-\log t)^{1/\alpha p'})^{1/p}}  \nonumber  \\
	& \asymp  \frac{ \omega_{\alpha + 1/p}(f,t(1-\log t)^{1/\alpha p'})_{\infty;\T} }{(t (1-\log t)^{1/\alpha p'})^{\alpha + 1/p}} \int_0^{t(1-\log t)^{1/\alpha p'}} u^\alpha \frac{du}{u}  \nonumber  \\
	& \lesssim \int_0^{t(1-\log t)^{1/\alpha p'}} u^{-1/p} \omega_{\alpha + 1/p}(f,u)_{p;\T} \frac{du}{u}. \label{KolyadaJaw6}
\end{align}
Therefore, (\ref{KolyadaJaw4})--(\ref{KolyadaJaw6}) yield that
\begin{equation} \label{KolyadaJaw7}
	\omega_\alpha(f,t)_{\infty;\T} \lesssim  \int_0^{t(1-\log t)^{1/\alpha p'}} u^{-1/p} \omega_{\alpha + 1/p}(f,u)_{p;\T} \frac{du}{u}.
\end{equation}
Moreover,  elementary computations lead to
\begin{align}
t^{\alpha} (1-\log t)^{1/p'} \left(\int_{t}^{t (1-\log t)^{1/\alpha p'}} (u^{-\alpha} (1 - \log u)^{-1} \omega_\alpha (f,u)_{\infty;\T})^p \frac{du}{u} \right)^{1/p} \nonumber \\
&\hspace{-10cm} \lesssim (1-\log t)^{1/p'} \omega_\alpha(f,t)_{\infty;\T} \left(\int_{t}^{t (1-\log t)^{1/\alpha p'}} (1-\log u)^{-p} \frac{du}{u} \right)^{1/p} \nonumber \\
& \hspace{-10cm} \lesssim \omega_\alpha(f,t)_{\infty;\T}. \label{KolyadaJaw8}
\end{align}

Finally, putting together (\ref{KolyadaJaw3}), (\ref{KolyadaJaw7}) and (\ref{KolyadaJaw8}), we arrive at (\ref{UlyanovBrezisWaingerSharp<*}).

(iii) Another proof of (\ref{ThmJawerth1*}) with $\alpha + d/p \in \mathbb{N}$ was obtained in \cite[Remark 3.3]{Dominguez}. See also Remark \ref{RemarkUlyanovBrezisWaingerSharp}(iii).

\end{rem}

\begin{proof}[Proof of Theorem \ref{ThmJawerth}]
	(i) $\Longrightarrow$ (ii): According to (i) and \eqref{6665}, we have
	\begin{equation}\label{ThmJawerth2}
	K(t^\alpha, f ; L_\infty(\T^d) + \infty \dot{H}^{\alpha + d/p}_p(\T^d) ,  \dot{B}^\alpha_{\infty, q}(\mathbb{T}^d)) \lesssim K(t^\alpha, f; \dot{B}^{d/p}_{p,1}(\T^d), \dot{H}^{\alpha + d/p}_p(\T^d)).
\end{equation}
It turns out that
\begin{equation}\label{ThmJawerth3}
		K(t^\alpha,f ; B^{d/p}_{p,1}(\T^d) ,  \dot{H}^{\alpha+d/p}_p(\T^d)) \asymp \int_0^t u^{-d/p} \omega_{\alpha + d/p}(f,u)_{p;\T^d} \frac{du}{u},
	\end{equation}
	see (\ref{8*}).

	
	Next we estimate $K(t, f ; L_\infty(\T^d) + \infty\dot{H}^{\alpha+d/p}_p(\T^d) ,  \dot{B}^\alpha_{\infty, q}(\mathbb{T}^d))$. Since
	\begin{align*}
		\dot{B}^\alpha_{\infty,q}(\T^d) &= (L_\infty(\T^d), \dot{H}^{\alpha + d/p}_\infty(\T^d))_{\frac{\alpha p }{\alpha p + d}, q}  \\
		&= (L_\infty(\T^d) + \infty \dot{H}^{\alpha + d/p}_\infty(\T^d) , \dot{H}^{\alpha + d/p}_\infty(\T^d))_{\frac{\alpha p }{\alpha p + d}, q}
	\end{align*}
	(see (\ref{BInter}) and \eqref{9099999}), we can invoke (\ref{LemmaHolmstedt1}) together with Lemma \ref{LemmaKfunctLS} to establish
	\begin{align}
		K(t^\alpha,f;  L_\infty(\T^d) + \infty\dot{H}^{\alpha+d/p}_p(\T^d),  \dot{B}^\alpha_{\infty, q}(\mathbb{T}^d))  \nonumber\\
		& \hspace{-6cm} \asymp t^\alpha \left(\int_t^\infty (u^{-\alpha} K(u^{\alpha + d/p},f; L_\infty(\T^d)+ \infty\dot{H}^{\alpha+d/p}_p(\T^d), \dot{H}^{\alpha + d/p}_\infty(\T^d)) )^q \frac{du}{u} \right)^{1/q}  \nonumber\\\
		& \hspace{-6cm} \asymp t^\alpha	\left( \int_t^\infty (u^{-\alpha} \omega_{\alpha + d/p}(f,u)_{\infty;\T^d})^q \frac{du}{u} \right)^{1/q}. \label{ThmJawerth4}
	\end{align}
	
	Collecting (\ref{ThmJawerth2}), (\ref{ThmJawerth3}) and (\ref{ThmJawerth4}), we obtain
	\begin{equation*}
		t^\alpha  \left(\int_{t}^\infty (u^{-\alpha} \omega_{\alpha + d/p}(f,u)_{\infty;\T^d})^q \frac{du}{u} \right)^{1/q} \lesssim \int_0^t u^{-d/p} \omega_{\alpha + d/p}(f,u)_{p;\T^d} \frac{du}{u}.
	\end{equation*}
	
	(ii) $\Longrightarrow$ (iii): Let $\lambda > 0$. Applying Fubini's theorem, we have
	\begin{align*}
		\|f\|_{\dot{B}^{\alpha - \lambda}_{\infty, q}(\mathbb{T}^d), \alpha + d/p} & = \left(\int_0^\infty u^{-(\alpha - \lambda) q} \omega_{\alpha + d/p}(f,u)_{\infty;\T^d}^q \frac{du}{u} \right)^{1/q} \\
		& \asymp \lambda^{1/q} \left(\int_0^\infty u^{-\alpha q} \omega_{\alpha + d/p}(f,u)_{\infty;\T^d}^q \int_0^u t^{\lambda q} \frac{dt}{t} \frac{du}{u} \right)^{1/q} \\
		& = \lambda^{1/q} \left(\int_0^\infty t^{\lambda q} \int_t^\infty u^{-\alpha q} \omega_{\alpha + d/p}(f,u)_{\infty;\T^d}^q \frac{du}{u} \frac{dt}{t} \right)^{1/q}.
	\end{align*}
	Therefore, in virtue of (\ref{ThmJawerth1}) we arrive at
	\begin{equation}\label{ThmJawerth5}
	\|f\|_{\dot{B}^{\alpha - \lambda}_{\infty, q}(\mathbb{T}^d), \alpha + d/p}  \lesssim  \lambda^{1/q} \left(\int_0^\infty t^{-(\alpha - \lambda) q} \left(\int_0^t u^{-d/p} \omega_{\alpha + d/p}(f,u)_{p;\T^d} \frac{du}{u} \right)^q \frac{dt}{t} \right)^{1/q}.
	\end{equation}
	
	We distinguish two possible cases. Firstly, assume that $q \geq 1$. Then, noting that $\lambda \to 0+$, it follows from (\ref{HardyIneq*}) that
	\begin{align}
		\left(\int_0^\infty t^{-(\alpha - \lambda) q} \left(\int_0^t u^{-d/p} \omega_{\alpha + d/p}(f,u)_{p;\T^d} \frac{du}{u} \right)^q \frac{dt}{t} \right)^{1/q} \nonumber\\
		& \hspace{-6cm}\lesssim \left(\int_0^\infty (t^{-(\alpha + d/p - \lambda)} \omega_{\alpha + d/p}(f,t)_{p;\T^d})^q \frac{dt}{t} \right)^{1/q} \label{ThmJawerth6}
	\end{align}
	uniformly with respect to $\lambda$.
	
	Let $q < 1$. Notice that we may assume without loss of generality that $\alpha > \lambda$ because $\lambda \to 0+$. Then applying monotonicity properties and Fubini's theorem, we have
	\begin{align}
	\int_0^\infty t^{-(\alpha - \lambda) q} \left(\int_0^t u^{-d/p} \omega_{\alpha + d/p}(f,u)_{p;\T^d} \frac{du}{u} \right)^q \frac{dt}{t} \nonumber \\
	& \hspace{-6cm} \asymp \sum_{j=-\infty}^\infty 2^{j (\alpha -\lambda) q} \left(\sum_{i=j}^\infty 2^{i d/p} \omega_{\alpha + d/p}(f, 2^{-i})_{p;\T^d} \right)^q \nonumber \\
	&\hspace{-6cm} \leq  \sum_{j=-\infty}^\infty 2^{j(\alpha -\lambda) q} \sum_{i=j}^\infty (2^{i d/p} \omega_{\alpha + d/p}(f, 2^{-i})_{p;\T^d})^q \nonumber \\
	& \hspace{-6cm} \lesssim \frac{1}{2^{(\alpha - \lambda) q} - 1}  \sum_{i=-\infty}^\infty (2^{i (\alpha+ d/p - \lambda)} \omega_{\alpha + d/p}(f, 2^{-i})_{p;\T^d})^q \nonumber \\
	& \hspace{-6cm} \asymp \int_0^\infty (t^{-(\alpha + d/p - \lambda)} \omega_{\alpha + d/p}(f,t)_{p;\T^d})^q \frac{dt}{t} \nonumber
	\end{align}
	since $\lambda \to 0+$.
Thus (\ref{ThmJawerth6}) holds for $q > 0$.

	Inserting this estimate into (\ref{ThmJawerth5}), we derive 
	 \begin{align*}
	 	\|f\|_{\dot{B}^{\alpha - \lambda}_{\infty, q}(\mathbb{T}^d), \alpha + d/p} & \lesssim  \lambda^{1/q}  \left(\int_0^\infty (t^{-(\alpha + d/p - \lambda)} \omega_{\alpha + d/p}(f,t)_{p;\T^d})^q \frac{dt}{t} \right)^{1/q}  \\
		&= \lambda^{1/q} \|f\|_{\dot{B}^{\alpha + d/p -\lambda}_{p,q}(\T^d), \alpha + d/p}.
	 \end{align*}
	
	 (iii) $\Longrightarrow$ (i): By Lemma \ref{LemmaKfunctLS}, one has
	 \begin{equation*}
	 \|f\|_{\dot{B}^{\alpha + d/p - \lambda}_{p, q}(\mathbb{T}^d), \alpha + d/p} \asymp \|f\|_{(L_p(\T^d), \dot{H}^{\alpha + d/p}_p(\T^d))_{1-\lambda (\alpha+d/p)^{-1}, q}}
	 \end{equation*}
	 with uniform constants with respect to $\lambda$. Applying now \cite[Theorem 1]{Milman05} (cf. also \cite[Theorem 1]{KaradzhovMilmanXiao}), Lemma \ref{LemmaKfunctLS}, monotonicity properties of the moduli of smoothness and (\ref{SobolevModuliDPer'}), we obtain
	 \begin{align*}
	 	\lim_{\lambda \to 0+} \lambda^{1/q} \|f\|_{\dot{B}^{\alpha + d/p - \lambda}_{p, q}(\mathbb{T}^d), \alpha + d/p} & \asymp \lim_{\lambda \to 0+} \lambda^{1/q}  \|f\|_{(L_p(\T^d), \dot{H}^{\alpha + d/p}_p(\T^d))_{1-\lambda (\alpha+d/p)^{-1}, q}} \\
		& \hspace{-4cm}= \lim_{t \to 0+} \frac{K(t,f; L_p(\T^d), \dot{H}^{\alpha + d/p}_p(\T^d))}{t} \asymp \sup_{t >0} \frac{\omega_{\alpha+d/p}(f,t)_{p;\T^d}}{t^{\alpha+d/p}} \asymp \|f\|_{\dot{H}^{\alpha+d/p}_p(\T^d)}.
	 \end{align*}
	 Thus the desired embedding (i) follows by taking limits as $\lambda \to 0+$ on both sides of \eqref{ThmJawerth1*}  and applying the monotone convergence theorem.
	 
	
	
	The equivalence between (i) and (iv) is well known. See \cite[Theorem 3.2.1]{SickelTriebel}. In this reference only inhomogeneous spaces are considered. However, the arguments carry over to homogeneous spaces.
	
	The same method of proof also works when $\T^d$ is replaced by $\R^d$.
\end{proof}

\begin{rem}
	The method of proof of Theorem \ref{ThmJawerth} can also be applied to show that the sharp estimate (\ref{ThmJawerth1*}) also holds when the Besov semi-norms $\|\cdot \|_{\dot{B}^{\alpha - \lambda}_{\infty, q}(\T^d), \alpha + d/p}$ and $\|\cdot \|_{\dot{B}^{\alpha + d/p - \lambda}_{p,q}(\T^d), \alpha + d/p}$ are replaced by $\vertiii{\cdot}_{B^{\alpha - \lambda}_{\infty, q}(\T^d), \alpha + d/p}$ and $\vertiii{\cdot}_{B^{\alpha + d/p - \lambda}_{p,q}(\T^d), \alpha + d/p}$, respectively. Further details are left to the reader.
\end{rem}

\subsection{Embeddings between Lipschitz spaces}
Our next goal is to provide additional insights on the Ulyanov inequality (\ref{UlyanovTik}) with $p=1$, i.e.,
	\begin{equation}\label{UlyanovTikNew}
		\omega_\alpha (f, t)_{q;\T} \lesssim \left( \int_0^t (u^{-1/q'} (1 - \log u)^{1/q} \omega_{\alpha + 1/q'} (f, u)_{1;\T})^{q} \frac{du}{u}\right)^{1/q}, \quad 1 < q < \infty.
	\end{equation}
	To establish  this we will rely on the following embeddings between logarithmic Lipschitz spaces, which were recently obtained in \cite{DominguezTikhonov19}.
	
	\begin{lem}\label{ThmEmbLipschitz}
	Let $\alpha > 0$ and $1 < q < \infty$. Then we have
	\begin{equation}\label{ThmEmbLipschitz*alpha}
	\emph{Lip}^{(\alpha + d/q',0)}_{1,\infty}(\T^d) \hookrightarrow \emph{Lip}^{(\alpha, -1/q)}_{q, \infty}(\T^d).
	\end{equation}
	The corresponding embedding for Lipschitz spaces on $\R^d$ also holds true. In addition, if $d=1$ and $b \geq 0$ then
	\begin{equation}\label{ThmEmbLipschitz*}
		\emph{Lip}^{(\alpha + 1/q',0)}_{1,\infty}(\T) \hookrightarrow \emph{Lip}^{(\alpha, -b)}_{q, \infty}(\T) \iff b \geq 1/q.
	\end{equation}
\end{lem}

\begin{thm}\label{UlyanovBrezisWainger2}
	Let $1 <q < \infty, \alpha > 0$ and $b \geq 0$. The following statements are equivalent:
	\begin{enumerate}[\upshape(i)]
	\item \begin{equation*}
		\emph{Lip}^{(\alpha + 1/q',0)}_{1,\infty}(\T) \hookrightarrow \emph{Lip}^{(\alpha,-b)}_{q, \infty}(\T),
	\end{equation*}
	\item for $f \in \dot{B}^{1/q'}_{1,q}(\T)$ and $t \in (0,1)$, we have
		\begin{equation}\label{UlyanovBrezisWainger2Sharp}
		\omega_\alpha (f, t)_{q;\T} \lesssim  \Big(\int_0^{t (1 - \log t)^{b/\alpha}} (u^{-1/q'} \omega_{\alpha + 1/q'} (f, u)_{1;\T})^q \frac{du}{u} \Big)^{1/q},
	\end{equation}
	\item if $\alpha_0 \to \alpha -$ then
		\begin{equation*}
		\emph{Lip}^{(\alpha + 1/q',0)}_{1,\infty}(\T)  \hookrightarrow \dot{B}^{\alpha_0}_{q, \infty}(\T)
		\end{equation*}
		 with norm $\mathcal{O}((\alpha - \alpha_0)^{-b})$, i.e., there exists $C > 0$, which is independent of $\alpha_0$, such that
		\begin{equation}\label{UlyanovBrezisWaingerSharp**<<<}
		\vertiii{f}_{B^{\alpha_0}_{q, \infty}(\T), \alpha} \leq C (\alpha- \alpha_0)^{-b} \|f\|_{\emph{Lip}^{(\alpha + 1/q',0)}_{1,\infty}(\T)},
		\end{equation}
	\item \begin{equation*}
		b \geq 1/q.
	\end{equation*}
	\end{enumerate}
\end{thm}

\begin{rem}
	(i) Inequality (\ref{UlyanovBrezisWainger2Sharp}) with $b=1/q$, i.e.,
	\begin{equation}\label{UlyanovTikNewFirst}
		\omega_\alpha (f, t)_{q;\T} \lesssim  \Big(\int_0^{t (1 - \log t)^{1/\alpha q}} (u^{-1/q'} \omega_{\alpha + 1/q'} (f, u)_{1;\T})^q \frac{du}{u} \Big)^{1/q}
	\end{equation}
 improves the known estimate (\ref{UlyanovTikNew}), i.e.,
	\begin{equation}\label{UlyanovTikNewSecond}
		\omega_\alpha (f, t)_{q;\T} \lesssim \left( \int_0^t (u^{-1/q'} (1 - \log u)^{1/q} \omega_{\alpha + 1/q'} (f, u)_{1;\T})^{q} \frac{du}{u}\right)^{1/q}.
	\end{equation}
Indeed,
 basic properties of the moduli of smoothness (cf. Section \ref{SectionModuli}) allow us to obtain
	\begin{align}
		 \left(\int_t^{t (1 - \log t)^{1/\alpha q}} (u^{-1/q'} \omega_{\alpha + 1/q'} (f, u)_{1;\T})^q \frac{du}{u}\right)^{1/q} \nonumber\\
		 & \hspace{-5cm}\lesssim  t^{-\alpha -1/q'} \omega_{\alpha+1/q'}(f,t)_{1;\T} \left(\int_t^{t (1-\log t)^{1/\alpha q}} u^{\alpha q} \frac{du}{u}\right)^{1/q} \nonumber \\
		& \hspace{-5cm} \lesssim  \left(\int_0^t (u^{-1/q'} (1-\log u)^{1/q}  \omega_{\alpha+1/q'} (f,u)_{1;\T})^q \frac{du}{u} \right)^{1/q}. \nonumber
	\end{align}
Thus
the right-hand side of (\ref{UlyanovTikNewFirst}) is dominated by the right-hand side of (\ref{UlyanovTikNewSecond}).	
We also observe that (\ref{UlyanovTikNewFirst}) holds true for any $f \in B^{1/q'}_{1,q}(\T)$ unlike (\ref{UlyanovTikNewSecond}), which can be applied for $f \in B^{1/q', 1/q}_{1,q}(\T)$.
 Note that $B^{1/q', 1/q}_{1,q}(\T)  \subsetneq B^{1/q'}_{1,q}(\T)$. 

 (ii) The higher-dimensional version of inequality \eqref{UlyanovTikNewFirst} also holds. See Remark \ref{RemLast2} below.
	
	(iii) Setting $\alpha = 1/q$ in (\ref{UlyanovBrezisWaingerSharp**<<<}), we deduce
	\begin{equation}\label{UlyanovTikNewFifth}
		\text{BV}(\T) \hookrightarrow \dot{B}^{\alpha_0}_{q,\infty}(\T), \qquad 1 < q < \infty, \quad 0 < \alpha_0 < \frac{1}{q},
	\end{equation}
	where $\dot{B}^{\alpha_0}_{q,\infty}(\T)$ is equipped with $\vertiii{\cdot}_{B^{\alpha_0}_{q,\infty}(\T), 1/q}$. This inequality is known because $\text{BV}(\T) \hookrightarrow \dot{B}^1_{1,\infty}(\T)$ and, in addition, it is plain to show that $\vertiii{f}_{B^{\alpha_0}_{q,\infty}(\T),k} \lesssim \vertiii{f}_{B^1_{1,\infty}(\T),k}$ for $k \in \N, \, k > \max\{\alpha_0,1\}$. However, the key novelty of (\ref{UlyanovBrezisWaingerSharp**<<<}) relies on the fact that the embedding constants of (\ref{UlyanovTikNewFifth}) can be estimated as follows
	\begin{equation*}
		t^{-\alpha_0} \omega_{1/q}(f,t)_{q;\T} \leq C (1/q - \alpha_0)^{-1/q} \|f\|_{\text{BV}(\T)},
	\end{equation*}
	where $C > 0$ is independent of $f, t \in (0,1)$ and $\alpha_0 \to 1/q-$. Furthermore, the exponent $1/q$ 
 is sharp.
	
	(iv) In analogy to Remark \ref{RemUlyanovBrezisWainger}(ii), the reason for working with the semi-norm $\vertiii{\cdot}_{B^{\alpha_0}_{q, \infty}(\T), \alpha}$ in (\ref{UlyanovBrezisWaingerSharp**<<<}) instead of $\|\cdot\|_{\dot{B}^{\alpha_0}_{q, \infty}(\T), \alpha}$ is to be able to characterize the space $\L^{(\alpha,-b)}_{q, \infty}(\T)$, which is endowed with (\ref{DefLipschitz}), in terms of extrapolation of the scale $\dot{B}^{\alpha_0}_{q, \infty}(\T)$ as $\alpha_0 \to \alpha-$.
	
\end{rem}

\begin{proof}[Proof of Theorem \ref{UlyanovBrezisWainger2}]

	(i) $\Longrightarrow$ (ii): Let $k \in \N$. According to \eqref{UlyanovClassic} and Lemma \ref{LemmaKfunctLS}(iii), one has
	\begin{equation}\label{0007}
		\|f\|_{L_q(\T) + \infty \dot{W}^k_q(\T)}\lesssim \|f\|_{\dot{B}^{1/q'}_{1,q}(\T), k}.
	\end{equation}
	where
	\begin{equation*}
		\|f\|_{L_q(\T) + \infty \dot{W}^k_q(\T)}:=\sup_{t > 0} K(t,f;L_q(\T), \dot{W}^k_q(\T)). 
	\end{equation*}
	Furthermore, a similar reasoning as in \eqref{0009} shows that $L_q(\T) + \infty \dot{W}^k_q(\T) = L_q(\T) + \infty \dot{H}^\alpha_q(\T)$ and thus, by \eqref{0007},
	\begin{equation*}
		\dot{B}^{1/q'}_{1,q}(\T) \hookrightarrow L_q(\T) + \infty \dot{H}^\alpha_q(\T).
	\end{equation*}
	This embedding together with (i) allows us to derive
	\begin{equation}\label{KFunctLip3}
		K(t,f ; L_q(\T) + \infty \dot{H}^\alpha_q(\T), \L^{(\alpha,-b)}_{q, \infty}(\T)) \lesssim K(t,f; \dot{B}^{1/q'}_{1,q}(\T) , \L^{(\alpha + 1/q',0)}_{1,\infty}(\T)).
	\end{equation}
	Since 
	\begin{equation*}
		K(t,f;L_q(\T) + \infty \dot{H}^\alpha_q(\T),  \dot{H}^\alpha_q(\T)) = K(t,f;L_q(\T), \dot{H}^\alpha_q(\T))
	\end{equation*}
	(cf. \cite[Theorem 1.5, Chapter 5, p. 297]{BennettSharpley}) we can apply (\ref{LipLimInter}) and (\ref{LemmaHolmstedt2}) to get
	\begin{align*}
		K(t (1 - \log t)^{b},f ; L_q(\T)+ \infty \dot{H}^\alpha_q(\T), \L^{(\alpha,-b)}_{q, \infty}(\T)) \\
		&\hspace{-6.5cm} \asymp K(t (1 - \log t)^{b},f;L_q(\T)+ \infty \dot{H}^\alpha_q(\T), (L_q(\T)+ \infty \dot{H}^\alpha_q(\T), \dot{H}^\alpha_q(\T))_{(1,-b),\infty}) \\
		& \hspace{-6.5cm}\asymp K(t, f; L_q(\T)+ \infty \dot{H}^\alpha_q(\T), \dot{H}^\alpha_q(\T)) \\
		& \hspace{-6cm}+ t (1-\log t)^{b} \sup_{t < u < 1} u^{-1} (1-\log u)^{-b} K(u,f;L_q(\T)+ \infty \dot{H}^\alpha_q(\T), \dot{H}^\alpha_q(\T)) \\
		& \hspace{-6.5cm} \gtrsim K(t, f; L_q(\T), \dot{H}^\alpha_q(\T))
	\end{align*}
	and thus, by Lemma \ref{LemmaKfunctLS},
	\begin{equation}\label{KFunctLip4}
		K(t^\alpha (1 - \log t)^{b},f ; L_q(\T)+ \infty \dot{H}^\alpha_q(\T), \L^{(\alpha,-b)}_{q, \infty}(\T)) \gtrsim \omega_\alpha(f,t)_{q;\T}.
	\end{equation}
	
	Next we estimate $K(t,f; \dot{B}^{1/q'}_{1,q}(\T), \L^{(\alpha + 1/q',0)}_{1,\infty}(\T))$. To this end, we will make use of the following interpolation formulas (see (\ref{BInter}) and (\ref{LipLimInter}))
	\begin{equation}\label{InterFormulas}
		\dot{B}^{1/q'}_{1,q}(\T) = (L_1(\T), \dot{H}^{\alpha + 1 -1/q}_1(\T))_{\frac{1}{1 + \alpha q'}, q}
	\end{equation}
	and
	\begin{equation}\label{InterFormulas*}
	\L^{(\alpha + 1-1/q,0)}_{1,\infty}(\T) = (L_1(\T), \dot{H}^{\alpha + 1 -1/q}_1(\T))_{(1,0),\infty}.
	\end{equation}
Therefore, by (\ref{InterFormulas}), (\ref{InterFormulas*}) and (\ref{LemmaHolmstedt3}), applying monotonicity properties of the $K$-functional and a simple change of variables, we derive 
	\begin{align*}
		K(t^{\frac{\alpha q'}{\alpha q' + 1}},f; \dot{B}^{\frac{1}{q'}}_{1,q}(\T) , \L^{(\alpha + 1-\frac{1}{q},0)}_{1,\infty}(\T) ) \\
		&\hspace{-5cm} \asymp K(t^{\frac{\alpha q'}{\alpha q' + 1}}, f; (L_1(\T) , \dot{H}^{\alpha + 1 -\frac{1}{q}}_1(\T) )_{\frac{1}{1 + \alpha q'}, q},  (L_1(\T) , \dot{H}^{\alpha + 1 -\frac{1}{q}}_1(\T) )_{(1,0),\infty}) \\
		& \hspace{-5cm} \asymp  \Big(\int_0^t (u^{- \frac{1}{1 + \alpha q'}} K(u,f;L_1(\T) , \dot{H}^{\alpha + 1 -\frac{1}{q}}_1(\T) ))^q \frac{du}{u} \Big)^{\frac{1}{q}} \\
		& \hspace{-3.5cm}+ t^{\frac{\alpha q'}{\alpha q' + 1}} \sup_{t < u < 1} u^{-1} K(u,f; L_1(\T) , \dot{H}^{\alpha + 1 -\frac{1}{q}}_1(\T) ) \\
		& \hspace{-5cm} \asymp  \Big(\int_0^t (u^{- \frac{1}{1 + \alpha q'}} K(u,f;L_1(\T) , \dot{H}^{\alpha + 1 -\frac{1}{q}}_1(\T) ))^q \frac{du}{u} \Big)^{\frac{1}{q}} \\
		& \hspace{-3.5cm}+ t^{-\frac{1}{1 +\alpha q'}}  K(t,f; L_1(\T) , \dot{H}^{\alpha + 1 -\frac{1}{q}}_1(\T) ) \\
		& \hspace{-5cm} \asymp  \Big(\int_0^{t^{(\alpha + \frac{1}{q'})^{-1}}} (u^{- \frac{1}{q'}} K(u^{\alpha + \frac{1}{q'}},f;L_1(\T) , \dot{H}^{\alpha + 1 -\frac{1}{q}}_1(\T) ))^q \frac{du}{u} \Big)^{\frac{1}{q}}.
	\end{align*}
	Applying now Lemma \ref{LemmaKfunctLS} in the previous estimate we obtain
	\begin{equation}\label{KFunctLip5}
	 K(t^{\frac{\alpha q'}{\alpha q' + 1}},f; \dot{B}^{\frac{1}{q'}}_{1,q}(\T) , \L^{(\alpha + 1-\frac{1}{q},0)}_{1,\infty}(\T))\asymp  \Big(\int_0^{t^{(\alpha + \frac{1}{q'})^{-1}}} (u^{- \frac{1}{q'}} \omega_{\alpha + 1 - \frac{1}{q}}(f,u)_{1;\T})^q \frac{du}{u} \Big)^{\frac{1}{q}}.
	\end{equation}
	Collecting (\ref{KFunctLip4}), (\ref{KFunctLip3}) and (\ref{KFunctLip5}), we arrive at
	\begin{align*}
		\omega_\alpha(f,t)_{q;\T} & \lesssim K(t^\alpha (1-\log t)^{b},f; \dot{B}^{1/q'}_{1,q}(\T) , \L^{(\alpha + 1-1/q,0)}_{1, \infty}(\T)) \\
		& \lesssim \Big(\int_0^{t (1- \log t)^{b/\alpha}} (u^{-1/q'} \omega_{\alpha + 1 -1/q}(f,u)_{1;\T})^q \frac{du}{u} \Big)^{1/q}.
	\end{align*}
	
	(ii) $\Longrightarrow$ (iii): By (\ref{UlyanovBrezisWainger2Sharp}),
	\begin{align*}
		\omega_\alpha(f,t)_{q;\T} & \lesssim \left(\int_0^{t (1-\log t)^{b/\alpha}} u^{\alpha q} \frac{du}{u} \right)^{1/q} \|f\|_{\L^{(\alpha + 1/q',0)}_{1, \infty}(\T)} \\
		&\asymp t^\alpha (1-\log t)^b  \|f\|_{\L^{(\alpha + 1/q',0)}_{1, \infty}(\T)}.
	\end{align*}
	Let $\alpha_0 \in (0,  \alpha)$. The previous estimate and (\ref{DefHolZygLog}) yield that
	\begin{align*}
		\vertiii{f}_{B^{\alpha_0}_{q, \infty}(\T), \alpha} & = \sup_{0 < t < 1} t^{-\alpha_0} \omega_\alpha(f,t)_{q;\T} \lesssim  \|f\|_{\L^{(\alpha + 1/q',0)}_{1, \infty}(\T)} \sup_{0 < t < 1} t^{\alpha - \alpha_0} (1 - \log t)^b \\
		& \asymp (\alpha - \alpha_0)^{-b} \|f\|_{\L^{(\alpha + 1/q',0)}_{1, \infty}(\T)}.
			\end{align*}
			
			(iii) $\Longrightarrow$ (i): Let $j_0 \in \N_0$ be such that $2^{-j_0} < \alpha$ and let $\alpha_j = \alpha-2^{-j}, \, j \geq j_0$. According to (\ref{UlyanovBrezisWaingerSharp**<<<}), we have
			\begin{equation}\label{UlyanovBrezisWainger2'}
				\sup_{j \geq j_0} 2^{-j b} \vertiii{f}_{B^{\alpha_j}_{q, \infty}(\T), \alpha} \lesssim \|f\|_{\L^{(\alpha + 1/q',0)}_{1,\infty}(\T)}.
			\end{equation}
			On the other hand, in view of (\ref{DefHolZygLog}) and applying Fubini's theorem, we derive
			\begin{align*}
				\sup_{j \geq j_0} 2^{-j b} \vertiii{f}_{B^{\alpha_j}_{q, \infty}(\T), \alpha} & =\sup_{j \geq j_0} 2^{-j b} \sup_{0 < t < 1} t^{-\alpha_j} \omega_\alpha(f,t)_{q;\T}\\
				& = \sup_{0 < t < 1} t^{-\alpha}  \omega_\alpha(f,t)_{q;\T}\sup_{j \geq j_0} 2^{-j b} t^{2^{-j}} \\
				& \asymp \sup_{0 < t < 1} t^{-\alpha} (1 - \log t)^{-b} \omega_\alpha(f,t)_{q;\T} = \|f\|_{\L^{(\alpha, -b)}_{q,\infty}(\T)}.
			\end{align*}
			Inserting this estimate into (\ref{UlyanovBrezisWainger2'}), we arrive at $\L^{(\alpha + 1/q',0)}_{1,\infty}(\T) \hookrightarrow \L^{(\alpha, -b)}_{q,\infty}(\T)$.
			
			Concerning the equivalence between (i) and (iv), we refer to (\ref{ThmEmbLipschitz*}).
\end{proof}

We conclude this paper with two remarks.

\begin{rem}\label{RemLast}
Inequality (\ref{UlyanovTikNewFirst}) is optimal in the following sense
	\begin{equation}\label{SharpnessAssUT*}
		\omega_\alpha (f, t)_{q;\T} \lesssim  \Big(\int_0^{t (1 - \log t)^{1/\alpha q}} (u^{-1/q'} \omega_{\alpha + 1/q'} (f, u)_{1;\T})^r \frac{du}{u} \Big)^{1/r} \iff r \leq q
	\end{equation}
	and
	\begin{equation}\label{SharpnessAssUT2*}
		\omega_\alpha (f, t)_{q;\T} \lesssim  \Big(\int_0^{t (1 - \log t)^{1/\alpha q}} (u^{-1/q'} (1-\log u)^b \omega_{\alpha + 1/q'} (f, u)_{1;\T})^q \frac{du}{u} \Big)^{1/q} \iff b \geq 0.
	\end{equation}
\end{rem}	
\begin{proof}[Proof of Remark \ref{RemLast}]
	We start by showing (\ref{SharpnessAssUT*}). Since
	\begin{equation}\label{SharpnessAssUT*new}
	 \Big\|f - \frac{1}{2 \pi} \int_0^{2 \pi} f(t) dt \Big\|_{L_q(\T)} \asymp \sup_{t > 0} \omega_\alpha (f, t)_{q;\T},
	 \end{equation}
	 see, e.g., \cite[(12)]{ivanov} and \cite[(13.12)]{DominguezTikhonov}, it follows from the inequality stated in (\ref{SharpnessAssUT*}) that $\dot{B}^{1/q'}_{1,r}(\T)$ is contained in $L_q(\T)$. According to \cite[Theorem 3.2.1]{SickelTriebel}, the previous result holds if and only if $r \leq q$.
	
	Assume that the estimate given in (\ref{SharpnessAssUT2*}) holds. Then, by (\ref{SharpnessAssUT*new}), the space $\dot{B}^{1/q', b}_{1, q} (\T)$ is contained in $L_q(\T)$. Since $L_{q}(\log L)_{b}(\T)$ is the r.i. hull of $B^{1/q', b}_{1, q} (\T)$ (see \cite{Martin}), we derive $L_{q}(\log L)_{b}(\T) \subseteq L_q(\T)$. This yields $b \geq 0$.

	\begin{rem}\label{RemLast2}
	The counterpart of \eqref{UlyanovTikNewFirst} in higher dimensions reads as follows. Assume $\alpha >0$ and $1 < q < \infty$. If $f \in\dot{B}^{d/q'}_{1,q}(\T^d)$ and $t \in (0,1)$ then
	\begin{equation}\label{RemLast2.1}
		\omega_\alpha (f, t)_{q;\T^d} \lesssim  \Big(\int_0^{t (1 - \log t)^{1/\alpha q}} (u^{-d/q'} \omega_{\alpha + d/q'} (f, u)_{1;\T^d})^q \frac{du}{u} \Big)^{1/q}.
	\end{equation}
	As a byproduct, we obtain the multivariate version of \eqref{UlyanovTikNew}, namely,
	\begin{equation}\label{RemLast2.2}
		\omega_\alpha (f, t)_{q;\T^d} \lesssim \left( \int_0^t (u^{-d/q'} (1 - \log u)^{1/q} \omega_{\alpha + d/q'} (f, u)_{1;\T^d})^{q} \frac{du}{u}\right)^{1/q},
	\end{equation}
see also \cite{KolomoitsevTikhonov}.
	Starting from the embedding \eqref{ThmEmbLipschitz*alpha}, the proof of \eqref{RemLast2.1} proceeds in the same vein as was done for the implication (i) $\Longrightarrow$ (ii) in Theorem \ref{UlyanovBrezisWainger2}. Note that this argument can also be applied to show that \eqref{RemLast2.1} (and so, \eqref{RemLast2.2}) also holds true for functions $f$ on $\R^d$.

	\end{rem}

\end{proof}

\newpage

\appendix
\section{Computation of $K$-functionals}

In this appendix we collect a number of well-known characterizations of $K$-functionals (cf. \eqref{DefKFunct}) that are relevant in the developments of this paper. Moreover, for the convenience of the reader, we include the calculations  of some $K$-functionals, which might be known but 
 we could not find
 them in the literature.

\vspace{2mm}

\textbf{$K$-functional for pairs of Lorentz spaces.} The characterization of the $K$-functional for pairs of Lorentz spaces in terms of rearrangements reads as follows. For our purposes, it is enough to consider Lorentz spaces defined on $\mathcal{X} \in \{\R^d, \T^d, \Omega\}$ equipped with the Lebesgue measure, however we note that the following result holds true for general measure spaces.

\begin{lem}[{\cite[Theorem 4.1]{Holmstedt}}]\label{LemmaHolm4}
Assume $0 < p_0 < p_1 < \infty, 0 < q_0, q_1 \leq \infty$. Let $1/\alpha = 1/p_0 - 1/p_1$. Then
\begin{equation}\label{LemmaKFunctLorentz}
	K(t, f; L_{p_0, q_0}(\mathcal{X}), L_{p_1, q_1}(\mathcal{X})) \asymp \left(\int_0^{t^\alpha} (u^{1/p_0} f^\ast(u))^{q_0} \frac{du}{u}  \right)^{1/q_0} + t \left(\int_{t^\alpha}^{|\mathcal{X}|_d} (u^{1/p_1} f^\ast(u))^{q_1} \frac{du}{u} \right)^{1/q_1}
\end{equation}
for $t >0$.

Assume $0 < p_0 < \infty$ and $0 < q_0 \leq \infty$. Then
\begin{equation}\label{LemmaKFunctLorentzLinfty}
	K(t, f; L_{p_0, q_0}(\mathcal{X}), L_\infty(\mathcal{X})) \asymp \left(\int_0^{t^{p_0}} (u^{1/p_0} f^\ast(u))^{q_0} \frac{du}{u}  \right)^{1/q_0}
\end{equation}
for $t >0$.
\end{lem}

Further characterizations of $K$-functionals for couples of Banach lattices may be found in \cite[Section 3.9.2]{Brudnyibook} and the references  therein.

\vspace{2mm}

\textbf{$K$-functionals for pairs of Lebesgue and Sobolev spaces.} Recall that $\Omega$ is a bounded Lipschitz domain in $\R^d$.

\begin{lem}\label{LemmaKfunctLS}
Let $1 \leq p \leq \infty, k \in \N$ and $s > 0$. For $t > 0$ we have
\begin{enumerate}[\upshape(i)]
	\item $K(t^k, f; L_p(\R^d), (\dot{W}^k_p(\R^d))_0) \asymp \omega_k(f,t)_{p;\R^d}$;
	\item $K(t^k, f; L_p(\R^d), W^k_p(\R^d)) \asymp \min\{1,t^k\} \|f\|_{L_p(\R^d)} +  \omega_k(f,t)_{p;\R^d}$;
	\item $K(t^s, f; L_p(\R^d), \dot{H}^s_p(\R^d)) \asymp \omega_s(f,t)_{p; \R^d}$;
	\item $K(t^k, f; L_p(\Omega), \dot{W}^k_p(\Omega)) = \inf_{g \in W^k_p(\Omega)} \{ \|f-g\|_{L_p(\Omega)} + t^k \| |\nabla^k g|\|_{L_p(\Omega)} \} \asymp \omega_k(f,t)_{p;\Omega}$.
	A density argument shows that the previous formula also holds true when the infimum is taken over the class formed by all $g_{|\Omega}$ with $g \in C^\infty_0(\R^d)$;
	\item $K(t^k, f; L_p(\Omega), W^k_p(\Omega))  \asymp \min\{1,t^k\} \|f\|_{L_p(\Omega)} +  \omega_k(f,t)_{p;\Omega}$.
	\end{enumerate}
	\end{lem}

The formulas (i) and (ii) may be found in \cite[Chapter 5, Theorem 4.12, p. 339 and formula (4.42), p. 341]{BennettSharpley}. The formula (iii) also holds true for function spaces over $\T^d$, see \cite[(1.32)]{KT}. Concerning (iv) and (v) we refer to \cite{JohnenScherer}. The previous results admit extensions to r.i. spaces, see \cite[Section 3.9.4]{Brudnyibook}.

For the purposes of this paper, we need the analogue of the formula (iv) when $\dot{W}^k_p(\Omega)$ is replaced by the smaller space $(\dot{W}^k_p(\Omega))_0$ (cf. Section \ref{SectionFunctionSpaces}). We were not able to find a suitable reference,
thus we include below a detailed proof of this result.

%

\begin{prop}\label{Prop} Let $\Omega_0$ be a compact set strictly contained in $\Omega$. Assume $1 \leq p < \infty$ and $k \leq d/p$. Let $f \in L_p(\Omega)$ be such that $\emph{supp }f \subset \Omega_0$. Then, for every $0 < t < \frac{\emph{dist} (\Omega_0, \partial \Omega)}{k}$, we have
\begin{equation}\label{Form}
	C_1 \omega_k(f,t)_{p;\Omega} \leq K(t^k, f; L_p(\Omega), (\dot{W}^k_p(\Omega))_0) \leq C_2 (\emph{dist} (\Omega_0, \partial \Omega))^{-k} \omega_k(f, t)_{p;\Omega},
\end{equation}
where $C_1$ and $C_2$ are positive constants which are independent of $f, t$ and $\Omega_0$.
\end{prop}

\begin{rem}
	The assumption $\text{supp }f \subset \Omega_0$ is necessary in the previous result, more precisely, one cannot overlook the constant $(\text{dist} (\Omega_0, \partial \Omega))^{-k}$ in the right-hand side inequality in \eqref{Form}. Indeed, if $g \in C^\infty_0(\Omega)$ and $t \in (0,1)$ then
	\begin{align*}
	t^k \|f\|_{L_p(\Omega)} \leq \|f-g\|_{L_p(\Omega)} + t^k \|g\|_{L_p(\Omega)} \lesssim  \|f-g\|_{L_p(\Omega)} + t^k \||\nabla^k g|\|_{L_p(\Omega)}
	\end{align*}
	where we have used the Poincar\'e inequality in the last step. Hence, taking the infimum over all $g \in C^\infty_0(\Omega)$, we obtain that
	\begin{equation*}
		t^k \|f\|_{L_p(\Omega)} \lesssim K(t^k, f; L_p(\Omega), (\dot{W}^k_p(\Omega))_0), \quad t \in (0,1).
	\end{equation*}
	Assuming that the right-hand side inequality in \eqref{Form} holds true uniformly in $L_p(\Omega)$, we  establish  that if $t$ is sufficiently small then
	\begin{equation}\label{Form2}
		t^k \|f\|_{L_p(\Omega)} \lesssim \omega_k(f,t)_{p;\Omega},
	\end{equation}
	which is not true even for $f \in C^\infty_0(\Omega)$. To see this is enough to consider a sequence of smooth functions $\{f_n\}$ with $f_n \to \chi_{\Omega}$ in $L_p(\Omega)$. According to \eqref{Form2}, we have $t^k \|f_n\|_{L_p(\Omega)} \lesssim \omega_k(f_n,t)_{L_p(\Omega)}$ for all $n$ and thus, taking limits as $n \to \infty$, we arrive at
	\begin{equation*}
		t^k|\Omega|_d^{1/p} \lesssim \omega_k(\chi_\Omega, t)_{p;\Omega} = 0.
	\end{equation*}
	This gives the desired contradiction.
\end{rem}

\begin{proof}[Proof of Proposition \ref{Prop}]
We will divide the proof into three steps.

\textsc{Step 1:} Let $\bar{f}$ be the extension by zero of $f$ to $\R^d$. We claim that there exists a positive constant $C_{p, d, k}$, depending only on $p, d$ and $k$, such that
\begin{align}
	K(t, \bar{f}; L_p(\mathbb{R}^d), (\dot{W}^k_p(\mathbb{R}^d))_0) & \leq  K(t, f; L_p(\Omega), (\dot{W}^k_p(\Omega))_0) \nonumber \\
	&  \leq C_{p, d, k} (\text{dist} (\Omega_0, \partial \Omega))^{-k}  K(t, \bar{f}; L_p(\mathbb{R}^d), (\dot{W}^k_p(\mathbb{R}^d))_0) \label{Claim}
\end{align}
for $t > 0$. To show this we will construct a mollifier $\varphi_\delta \in C^\infty(\R^d)$ depending on $\delta := \text{dist} (\Omega_0, \partial \Omega) > 0$ such that
\begin{itemize}
	\item $0 \leq \varphi_\delta \leq 1$,
	\item $\varphi_\delta = 1$ on $\Omega_0$,
	\item $\varphi_\delta = 0$ on $\R^d\backslash \Omega$,
	\item $\sup_{x \in \Omega} |\nabla^k \varphi_\delta (x)| \lesssim \delta^{-k}$.
\end{itemize}
This can be done in a standard way.
{Consider 
 $\psi \in C^\infty(\R^d)$ such that $0 \leq \psi \leq 1, \, \text{supp } \psi \subset B(0,1)$ and $\int_{B(0,1)} \psi = 1$. Here, $B(0,1)$ is the unit ball in $\R^d$. Define $\varphi_\delta = \chi_{(\Omega_0)_\delta} \ast \psi_\delta$ where $\psi_{\delta}(x) = \big(\frac{4}{\delta} \big)^{d} \psi \big(\frac{4 x}{\delta} \big)$ and $(\Omega_0)_\delta = \{x \in \Omega: \text{dist} (x, \Omega_0) \leq \delta/2\}$. Clearly, $(\Omega_0)_\delta$ is a closed set such that $\Omega_0 \subset (\Omega_0)_\delta \subset \Omega$. Basic computations yield that
\begin{equation}\label{Conv}
	\varphi_\delta(x) = \Big(\frac{4}{\delta} \Big)^{d} \int_{(\Omega_0)_\delta \cap B(x, \frac{\delta}{4})} \psi \Big(\frac{4 (x-y)}{\delta} \Big) \, dy.
\end{equation}
Hence, $0 \leq \varphi_\delta \leq 1$. Assume $x \in \Omega_0$. In this case we have $B(x, \frac{\delta}{4}) \subset (\Omega_0)_\delta$. Indeed, if $y \in B(x, \frac{\delta}{4})$ then $y \in \Omega$; otherwise, if $y \not \in \Omega$ then $\delta = \text{dist}(\Omega_0, \partial \Omega) \leq \text{dist} (x,\partial \Omega) = \text{dist} (x, \R^d \backslash \Omega) \leq |x-y| < \delta/4$, which is a contradiction. In addition, we have $\text{dist} (y, \Omega_0) \leq |y-x| < \delta/4$. According to \eqref{Conv} and a simple change of variables, we obtain
\begin{equation*}
	\varphi_\delta(x) =  \Big(\frac{4}{\delta} \Big)^{d} \int_{B(x, \frac{\delta}{4})} \psi \Big(\frac{4 (x-y)}{\delta} \Big) \, dy = \int_{B(0,1)} \psi(y) \, dy = 1.
\end{equation*}
Furthermore, we have
\begin{equation*}
\text{supp} (\varphi_\delta) \subset (\Omega_0)_\delta + B(0, \frac{\delta}{4}) \subset \Omega.
\end{equation*}
The first inclusion follows from basic properties of convolutions. Concerning the second inclusion, we may argue by contradiction: assume that there is $x \in (\Omega_0)_\delta$ and $|y| < \delta/4$ such that $x + y \not \in \Omega$. By assumption, there exists $x_0 \in \Omega_0$ such that $|x-x_0| \leq \delta/2$ and then
\begin{equation*}
	\delta = \text{dist} (\Omega_0, \partial \Omega) \leq \text{dist} (x_0, \partial \Omega) \leq |x_0 -(x + y)| \leq \frac{\delta}{2}  + \frac{\delta}{4} < \delta,
\end{equation*}
which is not true.

Let $\alpha = (\alpha_1, \ldots, \alpha_d) \in \mathbb{N}_0^d$ with $|\alpha| = \sum_{i=1}^d \alpha_i = k$. It follows from
\begin{equation*}
	D^\alpha \varphi_\delta (x) = \Big(\frac{4}{\delta} \Big)^{d+k}  \int_{B(x, \frac{\delta}{4})} \chi_{(\Omega_0)_\delta}(y)  D^\alpha \psi \Big(\frac{4(x-y)}{\delta} \Big) \, dy
\end{equation*}
and a change of variables that
\begin{equation}\label{partial}
	|D^\alpha \varphi_\delta (x)| \leq  \Big(\frac{4}{\delta} \Big)^{d+k}  \int_{B(x, \frac{\delta}{4})}  \Big| D^\alpha \psi \Big(\frac{4(x-y)}{\delta} \Big) \Big| \, dy = \Big(\frac{4}{\delta} \Big)^k  \|D^\alpha \psi\|_{L_1(\R^d)}.
\end{equation}}

Take any decomposition $\bar{f} = g + h$ with $g \in L_p(\mathbb{R}^d)$ and $h \in C^\infty_0(\R^d)$.  Define now $g_1 = g \varphi_\delta$ and $h_1 = h \varphi_\delta$. For $x \in \Omega$, we have
\begin{equation*}
	f(x) = f(x) \varphi_\delta(x) = \bar{f}(x) \varphi_\delta(x) = g(x) \varphi_\delta(x) + h(x) \varphi_\delta(x) = g_1(x)+h_1(x).
\end{equation*}
Clearly, $g_1 \in L_p(\Omega)$. On the other hand, using the Leibniz formula and \eqref{partial}, we derive
\begin{align}
	\| |\nabla^k h_1|\|_{L_p(\Omega)} &\leq \sum_{|\alpha| = k} \|D^\alpha h_1\|_{L_p(\Omega)} \leq \sum_{|\alpha| = k} \sum_{\beta \leq \alpha} {{\alpha} \choose {\beta}} \|D^\beta \varphi_\delta D^{\alpha-\beta}  h \|_{L_p(\Omega)} \nonumber \\
	& \leq \sum_{|\alpha| = k} \sum_{\beta \leq \alpha} {{\alpha} \choose {\beta}} \Big(\frac{4}{\delta} \Big)^{|\beta|}  \|D^\beta \psi\|_{L_1(\R^d)} \|D^{\alpha-\beta}  h \|_{L_p(\Omega)}. \label{Der}
\end{align}
Assume first $k < d/p$. Thus, by the Sobolev inequality, we have
\begin{equation}\label{SobHom}
	\||\nabla^l h|\|_{L_{q}(\R^d)} \lesssim \||\nabla^k h|\|_{L_p(\R^d)}, \quad 0 \leq l < k, \quad \frac{1}{q}  =  \frac{l-k}{d} + \frac{1}{p},
\end{equation}
and, in particular,
\begin{equation*}
	\||\nabla^l h|\|_{L_p(\Omega)} \leq \||\nabla^l h|\|_{L_q(\Omega)}  \lesssim \||\nabla^k h|\|_{L_p(\R^d)}.
\end{equation*}
Assume now $k=d/p$. In this case, the inequality \eqref{SobHom} with $l=0$ (and so, $q=\infty$) fails to be true. In fact, one cannot replace $L_\infty(\R^d)$ in \eqref{SobHom} by a r.i. space on $\R^d$, but  one can only expect local type inequalities. Indeed, we can use the well-known fact that $(\dot{W}^k_p(\R^d))_0$ is continuously embedded into the spaces $\text{BMO}(\R^d)$ formed by all functions of bounded mean oscillation and thus, by the John-Nirenberg inequality, the space $(\dot{W}^k_p(\R^d))_0$ is locally embedded into $L_p(\Omega)$. Therefore
\begin{equation}\label{SobHom2}
	\||\nabla^l h|\|_{L_p(\Omega)} \lesssim \||\nabla^k h|\|_{L_p(\R^d)}, \quad l = 0, \ldots, k-1, \quad k = \frac{d}{p}.
\end{equation}
Note that standard scaling arguments show that the previous inequality cannot hold if $k > d/p$.

Inserting the estimates \eqref{SobHom} and \eqref{SobHom2} into \eqref{Der}, we obtain
\begin{equation*}
	\| |\nabla^k h_1|\|_{L_p(\Omega)} \lesssim \delta^{-k} \||\nabla^k h|\|_{L_p(\R^d)}, \quad k \leq \frac{d}{p}.
\end{equation*}
Thus $|\nabla^k h_1| \in L_p(\Omega)$ and, in addition, $h_1(x)= h(x) \varphi_\delta(x) = 0$  for all $x \in \partial \Omega$. Consequently,
\begin{align*}
	K(t, f; L_p(\Omega), (\dot{W}^k_p(\Omega))_0) &\leq \|g_1\|_{L_p(\Omega)} + t \||\nabla^k h_1|\|_{L_p(\Omega)} \\
	& \leq C_{k, \psi, d, p} \delta^{-k} [  \|g\|_{L_p(\mathbb{R}^d)} + t \||\nabla^k h|\|_{L_p(\mathbb{R}^d)} ]
\end{align*}
and taking the infimum over all possible decompositions of $\bar{f}$ we arrive at
\begin{equation*}
	K(t, f; L_p(\Omega), (\dot{W}^k_p(\Omega))_0)  \leq C_{k, \psi, d, p} \delta^{-k}  K(t, \bar{f}; L_p(\mathbb{R}^d), (\dot{W}^k_p(\mathbb{R}^d))_0).
\end{equation*}
Next we proceed with the converse estimate. Take any decomposition $f = g + h$ with $g \in L_p(\Omega)$ and $h \in C^\infty_0(\Omega)$. Then $\bar{f} = \bar{g}+\bar{h}$, where $\| \bar{g}\|_{L_p(\R^d)} = \|g\|_{L_p(\Omega)}$ and $\|\nabla^k \bar{h}\|_{L_p(\R^d)} = \|\nabla^k h\|_{L_p(\Omega)}$. Hence,
\begin{equation*}
	K(t, \bar{f}; L_p(\mathbb{R}^d), (\dot{W}^k_p(\mathbb{R}^d))_0) \leq \| \bar{g}\|_{L_p(\R^d)} + t \|\nabla^k \bar{h}\|_{L_p(\R^d)} =  \|g\|_{L_p(\Omega)} + t  \|\nabla^k h\|_{L_p(\Omega)},
\end{equation*}
which yields that $K(t, \bar{f}; L_p(\mathbb{R}^d), (\dot{W}^k_p(\mathbb{R}^d))_0) \leq K(t, f; L_p(\Omega), (\dot{W}^k_p(\Omega))_0)$. This completes the proof of \eqref{Claim}.

\textsc{Step 2:}
We will show that
\begin{equation}\label{Claim2}
 \omega_k(\bar{f},t)_{p;\R^d} = \omega_k(f, t)_{p;\Omega}, \quad 0 < t < \frac{\delta}{k}.
 \end{equation}
For $|h| \leq t$ recall that $\Omega_{k h} = \{x : x + \eta k h \in \Omega, \text{ for all } 0 \leq \eta \leq 1\}$ (see Section \ref{SectionFunctionSpaces}). We have
 \begin{align*}
 	\||\Delta^k_h \bar{f}|\|_{L_p(\R^d)}^p &= \int_{\Omega_{k h}} |\Delta^k_h f(x)|^p \,dx +   \int_{\R^d \backslash \Omega_{k h}} |\Delta^k_h \bar{f}(x)|^p \,dx.
 \end{align*}
 Thus the proof of \eqref{Claim2} will be finished if we show that
 \begin{equation}\label{Zero}
 	 \int_{\R^d \backslash \Omega_{k h}} |\Delta^k_h \bar{f}(x)|^p \,dx = 0, \quad |h| \leq t.
 \end{equation}
 To prove this, it will  be convenient to introduce the notation $(\tilde{\Omega}_0)_{k h} = \{x \in \R^d :  x + j h \in \Omega_0 \text{ for some } j \in \{0, \ldots, k\}\} $. Since $\Delta^k_h \bar{f} (x) = 0$ for $x \not \in (\tilde{\Omega}_0)_{k h}$,  we have
 \begin{equation*}
 	 \int_{\R^d \backslash \Omega_{k h}} |\Delta^k_h \bar{f}(x)|^p \,dx  =  \int_{(\R^d \backslash \Omega_{k h}) \cap (\tilde{\Omega}_0)_{k h}} |\Delta^k_h \bar{f}(x)|^p \,dx.
 \end{equation*}
 Then to verify \eqref{Zero} it is sufficient to show that the set $(\R^d \backslash \Omega_{k h}) \cap (\tilde{\Omega}_0)_{k h}$ is empty. Assume that there exists $x \in \R^d$ such that $x \not \in \Omega_{k h}$ and $x \in  (\tilde{\Omega}_0)_{k h}$. This means that there are $j \in \{0, \ldots, k\}$ and $\eta \in [0,1]$ such that $x + j h \in \Omega_0$ and $x + \eta k h \not \in \Omega$. Then
 \begin{equation*}
 	\delta = \text{dist}(\Omega_0, \partial \Omega) \leq |j - \eta k| |h| \leq  |j - \eta k| t \leq k t < \delta,
 \end{equation*}
 which is not possible. This shows that $(\R^d \backslash \Omega_{k h}) \cap (\tilde{\Omega}_0)_{k h} = \emptyset$.

\textsc{Step 3:}
By Lemma \ref{LemmaKfunctLS}(i) and \eqref{Claim}, we have
\begin{equation*}
	 \omega_k(\bar{f},t)_{p;\R^d} \leq C_{p,d,k}  K(t, f; L_p(\Omega), (\dot{W}^k_p(\Omega))_0)   \leq C_{p, d, k} (\text{dist} (\Omega_0, \partial \Omega))^{-k}   \omega_k(\bar{f},t)_{p;\R^d}
\end{equation*}
for $t > 0$. Finally, \eqref{Claim2} completes the proof.
\end{proof}

\vspace{2mm}

\textbf{$K$-functional for Sobolev spaces.} Let $\mathcal{X} \in \{\R^d, \Omega\}$. The $K$-functional for the pair $(W^k_1(\mathcal{X} ), W^k_\infty(\mathcal{X} )), \, k \in \N,$ was obtained by DeVore and Scherer \cite{DeVoreScherer}. Namely, they showed that
	\begin{equation}
	K(t, f; W^k_1(\mathcal{X} ), W^k_\infty(\mathcal{X} )) \asymp \sum_{|\alpha| \leq k} \int_0^t (D^\alpha f)^*(u) \, du \asymp \int_0^t \Big[ f^*(u) + \sum_{|\alpha| = k}  (D^\alpha f)^*(u) \Big] \, du  \label{DS}
	\end{equation}
	for $t > 0$ and $f \in W^k_1(\mathcal{X} ) + W^k_\infty(\mathcal{X} )$. Another approaches to this formula may be found in \cite{CalderonMilman} and \cite{DeVoreSharpleyb}.
	
	A well-known argument based on \eqref{DS}, the reiteration property of interpolation spaces \cite[Chapter 5, Theorem 2.4, p. 311]{BennettSharpley} and Lemma \ref{LemmaHolm4} yield the following

\begin{lem}
\begin{enumerate}[\upshape(i)] Let $k \in \N$. Assume that one of the conditions
\begin{equation}\label{Cond}
		 \left\{\begin{array}{cll}  p_0 = q_0 = 1, & 1 < p_1 < \infty, & 0 < q_1 \leq \infty, \\
		 & & \\
		 1 < p_0 < p_1 < \infty, & 0 < q_0 , q_1 \leq \infty, \\
		 & & \\
		 1 < p_0 < \infty, & 0 < q_0 \leq \infty, & p_1=q_1=\infty, \\
		 & & \\
		 p_0 = q_0=1, & p_1=q_1=\infty,
		       \end{array}
                        \right.
	\end{equation}
	is satisfied. If $f \in W^k L_{p_0,q_0}(\mathcal{X}) + W^k L_{p_1,q_1}(\mathcal{X})$ and $t > 0$ then
	\begin{equation}
	K(t, f; W^k L_{p_0,q_0}(\mathcal{X}), W^k L_{p_1,q_1}(\mathcal{X}))   \asymp \sum_{|\alpha| \leq k} K(t, D^\alpha f; L_{p_0,q_0}(\mathcal{X}), L_{p_1,q_1}(\mathcal{X})). \label{DS2}
	\end{equation}

\end{enumerate}
\end{lem}

	A more general formulation of the previous result may be found in \cite[(6.10)]{Nilsson}.

Working with homogeneous Sobolev spaces on $\R^d$, it is known that
	\begin{equation}\label{DS2-}
	K(t, f; \dot{W}^1_1(\R^d), \dot{W}^1_\infty(\R^d)) \asymp \sum_{|\alpha| = 1} \int_0^t (D^\alpha f)^*(u) \, du
	\end{equation}
	for $f \in W^1_1(\R^d) + W^1_\infty(\R^d)$ and uniformly in $t > 0$; see \cite{MartinMilman06}. The following result provides the counterparts of \eqref{DS2}-\eqref{DS2-}  for the pair $((V^k_1(\R^d))_0, (V^k_\infty(\R^d))_0)$.

\begin{prop}\label{PropDS}
	Let $k \in \N$. If $f \in (V^k_1(\R^d))_0 + (V^k_\infty(\R^d))_0$ and $t > 0$, then
		\begin{equation}\label{DS2*}
	K(t, f; (V^k_1(\R^d))_0, (V^k_\infty(\R^d))_0) \asymp \sum_{|\alpha| = k} \int_0^t (D^\alpha f)^*(u) \, du.
	\end{equation}
	Consequently, if any of the conditions given in \eqref{Cond} holds, then
	\begin{equation}\label{LemmaKFunctLorentzSobolev}
		K(t, f; (V^k L_{p_0,q_0}(\R^d))_0, (V^k L_{p_1,q_1}(\R^d))_0) \asymp K(t, |\nabla^k f|; L_{p_0,q_0}(\R^d), L_{p_1,q_1}(\R^d))
	\end{equation}
	for $f \in (V^k L_{p_0,q_0}(\R^d))_0+ (V^k L_{p_1,q_1}(\R^d))_0$ and $t > 0$.
\end{prop}

The proof of \eqref{DS2*} follows essentially the methodology given in \cite[Theorem 8.4]{DeVoreSharpleyb} to show \eqref{DS}. For completeness, we provide some details below.

\begin{proof}[Proof of Proposition \ref{PropDS}]
 Let $f \in (V^k_1(\R^d))_0 + (V^k_\infty(\R^d))_0$. For $t > 0$, we set
\begin{equation*}
E_t = \{x \in \R^d : f^{b}_k (x) > (f^b_k)^*(t)\} \cup \{x \in \R^d : M f(x) > (M f)^*(t)\}, \quad F_t = \R^d \backslash E_t,
\end{equation*}
where $M$ is the Hardy--Littlewood maximal operator and $f^b_k$ is the sharp maximal operator (see \cite[(1.6)]{DeVoreSharpleyb}). The set $F_t$ is closed and $|E_t|_d \leq 2 t$. 

Let $g_t$ be the extension of $f$ from $F_t$ to $\R^d$ given by  \cite[(8.4)]{DeVoreSharpleyb}. 
We first show that, for each $\lambda > 0$ and $m \in \{ 0, \ldots, k-1\}$, $g_t$ satisfies
the condition 
\begin{equation}\label{C}
	|\{x \in \R^d : |\nabla^m g_t (x)| > \lambda\}|_d < \infty.
\end{equation}
Indeed, we have
\begin{align*}
	\{x \in \R^d : |\nabla^m g_t (x)| > \lambda\} & = \{x \in E_t : |\nabla^m g_t (x)| > \lambda\} \cup \{x \in F_t :  |\nabla^m g_t (x)| > \lambda\} \\
	& \subset E_t \cup  \{x \in F_t :  |\nabla^m f (x)| > \lambda\} \\
	& \subset E_t \cup  \{x \in \R^d :  |\nabla^m f (x)| > \lambda\}
\end{align*}
and thus
\begin{align*}
	|\{x \in \R^d : |\nabla^m g_t (x)| > \lambda\}|_d &\leq |E_t|_d + | \{x \in \R^d :  |\nabla^m f (x)| > \lambda\} |_d \\
	& \leq 2 t +  | \{x \in \R^d :  |\nabla^m f (x)| > \lambda\} |_d < \infty
\end{align*}
because $f \in (V^k_1(\R^d))_0 + (V^k_\infty(\R^d))_0$. 

 According to \cite[Lemma 8.1]{DeVoreSharpleyb}, there holds 
\begin{equation}\label{CZ}
	|g_t(x)| \leq c (M f)^*(t) \quad \text{and} \quad (g_t)^b_k(x) \leq c (f^b_k)^*(t)
\end{equation}
for all $x \in \R^d$. Here $c$ is a positive constant, which is independent of $f, x$ and $t$. We consider the decomposition $f= h_t + g_t$, where $h_t = f-g_t$ and we will show that
\begin{equation}\label{proof1}
	t \||\nabla^k g_t|\|_{L_\infty(\R^d)} \lesssim   \int_0^t |\nabla^k f|^*(u) \, du
\end{equation}
and
\begin{equation}\label{proof2}
	\||\nabla^k h_t|\|_{L_1(\R^d)} \lesssim \int_0^t |\nabla^k f|^*(u) \, du.
\end{equation}
Assuming that these estimates hold, we obtain
\begin{equation*}
 f = h_t + g_t \in (V^k_1(\R^d))_0 + (V^k_\infty(\R^d))_0
 \end{equation*}
  and
\begin{equation*}
	K(t, f; (V^k_1(\R^d))_0, (V^k_\infty(\R^d))_0) \leq \||\nabla^k h_t|\|_{L_1(\R^d)} +  t \||\nabla^k g_t|\|_{L_\infty(\R^d)} \lesssim \int_0^t |\nabla^k f|^*(u) \, du.
\end{equation*}
We start by proving \eqref{proof1}. Using the estimates $f^b_k(x) \lesssim M (|\nabla^k f|)(x)$ and $|\nabla^k g_t (x)| \lesssim (g_t)^b_k(x)$ (see \cite[Theorem 5.6]{DeVoreSharpleyb}) together with the second inequality given in \eqref{CZ}, we infer that
\begin{equation*}
	\||\nabla^k g_t|\|_{L_\infty(\R^d)}  \lesssim \|(g_t)^b_k\|_{L_\infty(\R^d)} \lesssim (f^b_k)^*(t) \lesssim (M(|\nabla^k f|))^*(t) \asymp |\nabla^k f|^{**}(t)
\end{equation*}
where we have used the Herz-Stein inequality (see \cite[Chapter 3, Theorem 3.8, p. 122]{BennettSharpley}) in the last step.

As far as \eqref{proof2}, since $h_t=0$ on $F_t$, we have
\begin{align*}
	\||\nabla^k h_t|\|_{L_1(\R^d)} &= \int_{E_t} |\nabla^k h_t (x) | \, dx \leq \int_{E_t} |\nabla^k f(x)| \, dx + \int_{E_t} |\nabla^k g_t(x)| \, dx \\
	& \leq \int_0^{|E_t|} |\nabla^k f|^*(u) \, du + \||\nabla^k g_t|\|_{L_\infty(\R^d)} |E_t|_d \\
	& \lesssim \int_0^t  |\nabla^k f|^*(u) \, du + t \||\nabla^k g_t|\|_{L_\infty(\R^d)}  \\
	& \lesssim \int_0^t  |\nabla^k f|^*(u) \, du
\end{align*}
where the last estimate follows from \eqref{proof1}.

On the other hand, the estimate
\begin{equation*}
	 \int_0^t |\nabla^k f|^*(u) \, du \lesssim K(t, f; (V^k_1(\R^d))_0, (V^k_\infty(\R^d))_0)
\end{equation*}
follows easily from the fact that $f^*(t) \leq f^{**}(t)$ and the subadditivity of the map $f \to f^{**}$.

\end{proof}

\end{document}